\documentclass[10pt]{article}

\input diagram.tex

\usepackage{amsmath} 
\usepackage[mathscr]{eucal}
\usepackage{mathrsfs}
\usepackage{amssymb} 
\usepackage{theorem}
\usepackage{rotating}
\usepackage{stmaryrd}

\theoremstyle{change} 
\newtheorem{theorem}{Theorem}[section] 
\newtheorem{lemma}[theorem]{Lemma} 
\newtheorem{proposition}[theorem]{Proposition}
\newtheorem{corollary}[theorem]{Corollary}
\theorembodyfont{\rmfamily} 
\newtheorem{remark}[theorem]{Remark}

\newtheorem{definition}[theorem]{Definition}
\newtheorem{notation}[theorem]{Notation}
\newtheorem{nothing}[theorem]{} 

\newenvironment{proof}{\noindent{\bf Proof}\ }{\qed\bigskip}

\renewcommand{\le}{\leqslant}
\renewcommand{\ge}{\geqslant} 

\renewcommand{\unlhd}{\trianglelefteqslant}

\newcommand{\alphabar}{\overline{\alpha}}
\newcommand{\alphatilde}{\tilde{\alpha}}
\newcommand{\Bbar}{\overline{B}}

\newcommand{\betatilde}{\tilde{\beta}}
\newcommand{\BIGOP}[1]
  {\mathop{\mathchoice
  {\raise-0.22em\hbox{\huge $#1$}}
  {\raise-0.05em\hbox{\Large $#1$}}{\hbox{\large $#1$}}{#1}}}

\newcommand{\br}{\mathrm{br}}
\newcommand{\Br}{\mathrm{Br}}
\newcommand{\calA}{\mathcal{A}}
\newcommand{\calB}{\mathcal{B}}
\newcommand{\calBP}{\mathcal{BP}}
\newcommand{\calBPtilde}{\widetilde{\calBP}}
\newcommand{\calF}{\mathcal{F}}
\newcommand{\calI}{\mathcal{I}}
\newcommand{\calJ}{\mathcal{J}}
\newcommand{\calO}{\mathcal{O}}
\newcommand{\catfont}{\mathsf}

\newcommand{\cdotG}{\mathop{\cdot}\limits_{G}}
\newcommand{\cdotH}{\mathop{\cdot}\limits_{H}}
\newcommand{\cdotHi}{\mathop{\cdot}\limits_{H_i}}
\newcommand{\Def}{\mathrm{Def}}
\newcommand{\defl}{\mathrm{def}}
\newcommand{\dott}[3]{\mathop{\cdot}\limits^{{#1},{#2}}_{#3}}
\newcommand{\dottt}[1]{\mathop{\cdot}\limits_{#1}}
\newcommand{\dotXYH}{\dott{X}{Y}{H}}
\newcommand{\Dtilde}{\tilde{D}}
\newcommand{\ebar}{\overline{e}}
\newcommand{\etabar}{\overline{\eta}}
\newcommand{\Etilde}{\tilde{E}}
\newcommand{\etilde}{\tilde{e}}
\newcommand{\End}{\mathrm{End}}
\newcommand{\fbar}{\overline{f}}
\newcommand{\ftilde}{\tilde{f}}
\newcommand{\gammabar}{\overline{\gamma}}
\newcommand{\Gammahat}{\widehat{\Gamma}}
\newcommand{\gammatilde}{\tilde{\gamma}}
\newcommand{\Gammatilde}{\widetilde{\Gamma}}
\newcommand{\Gbar}{\overline{G}}
\newcommand{\gH}{\lexp{g}{H}}
\newcommand{\gM}{\lexp{g}{M}}
\newcommand{\gP}{\lexp{g}{P}}

\newcommand{\Hom}{\mathrm{Hom}}
\newcommand{\hQ}{\lexp{h}{Q}}
\newcommand{\Ibar}{\overline{I}}
\newcommand{\id}{\mathrm{id}}
\newcommand{\Ind}{\mathrm{Ind}}
\newcommand{\ind}{\mathrm{ind}}
\newcommand{\Inf}{\mathrm{Inf}}
\newcommand{\infl}{\mathrm{inf}}
\newcommand{\Irr}{\mathrm{Irr}}
\newcommand{\Jbar}{\overline{J}}
\newcommand{\kk}{\mathbb{\Bbbk}}
\newcommand{\KK}{\mathbb{K}}
\newcommand{\Lambdahat}{\widehat{\Lambda}}
\newcommand{\Lambdatilde}{\widetilde{\Lambda}}
\newcommand{\Lbar}{\overline{L}}
\newcommand{\lexp}[2]{\setbox0=\hbox{$#2$} \setbox1=\vbox to
                 \ht0{}\,\box1^{#1}\!#2}
\newcommand{\LHom}{\mathrm{LHom}}
\newcommand{\lmod}[1]{\llap{\phantom{|}}_{#1}\catfont{mod}}
\newcommand{\lMod}[1]{\llap{\phantom{|}}_{#1}\catfont{Mod}}
\newcommand{\lmapsto}{\leftarrow\!\mapstochar}
\newcommand{\ltriv}[1]{\llap{\phantom{|}}_{#1}\catfont{triv}}
\newcommand{\Mbar}{\overline{M}}
\newcommand{\Mtilde}{\tilde{M}}
\newcommand{\myiso}{\buildrel\sim\over\to}
\newcommand{\Nbar}{\overline{N}}
\newcommand{\Omegatilde}{\widetilde{\Omega}}
\newcommand{\Omegahat}{\widehat{\Omega}}
\newcommand{\omegabar}{\overline{\omega}}
\newcommand{\pbar}{\overline{p}}
\newcommand{\phitilde}{\tilde{\phi}}

\newcommand{\qed}{\nobreak\hfill
                   \vbox{\hrule\hbox{\vrule\hbox to 5pt
                   {\vbox to 8pt{\vfil}\hfil}\vrule}\hrule}}
\newcommand{\QQ}{\mathbb{Q}}
\newcommand{\res}{\mathrm{res}}
\newcommand{\Res}{\mathrm{Res}}
\newcommand{\RHom}{\mathrm{RHom}}
\newcommand{\rk}{\mathrm{rk}}
\newcommand{\stab}{\mathrm{stab}}
\newcommand{\tens}[3]{\mathop{\otimes}\limits^{{#1},{#2}}_{#3}}
\newcommand{\tenskXYH}{\tens{X}{Y}{\kk H}}
\newcommand{\tensOXYH}{\tens{X}{Y}{\calO H}}
\newcommand{\tr}{\mathrm{tr}}
\newcommand{\Vbar}{\overline{V}}
\newcommand{\Wbar}{\overline{W}}
\newcommand{\Wtilde}{\tilde{W}}
\newcommand{\Xtilde}{\tilde{X}}
\newcommand{\Ytilde}{\tilde{Y}}
\newcommand{\zetabar}{\overline{\zeta}}
\newcommand{\ZZ}{\mathbb{Z}}

\title{$p$-Permutation Equivalences between Blocks of Group Algebras\footnote{{\bf MR Subject Classification:}  
20C20, 19A22{\bf Keywords:}  $p$-permutation modules, trivial source modules, blocks of group algebras, fusion systems, perfect isometries, isotypies, splendid Rickard complexes.}}

\author{\small Robert Boltje\\
  \small Department of Mathematics\\
  \small University of California\\
  \small Santa Cruz, CA 95064\\
  \small U.S.A.\\
  \small boltje@ucsc.edu
  \and
  \small Philipp Perepelitsky\\
  \small Department of Mathematics\\ 
  \small University of California\\
  \small Santa Cruz, CA 95064\\
  \small U.S.A.\\
  \small pperepel@ucsc.edu}
\date{July 17, 2020}

\begin{document}
\sloppy


\maketitle


\begin{center}
\parbox{12cm}{
\begin{center} {\bf Contents} \end{center}
{\it 
1.\ Introduction\quad 2.\ Subgroups of direct product groups\quad 3.\ $p$-permutation modules\quad 4.\ Block theoretic preliminaries\quad 5.\ Brauer pairs for $p$-permutation modules\quad 6.\ Extended tensor products and homomorphisms\quad
7.\ Tensor products of $p$-permutation bimodules\quad 8.\ Character groups and perfect isometries\quad 9.\ Grothendieck groups of $p$-permutation modules and $p$-permutation equivalences\quad 10.\ Brauer pairs of $p$-permutation equivalences\quad 11.\ Fusion systems, local equivalences and finiteness\quad 12.\ A character theoretic criterion and moving from left to right\quad 13.\ K\"ulshammer-Puig classes\quad 14.\ The maximal module of a $p$-permutation equivalence\quad 15.\ Connection with isotypies and splendid Rickard equivalences
}}
\end{center}


\begin{abstract}
We extend the notion of a {$p$-permutation equivalence} between two $p$-blocks $A$ and $B$ of finite groups $G$ and $H$, from the definition in \cite{BoltjeXu2008} to a virtual $p$-permutation bimodule whose components have twisted diagonal vertices. It is shown that various invariants of $A$ and $B$ are preserved, including defect groups, fusion systems, and K\"ulshammer-Puig classes. Moreover it is shown that $p$-permutation equivalences have additional surprising properties. They have only one constituent with maximal vertex and the set of  $p$-permutation equivalences between $A$ and $B$ is finite (possibly empty). The paper uses new methods: a consequent use of module structures on subgroups of $G\times H$ arising from Brauer constructions which in general are not direct product subgroups, the necessary adaptation of the notion of tensor products between bimodules, and a general formula (stated in these new terms) for the Brauer construction of a tensor product of $p$-permutation bimodules.
\end{abstract}


\section{Introduction}\label{sec intro}
Let $G$ and $H$ be finite groups, $\calO$ a complete discrete valuation ring of characteristic $0$ which contains a root of unity whose order is equal to the exponent of $G\times H$. We denote by $\KK$ the field of fractions of $\calO$, and by $F$ its residue field whose characteristic we assume to be a prime $p$. Furthermore we assume that $A$ is a block algebra of $\calO G$ and $B$ is a block algebra of $\calO H$. Various authors have defined notions of equivalence between $A$ and $B$ (e.g.~\cite{Broue1990}, \cite{Broue1995}, \cite{Rickard1996}, \cite{Puig1999}, \cite{BoltjeXu2008}, \cite{Linckelmann2009}). They are divided into two parts: those that are equivalences of categories (as for instance Morita equivalences, derived equivalences, splendid Rickard equivalences), and those that are isomorphisms between associated representation rings (as for instance perfect isometries and isotypies), preserving additional features on the representation ring level. The first attempt to define a strongest possible equivalence on a representation ring level goes back to \cite{BoltjeXu2008}, where a preliminary notion of a {\em $p$-permutation equivalence} was defined. \cite{BoltjeXu2008} made the restrictive assumptions that the blocks $A$ and $B$ have a common defect group $D$, that certain fusion categories are equivalent, and some results on this notion were only proved under the hypothesis that $D$ is abelian. This notion was soon after extended in \cite{Linckelmann2009} to source algebras, see also \cite[Section 9.5]{Linckelmann2018}.

In this paper we cast the net much wider and define a {\rm $p$-permutation equivalence} as an element $\gamma\in T^\Delta(A,B)$, the representation group of finitely generated $p$-permutation $(A,B)$-bimodules whose indecomposable direct summands, when regarded as left $\calO[G\times H]$-modules, have {\em twisted diagonal vertices}, i.e., vertices of the form $\Delta(P,\phi,Q):=\{(\phi(y),y)\mid y\in Q\}$, the graph of an isomorphism $\phi\colon Q\myiso P$ between a $p$-subgroup $Q$ of $H$ and a $p$-subgroup $P$ of $G$, with the property that
\begin{equation}\label{eqn ppeq conditions}
  \gamma\cdotH\gamma^\circ = [A] \in T^\Delta(A,A)\quad \text{and} \quad 
  \gamma^\circ \cdotG \gamma = [B]\in T^\Delta(B,B)\,,
\end{equation}
where $\gamma^\circ\in T^\Delta(B,A)$ is the $\calO$-dual of $\gamma$ and $\cdotH$ is induced by the tensor product over $\calO H$, or equivalently, over $B$.
We show that a $p$-permutation equivalence forces the defect groups and fusion systems of $A$ and $B$ to be isomorphic. In fact, via the notion of {$\gamma$-Brauer pairs}, it selects an isomorphism between defect groups and fusion systems, unique up to $G\times H$-conjugation in a precise sense, see Theorem~\ref{thm A}. Moreover, we show that a splendid Rickard equivalence between $A$ and $B$ induces a $p$-permutation equivalence, and that a $p$-permutation equivalence between $A$ and $B$ induces an isotypy, see Theorem~\ref{thm F}. The goal of this paper is twofold: On the one hand, we want to understand what $p$-permutation equivalences can look like by finding necessary conditions on such an element $\gamma\in T^\Delta(A,B)$. On the other hand, we want to study which invariants of $A$ and $B$ are preserved under a $p$-permutation equivalence.

It turns out that the language of Brauer pairs (cf.~\ref{noth blocks}(b) for a definition) is crucial for the study of $p$-permutation equivalences. We denote by $-^*$ the antipode $x\mapsto x^{-1}$ of any group algebra of a group $X$. Note that any $(A,B)$-bimodule belongs to the block algebra $A\otimes_\calO B^*$ of $\calO G\otimes \calO H$, when viewed as $\calO[G\times H]\cong \calO G\otimes_\calO \calO H$-module. A {\em $\gamma$-Brauer pair} is an $(A\otimes_\calO B^*)$-Brauer pair $(X, e\otimes f^*)$, where $X$ is a $p$-subgroup of $G\times H$, $e$ is a block idempotent of $\calO C_G(p_1(X))$ and $f$ is a block idempotent of $\calO C_H(p_2(X))$, where $p_1\colon G\times H\to G$ and $p_2\colon G\times H\to H$ denote the canonical projections, which satisfies $\gamma(X,e\otimes f^*)\neq 0\in T(FN_{G\times H}(X,e\otimes f^*))$. Here, $N_{G\times H}(X,e\otimes f^*)$ denotes the $G\times H$-stabilizer of the Brauer pair $(X,e\otimes f^*)$, and the expression $\gamma(X,e\otimes f^*)$, is defined by applying the Brauer construction with respect to $X$ (cf.~\ref{noth Brauer con}(b)) to $\gamma$ and then cutting with the idempotent $e\otimes f^*$. Note that $X$ is necessarily a twisted diagonal subgroup $\Delta(P,\phi,Q)$ and that $C_{G\times H}(X)=C_G(P)\times C_H(Q)$, so that $e$ is a block idempotent of $\calO C_G(P)$ and $f$ is a block idempotent of $\calO C_H(Q)$. The following theorem shows that even though $A$ and $B$ are no longer required to have a {\em common} defect group, the element $\gamma$ selects through the choice of a maximal $\gamma$-Brauer pair an isomorphism $\phi\colon E\myiso D$ between defect groups $D$ and $E$ of $A$ and $B$, respectively. Recall that $(A\otimes_\calO B^*)$-Brauer pairs form a $G\times H$-poset.

\begin{theorem}\label{thm A}
Assume that $\gamma\in T^\Delta(A,B)$ is a $p$-permutation equivalence between $A$ and $B$.

\smallskip
{\rm (a)} The set of $\gamma$-Brauer pairs is closed under $G\times H$-conjugation and under taking smaller Brauer pairs. Moreover, the maximal $\gamma$-Brauer pairs form a single $G\times H$-conjugacy class.

\smallskip
{\rm (b)} Let $(\Delta(D,\phi,E), e\otimes f^*)$ be a maximal $\gamma$-Brauer pair. Then $D$ is a defect group of $A$, $E$ is a defect group of $B$, $(D,e)$ is a maximal $A$-Brauer pair, $(E,f)$ is a maximal $B$-Brauer pair, and the isomorphism $\phi\colon E\myiso D$ is an isomorphism between the fusion systems of $B$ and $A$ associated to $(E,f)$ and $(D,e)$, respectively.
\end{theorem}

The above theorem follows from the more precise Theorems~\ref{thm gamma is uniform} and \ref{thm isomorphic fusion systems}.

\smallskip
The following theorem states additional restrictive properties of $p$-permutation equivalences.

\begin{theorem}\label{thm B}
Suppose that $\gamma\in T^\Delta(A,B)$ is a $p$-permutation equivalence between $A$ and $B$ and let $(\Delta(D,\phi,E), e\otimes f^*)$ be a maximal $\gamma$-Brauer pair.

\smallskip
{\rm (a)} Every indecomposable $(A,B)$-bimodule appearing in $\gamma$ has a vertex contained in $\Delta(D,\phi,E)$. 

\smallskip
{\rm (b)} Up to isomorphism, there exists a unique indecomposable $(A,B)$-bimodule $M$ appearing in $\gamma$ with vertex $\Delta(D,\phi,E)$. Its coefficient in $\gamma$ is $1$ or $-1$.
\end{theorem}

The module $M$ in Theorem~\ref{thm B} is called the {\em maximal module} of $\gamma$. Theorem~\ref{thm B} follows from the stronger statements in Theorem~\ref{thm maximal module} and \ref{thm Brauer pairs of appearing modules}.

The following theorem shows that Brauer constructions of $p$-permutation equivalences lead again to $p$-permutation equivalences or even Morita equivalences. It follows from the more precise statements in Theorem~\ref{thm local ppeqs}, Theorem~\ref{thm Morita between Brauer correspondents}, and Proposition~\ref{prop Morita on local level}.

\begin{theorem}\label{thm C}
Let $\gamma\in T^\Delta(A,B)$ be a $p$-permutation equivalence and let $(\Delta(P,\phi,Q),e\otimes f^*)$ be a $\gamma$-Brauer pair. Set $I:=N_G(P,e)$ and $J:=N_H(Q,f)$, let $Y:=N_{G\times H}(\Delta(P,\phi,Q), e\otimes f^*)=N_{I\times J}(\Delta(P,\phi,Q))$, and let $\gamma'\in T(\calO Y)$ denote the unique lift of $\gamma(\Delta(P,\phi,Q), e\otimes f^*)\in T(FY)$. 

\smallskip
{\rm (a)} $\res^Y_{C_G(P)\times C_H(Q)}(\gamma')\in T^\Delta(\calO C_G(P)e,\calO C_H(Q)f)$ is a $p$-permutation equivalence between $\calO C_G(P)e$ and $\calO C_H(Q)f$.

\smallskip
{\rm (b)} $\ind_Y^{I\times J}(\gamma')\in T^\Delta(\calO Ie,\calO Jf)$ is a $p$-permutation equivalence between $\calO Ie$ and $\calO Jf$.

\smallskip
{\rm (c)} If $(\Delta(P,\phi,Q),e\otimes f^*)$ is a maximal $\gamma$-Brauer pair then $\gamma'=\pm[M']$ for an indecomposable $p$-permutation $\calO Y$-module $M'$ which arises from the maximal module $M$ of $\gamma$ by $M'=M(\Delta(P,\phi,Q),e\otimes f^*)$. Moreover, $\Res^Y_{C_G(P)\times C_H(Q)}(M')$ induces a Morita equivalence between $\calO C_G(P)e$ and $\calO C_H(Q)f$, and $\Ind_Y^{I\times J}(M')$ induces a Morita equivalence between $\calO Ie$ and $\calO Jf$.

\smallskip
{\rm (d)} Suppose that $(\Delta(P,\phi,Q),e\otimes f^*)$ is a maximal $\gamma$-Brauer pair, set $\Ytilde:=N_{G\times H}(\Delta(P,\phi,Q))$ and let $\Mtilde\in\lmod{\calO Y}$ be the Green correspondent of $M$, then the $p$-permutation $(\calO N_G(P),\calO N_H(Q))$-bimodule $\Ind_{Y}^{N_G(P)\times N_H(Q)}(\Mtilde)$ induces a Morita equivalence between the Brauer correspondents of $A$ and $B$.
\end{theorem}

The following interesting additional properties of $p$-permutation equivalences follow from the more precise Theorems~\ref{thm finiteness} and \ref{thm left equals right}.

\begin{theorem}\label{thm D}
{\rm (a)} The number of $p$-permutation equivalences between $A$ and $B$ is finite (possibly zero).

\smallskip
{\rm (b)} If $\gamma\in T^\Delta(A,B)$ satisfies one of the two equations in (\ref{eqn ppeq conditions}) then it also satisfies the other.
\end{theorem}

Another invariant of a block algebra is given by the collection of K\"ulshammer-Puig classes (see \ref{noth KP classes}), one for every centric subgroup of a defect group in the associated fusion system. Since the fusion systems of $A$ and $B$ are isomorphic, centric subgroups correspond. This gives a way to compare K\"ulshammer-Puig classes of $A$ and $B$. The following theorem follows from the more precise Theorem~\ref{thm K-P classes}. 

\begin{theorem}\label{thm E}
Suppose that $\gamma\in T^\Delta(A,B)$ is a $p$-permutation equivalence between $A$ and $B$ and let $(\Delta(P,\phi,Q), e\otimes f^*)$ be a $\gamma$-Brauer pair such that $Z(P)$ is a defect group of the block algebra $\calO C_G(P)e$. Set $I:=N_G(P,e)$, $J:=N_H(Q,f)$, $\Ibar:=I/PC_G(P)$ and $\Jbar:=J/QC_H(Q)$, and let $\kappa\in H^2(\Ibar,F^\times)$ and $\lambda\in H^2(\Jbar,F^\times)$ be the corresponding K\"ulshammer-Puig classes of $(P,e)$ and $(Q,f)$. Then the isomorphism between $\Ibar$ and $\Jbar$ induced by $N_{G\times H}(\Delta(P,\phi,Q))$ (see Proposition~\ref{prop isomorphic inertia quotients}) makes $\kappa$ correspond to $\lambda$.
\end{theorem}

The following theorem is proved in Section~\ref{sec isotypies}.

\begin{theorem}\label{thm F}
{\rm (a)} Suppose that the chain complex $C_\bullet$ is a splendid Rickard equivalence between $A$ and $B$ (see Definition~\ref{def splendid Rickard equivalence}). Then the element $\gamma:=\sum_{n\in \ZZ} (-1)^n[C_n]\in T^\Delta(A,B)$ is a $p$-permutation equivalence between $A$ and $B$.

\smallskip
{\rm (b)} Suppose that $\gamma\in T^\Delta(A,B)$ is a $p$-permutation equivalence between $A$ and $B$ and let $(\Delta(D,\phi,E), e\otimes f^*)$ be a maximal $\gamma$-Brauer pair. Then the Brauer constructions with respect to subgroups of $\Delta(D,\phi,E)$, yield an isotypy between $A$ and $B$.
\end{theorem}

In \cite{Puig1999}, Puig proved that some of the invariants (defect groups, fusion systems) of blocks considered here are preserved by splendid Rickard equivalences. Therefore, in view of Theorem~\ref{thm F}(a), our results provide a significant improvement. We also use different techniques. That K\"ulshammer-Puig classes, cf.~Theorem~\ref{thm E} are preserved was not even known under the stronger hypothesis of a splendid Rickard equivalence.

\smallskip
One main point of view and crucial tool in this paper is that the Brauer construction of an $(A,B)$-bimodule with respect to a twisted diagonal subgroup $\Delta(P,\phi,Q)$ of $G\times H$ yields a module for the normalizer $Y$ of $\Delta(P,\phi,Q)$. Rather than working with the restriction of this Brauer construction to $C_G(P)\times C_H(Q)$, we consistently work with  the resulting $\calO Y$-module. This requires to lift the construction of tensor products of bimodules to a generalized tensor product functor
\begin{equation}\label{eqn gen tens}
  -\tens{X}{Y}{H}-\colon\lmod{\kk X}\times \lmod{\kk Y}\to\lmod{\kk[X*Y]}
\end{equation}
for an arbitrary commutative ring $\kk$ and subgroups $X\le G\times H$ and $Y\le H\times K$, where $X*Y\le G\times K$ is the composition of $X$ and $Y$, viewed as correspondences between $G$ and $H$, and $H$ and $K$, respectively . This type of generalized tensor product was first used by Bouc in \cite{Bouc2010b}. In Section~\ref{sec tensor products} we develop the necessary properties of this generalized tensor product. Another main ingredient of our approach is a formula for the Brauer construction of the tensor product $M\otimes_{\calO H} N$ of two $p$-permutation bimodules $M\in\lmod{\calO G}_{\calO H}$ and $N\in\lmod{\calO H}_{\calO K}$ in terms of a direct sum of tensor products of Brauer constructions of  $X$ and of $Y$. This formula goes back to earlier work in \cite{BoltjeDanz2012} and is refined to a block-wise version in Section~\ref{sec p-perm bimodules}. 

\smallskip
The paper is arranged as follows. Section~\ref{sec groups} recalls facts about subgroups $X$ of direct product groups $G\times H$ and the composition $X*Y\le G\times K$ if $Y\le H\times K$. In Section~\ref{sec modules} we recall the necessary preliminaries on $p$-permutation modules and add some technical lemmas that are used later. Preliminaries on blocks and Brauer pairs, together with some additional results on $p$-permutation modules in a block are given in Section~\ref{sec blocks}. 
In Section~\ref{sec Brauer pairs for M} we introduce Brauer pairs for $p$-permutation modules, generalizing the concept of Brauer pairs of group algebras and blocks. We show that these Brauer pairs have very similar properties as the ones for blocks, cf.~Proposition~\ref{prop Brauer pairs for M}.
In Section~\ref{sec tensor products} we introduce the generalized tensor product (\ref{eqn gen tens}) and prove basic properties of this construction. The main result in Section~\ref{sec p-perm bimodules} is a formula (see Theorem~\ref{thm BP decomp}) for the Brauer construction of the tensor product of two $p$-permutation bimodules. This formula incorporates the generalized tensor product and blocks. In Section~\ref{sec character groups} we recall Brou\'e's notion of {\em perfect isometry} and prove several related results, while Section~\ref{sec p-permutation equivalences} introduces various representation groups associated to $p$-permutation modules, their relations with other representation groups, the notion of a Brauer pair for an element in the representation group of $p$-permutation modules, and the notion of a $p$-permutation equivalence. In Section~\ref{sec Brauer pairs of ppeqs} we study Brauer pairs of $p$-permutation equivalences and prove several surprising results; surprising, because one would not expect them to hold for virtual modules, but only for actual modules (cf.~Proposition~\ref{prop equiv gamma Brpair cond} and Theorem~\ref{thm gamma is uniform}).
Section~\ref{sec fusion systems} establishes that a $p$-permutation equivalence $\gamma$ induces an isomorphism between the fusion systems of the underlying blocks (Theorem~\ref{thm isomorphic fusion systems}) and $p$-permutation equivalences on local levels through the Brauer construction with respect to a $\gamma$-Brauer pair (see Theorem~\ref{thm local ppeqs}). This section also contains character theoretic results that are interesting in their own right (see  
Proposition~\ref{prop diagonal extension}) and imply that irreducible characters that correspond via local equivalences have the same extension properties with respect to the inertia groups of their corresponding Brauer pairs. These properties lead to the finiteness of $p$-permutation equivalences between two given blocks.
In Section~\ref{sec char criterion} we prove a character theoretic criterion for an element in $T^\Delta(A,B)$ to be a $p$-permutation equivalence which leads to the equivalence of the two conditions in (\ref{eqn ppeq conditions}). That a $p$-permutation equivalence preserves the K\"ulshammer-Puig classes is proved in Section~\ref{sec KP classes preserved}. Section~\ref{sec maximal module} establishes that every $p$-permutation equivalence has a {\em maximal module} and that the Green correspondent of the maximal module induces Morita equivalences between associated blocks on the local levels associated with the defect groups. Finally, in Section~\ref{sec isotypies} we show that $p$-permutation equivalences are logically nested between splendid Rickard equivalences and isotypies.

\begin{notation}
Throughout this paper we will use the following notation:

For a group $G$ and $g\in G$ we write $c_g\colon G\to G$ or just $\lexp{g}{-}$ for the conjugation map $x\mapsto gxg^{-1}$. For subgroups $K$ and $H$ of $G$, we write $K\le_G H$ to denote that $K$ is $G$-conjugate to a subgroup of $H$.

For a ring $R$ we denote by $Z(R)$, $R^\times$, $J(R)$ its center, its unit group, and its Jacobson radical, respectively. Unadorned tensor products are taken over the ground ring that should be apparent from the context. For $R$-modules $M$ and $N$, we write $M\mid N$ to indicate that $M$ is isomorphic to a direct summand of $N$.

For a left $G$-set $X$, $H\le G$, and $x\in X$, we write $X^H$ for the set of $H$-fixed points of $X$, $[H\backslash X]$ for a set of representatives of the $H$-orbits of $X$, and $\stab_H(x)$ for the stabilizer of $x$ in $H$.

Throughout this paper, $(\KK,\calO,F)$ denotes a $p$-modular system and $\pi$ denotes a prime element of $\calO$. We say that $(\KK,\calO,F)$ is {\em large enough} for a finite group $G$ if $\calO$ contains a root of unity of order $|G|$. We write $\bar{\cdot}\colon \calO \to F$ and $\bar{\cdot}\colon \calO G\to FG$ for the natural epimorphisms. If $M$ is an $FG$-module, we will view it without further explanation also as $\calO G$-module via restriction along $\calO G\to FG$. Thus, expressions as $am\in M$ and $ab\in FG$ are defined for $a\in \calO G$, $m\in M$ and $b\in FG$.
\end{notation}

\section{Subgroups of direct product groups}\label{sec groups}

This section recalls basic facts and constructions related to subgroups of direct product groups. For more details the reader is referred to \cite{Bouc2010a}

\begin{nothing} \label{noth dir prod}
Let $G$, $H$, and $K$ be finite groups and let $X\le G\times H$ and $Y\le H\times K$ be subgroups.

\smallskip
(a) We denote by $p_1\colon G\times H\to G$ and $p_2\colon G\times H\to H$ the canonical projections. Setting
\begin{equation*}
  k_1(X):=\{g\in G\mid (g,1)\in X\}\quad\text{and}\quad k_2(X):=\{h\in H\mid (1,h)\in X\}
\end{equation*}
one obtains normal subgroups $k_i(X)$ of $p_i(X)$ and canonical isomorphisms $X/(k_1(X)\times k_2(X))\to p_i(X)/k_i(X)$, for $i=1,2$, induced by the projection maps $p_i$. The resulting isomorphism $\eta_X\colon p_2(X)/k_2(X)\myiso p_1(X)/k_1(X)$ satisfies $\eta_X(h k_2(X))=g k_1(X)$ if and only if $(g,h)\in X$. Here, $(g,h)\in p_1(X)\times p_2(X)$.

\smallskip
(b) If $\phi\colon Q\myiso P$ is an isomorphism between subgroups $Q\le H$ and $P\le G$ then
\begin{equation*}
  \Delta(P,\phi,Q):=\{(\phi(y),y)\mid y\in Q\}
\end{equation*}
is a subgroup of $G\times H$. Subgroups arising this way will be called {\em twisted diagonal} subgroups of $G\times H$. For $P\le G$ we also set $\Delta(P):=\Delta(P,\id_P,P)\le G\times G$. Note that a subgroup $X\le G\times H$ is twisted diagonal if and only if $k_1(X)= \{1\}$ and $k_2(X)=\{1\}$. Note also that for $(g,h)\in G\times H$ one has $\lexp{(g,h)}{\Delta(P,\phi,Q)}=\Delta(\lexp{g}{P},c_g\phi c_h^{-1},\lexp{h}{Q})$.

\smallskip
(c) The subgroup $X^\circ:=\{(h,g)\in H\times G\mid (g,h)\in X\}$ of $H\times G$ is called the {\em opposite} subgroup of $X$. Clearly, one has $(X^\circ)^\circ=X$.

\smallskip
(d) The {\em composition} of $X$ and $Y$ is defined as
\begin{equation*}
  X*Y:=\{(g,k)\in G\times K\mid \exists h\in H\colon (g,h)\in X, (h,k)\in Y\}\,.
\end{equation*}
It is a subgroup of $G\times K$. Composition is associative. If $\Delta(P,\phi,Q)\le G\times H$ and $\Delta(Q,\psi,R)\le H\times K$ are twisted diagonal subgroups then $\Delta(P,\phi,Q)*\Delta(Q,\psi,R)=\Delta(P,\phi\psi,R)\le G\times K$. Note that $(X*Y)^\circ=Y^\circ * X^\circ$, for arbitrary $X\le G\times H$ and $Y\le H\times K$.
\end{nothing}

The following lemma follows immediately from the definitions.

\begin{lemma}\label{lem comp form}
Let $G$, $H$ and $K$ be finite groups and let $X\le G\times H$ and $Y\le H\times K$ be subgroups.

\smallskip
{\rm (a)} One has $X*X^\circ=\Delta(p_1(X))\cdot(k_1(X)\times\{1\}) = \Delta(p_1(X))\cdot(\{1\}\times k_2(X)) = \Delta(p_1(X))\cdot(k_1(X)\times k_2(X))$.

\smallskip
{\rm (b)} One has $X*X^\circ *X = X$.

\smallskip
{\rm (c)} If $p_1(X)\le P\le N_G(k_1(X))$ then $\bigl(\Delta(P)\cdot(k_1(X)\times\{1\})\bigr) * X = X$.

\smallskip
{\rm (d)} For any $g\in G$, $h\in H$, and $k\in K$, one has $\lexp{(g,h)}{X}*\lexp{(h,k)}{Y} = \lexp{(g,k)}{(X*Y)}$.
\end{lemma}

\begin{notation}\label{not Nphi}
Let $G$ and $H$ be finite groups, let $\Delta(P,\phi,Q)$ be a twisted diagonal subgroup of $G\times H$, and let $S\le N_G(P)$ and $T\le N_H(Q)$. We denote by $N_{(S,\phi,T)}$ the subgroup of $S$ consisting of all elements $g\in S$ such that there exists an element $h\in T$ satisfying $c_g\phi c_h=\phi$ as functions from $Q$ to $P$. Note that if $S$ contains $C_G(P)$ then also $N_{(S,\phi,T)}$ contains $C_G(P)$. Moreover, if $P\le S$ and $Q\le T$ then $P\le N_{(S,\phi,T)}$. We further set $N_\phi:=N_{(N_G(P),\phi,N_H(Q))}$. Note that this definition of $N_\phi$ corresponds to $N_{\phi^{-1}}$ in the literature on fusion systems, see for instance \cite{AKO2011}.
\end{notation}

The following proposition follows again immediately from the definitions and the conjugation formula in \ref{noth dir prod}(b).

\begin{proposition}\label{prop Nphi}
Let $G$ and $H$ be finite groups, let $\Delta(P,\phi,Q)$ be a twisted diagonal subgroup of $G\times H$, and let $C_G(P)\le S\le N_G(P)$ and $C_H(Q)\le T\le N_H(Q)$ be intermediate subgroups.

\smallskip
{\rm (a)} One has $N_{G\times G}(\Delta(P)) = \Delta(N_G(P))\cdot(C_G(P)\times\{1\}) = \Delta(N_G(P))\cdot(\{1\}\times C_G(P))$.

\smallskip
{\rm (b)} For $X:=N_{G\times H}(\Delta(P,\phi,Q))$ one has $k_1(X)=C_G(P)$, $k_2(X)= C_H(Q)$, $p_1(X)=N_\phi$, $p_2(X)=N_{\phi^{-1}}$.

\smallskip
{\rm (c)} For $X:=N_{S\times T}(\Delta(P,\phi,Q))$ one has $k_1(X)=C_G(P)$, $k_2(X)=C_H(Q)$, $p_1(X)= N_{(S,\phi,T)}$, $p_2(X)=N_{(T,\phi^{-1},S)}$.
\end{proposition}

\section{$p$-permutation modules}\label{sec modules}

In this section we recall module theoretic preliminaries and prove some results on $p$-permutation modules that will be needed later. For standard concepts of modular representation theory the reader is referred to \cite{NagaoTsushima1989}. 

\smallskip
We first recall concepts for modules over group rings $\kk G$, where $\kk$ is an arbitrary commutative ring.

\begin{nothing}\label{noth bimodules} 
Let $\kk$ be a commutative ring and let $G$ and $H$ be finite groups.

\smallskip
(a) We always assume that, for a $(\kk G,\kk H)$-bimodule $M$, the induced left and right $\kk$-module structures coincide: $\alpha m =m\alpha$ for $m\in M$ and $\alpha\in \kk$. One obtains an isomorphism between categories $\lmod{kG}_{\kk H}\cong\lmod{\kk[G\times H]}$ via the formula $(g,h) m = gmh^{-1}$ for $m\in M$ and $(g,h)\in G\times H$. We try to be consistent to translate a left $G\times H$-action into a $(G,H)$-biaction, i.e., we take the second component of the direct product to the right hand side. Similarly, we may view left $\kk G$-modules as right $\kk G$-modules via the anti-involution $-^*\colon kG\to kG$, $g\mapsto g^{-1}$. 

\smallskip
(b) For a $(\kk G,\kk H)$-bimodule $M$, we usually view its {\em $\kk$-dual} $M^\circ:=\Hom_{\kk}(M,\kk)$ as a $(\kk H,\kk G)$-bimodule via $(hfg)(m):= f(gmh)$, for $f\in M^\circ$, $m\in M$, $g\in G$ and $h\in H$. Similarly, if $M$ is a left (resp.~right) $\kk G$-module, we consider $M^\circ$ as right (resp.~left) $\kk G$-module. However, we sometimes switch sides using the identification in (a). More generally, if $X\le G\times H$ and $M$ is a left $\kk X$-module, we can view $M^\circ$ as left $\kk X^\circ$-module via $((h,g)f)(m)=f((g^{-1},h^{-1})m)$ for $f\in M^\circ$,  $m\in M$ and $(g,h)\in X$.

\smallskip
(c) The {\em trivial $\kk G$-module} has underlying $\kk$-module $\kk$ and satisfies $g\alpha= \alpha$ for all $g\in G$ and $\alpha\in \kk$. We denote it by $\kk_G$.

\smallskip
(d) If $H\le G$, $g\in G$, and $M$ is a $\kk H$-module, then we denote by $\gM$ the left $\kk[\lexp{g}{H}]$-module with underlying $\kk$-module $M$ and $\gH$-action given by restricting the $H$-action along the isomorphism $c_g^{-1}\colon \gH\to H$. 

\smallskip
(e) For subgroups $Q\le P\le G$ and a $\kk G$-module $M$ we denote by $M^P:=\{m\in M\mid xm=m \text{ for all } x\in P\}$ the set of $P$-fixed points and by $\tr_Q^P\colon M^Q\to M^P$ the {\em relative trace map} defined by $m\mapsto\sum_{x\in [P/Q]}xm$.

\smallskip
(f) Let $H$ be a subgroup of $G$, let $e\in \kk H\cap Z(\kk G)$ be an idempotent, and let $N$ be a $\kk H$-module. Then $\Ind_H^G(eN)\cong e\Ind_H^G(N)$ as $\kk G$-modules. In fact, this follows from the obvious $\kk G$-module isomorphism $\kk G\otimes_{\kk H} eN \cong e(\kk G\otimes_{\kk H} N)$.

\smallskip
(g) Recall that a {\em permutation $\kk G$-module} $M$ is a module that is isomorphic to $\kk X$ for some finite left $G$-set $X$. Equivalently, $M$ has a $\kk$-basis that is permuted by $G$.
\end{nothing}

Recall that $(\KK,\calO,F)$ denotes a $p$-modular system.

\begin{nothing}\label{noth Brauer con} Let $G$ be a finite group. The following constructions and statements will be used extensively throughout the paper.

\smallskip
(a) One has a functor $\overline{?}\colon\lmod{\calO G}\to \lmod{FG}$, given on objects by $M\mapsto \Mbar:=FG\otimes_{\calO G}M\cong F\otimes_{\calO} M\cong M/\pi M$.

\smallskip
(b) For any $p$-subgroup $P\le G$, one has a functor $-(P)\colon \lmod{\calO G}\to \lmod{F[N_G(P)/P]}$ given on objects by
\begin{equation*}
  M\mapsto M(P):= M^P/\bigl(\pi M^P + \sum_{Q<P}\tr_Q^P(M^Q)\bigr)\,.
\end{equation*}
The module $M(P)$ is called the {\em Brauer construction} of $M$ at $P$. We often view $M(P)$ as $F[N_G(P)]$-module via inflation without notational indication.  Similarly, one defines the Brauer construction $M(P)$ of an $FG$-module $M$ by $M(P)=M^P/\sum_{Q<P}\tr_Q^P(M^Q)$. The canonical map $M^P\to M(P)$ is denoted by $\Br_P^M$, or just $\Br_P$, and is called the {\em Brauer map}. Note that, for any intermediate subgroup $P\le H\le N_G(P)$ and any $\calO G$-module or $FG$-module $M$, one has $\Res^{N_G(P)}_H(M(P))=\bigl(\Res^G_{H}(M)\bigr)(P)$ as $FH$-modules.

\smallskip
(c) For a finite $G$-set $X$, the Brauer construction of the permutation $\calO G$-module $\calO X$ can be described as follows: The composition of the canonical maps $\calO [X^P]\to (\calO X)^P\to (\calO X)(P)$ induces an isomorphism of $F[N_G(P)/P]$-modules $F[X^P]\myiso (\calO X)(P)$, see \cite[(1.1)(3)]{Broue1985}. In the special case where $M=\calO G$ and $G$ acts by conjugation, the map $\Br_P^{\calO G}$ translates under the above canonical isomorphism to the projection map $\br_P=\br^G_P\colon  (\calO G)^P\to FC_G(P)$, $\sum_{g\in G}\alpha_g g\mapsto \sum_{g\in C_G(P)} \alphabar_g g$. This map is an $\calO$-algebra homomorphism, called the {\em Brauer homomorphism}. Note that $\br_P(Z(\calO G))\subseteq Z(FC_G(P))^{N_G(P)}=Z(FN_G(P))\cap FC_G(P)$.

\smallskip
(d) Let $H$ be a subgroup of $G$, let $M$ be an indecomposable $\calO G$-module (resp.~$FG$-module) with vertex $P$ and let $N$ be an indecomposable $\calO H$-module (resp.~$FH$-module) with vertex $Q$. If $M\mid \Ind_H^G(N)$ then $P\le_G Q$. If $N\mid \Res^G_H(M)$ then $Q\le_G P$. 
\end{nothing}

Recall that an $\calO G$-module (resp.~$FG$-module) $M$ is called a {\em $p$-permutation module} if $\Res^G_P(M)$ is a permutation module for each $p$-subgroup $P$ of $G$. An $(\calO G, \calO H)$-bimodule (resp.~$(FG,FH)$-bimodule) $M$ is called a {\em $p$-permutation bimodule} if it is a $p$-permutation module when considered as $\calO[G\times H]$-module (resp.~$F[G\times H]$-module). If $M$ is indecomposable its {\em vertices} form a conjugacy class of $p$-subgroups of $G\times H$.

In the next part of this section we recall some basic properties of $p$-permutation modules that will be used later. For more details on $p$-permutation modules we refer the reader to \cite{Broue1985} and to \cite[Sections 5.10, 5.11]{Linckelmann2018}.

\begin{proposition}\label{prop p-perm equiv} 
Let $G$ be a finite group.

\smallskip
{\rm (a)} Let $M$ be an $\calO G$-module or an $FG$-module. Then the following are equivalent:

\quad{\rm (i)} $M$ is a $p$-permutation module.

\quad{\rm (ii)} $M$ is isomorphic to a direct summand of a permutation module.

\quad{\rm (iii)} Each indecomposable direct summand of $M$ has the trivial module as source.

\smallskip
{\rm (b)} The functor $M\mapsto \Mbar$ induces a vertex-preserving bijection between the set of isomorphism classes of indecomposable $p$-permutation $\calO G$-modules and the set of isomorphism classes of indecomposable $p$-permutation $FG$-modules.

\smallskip
{\rm (c)} For each $p$-subgroup $P$ of $G$, the Brauer construction $M\mapsto M(P)$ induces a bijection between the set of isomorphism classes of indecomposable $p$-permutation $\calO G$-modules (resp.~$FG$-modules) with vertex $P$  and the set of isomorphism classes of indecomposable projective $F[N_G(P)/P]$-modules. Moreover, if $M$ is an indecomposable $p$-permutation $FG$-module with vertex $P$ then the $F[N_G(P)]$-module $M(P)$ is the Green correspondent of $M$.
\end{proposition}

\begin{proof} 
See \cite[(0.4), (3.2), (3.4), (3.5)]{Broue1985}.
\end{proof}

\begin{remark}\label{rem p-perm equiv}
In view of (iii) in Part~(a) of the previous proposition, $p$-permutation modules are often called {\em trivial source modules} in the literature. The reformulations of Part~(a) allow to see quickly that the class of trivial source modules is 
closed under the usual constructions of restriction, induction, inflation, Brauer construction, $\otimes_\calO$, $\oplus$, taking direct summands, and taking duals. Note that projective modules are $p$-permutation modules. Note also that Part~(b) and the Krull-Schmidt theorem imply that the functor $\overline{?}\colon \lmod{\calO G}\to\lmod{FG}$ induces a bijection between the set of isomorphism classes of $p$-permutation $\calO G$-modules and the set of isomorphism classes of $p$-permutation $FG$-modules. We will denote the category of $p$-permutation $\calO G$-modules by $\ltriv{\calO G}$. Similarly we define $\ltriv{FG}$, $\ltriv{\calO G}_{\calO H}$, etc.
\end{remark}

\begin{proposition}\label{prop more on p-perm}
Let $G$ be a finite group, let $M,N\in\ltriv{\calO G}$, and let $P$ be a $p$-subgroup of $G$.

\smallskip
{\rm (a)} One has $M(P)^\circ\cong (M^\circ)(P)$ as $F[N_G(P)]$-modules.

\smallskip
{\rm (b)} One has canonical isomorphisms $M(P)\cong\Mbar(P)$ and  $(M\otimes_\calO N)(P)\cong M(P)\otimes_F N(P)$ of $F[N_G(P)]$-modules. Moreover, if $Q$ is a $p$-subgroup of $N_G(P)$ one has a canonical isomorphism $(M(P))(Q)\cong M(PQ)$ of $F[N_G(P)\cap N_G(Q)]$-modules.

\smallskip
{\rm (c)} The canonical map $\Hom_{\calO G}(M,N)\to \Hom_{FG}(\Mbar,\Nbar)$, $f\mapsto \fbar$, induces an $F$-linear isomorphism $F\otimes_{\calO}\Hom_{\calO G}(M,N)\myiso \Hom_{FG}(\Mbar,\Nbar)$. In particular one has 
\begin{equation*}
  \dim_\KK \Hom_{\KK G}(\KK\otimes_{\calO} M, \KK\otimes_{\calO} N) = \rk_{\calO}\Hom_{\calO G}(M,N) 
  = \dim_F\Hom_{FG}(\Mbar,\Nbar)\,.
\end{equation*}
\end{proposition}

\begin{proof}
(a) See the proof of \cite[(2.4)(2)]{Broue1985}.

\smallskip
(b) In all three cases one obtains natural homomorphisms which are functorial in $M$ (resp.~$M$ and $N$). Moreover, if $M$ (resp.~$M$ and $N$) is a permutation module then \ref{noth Brauer con}(c) implies that these maps are isomorphisms. Thus, they are also isomorphisms for direct summands of permutation modules. 

\smallskip
(c) This follows again immediately by reduction to the case of permutation modules.
\end{proof}

The statements of the following lemma are folklore. We provide quick proofs for the reader's convenience and note that all statements also hold if $M$ is a $p$-permutation $FG$-module, mutatis mutandis.

\begin{lemma}\label{lem Brauer con and vertex} 
Let $G$ be a finite group, let $P$ be a $p$-subgroup of $G$, and let $M\in\ltriv{\calO G}$. 

\smallskip
{\rm (a)} Assume that $M$ is indecomposable and that $Q$ is a vertex of $M$. Then
\begin{equation*}
  M(P)\neq\{0\} \iff \calO_P\mid \Res^G_P(M) \iff P\le_G Q\,.
\end{equation*}

\smallskip
{\rm (b)} Assume that $M$ is indecomposable, that $Q$ is a vertex of $M$, and that $P$ is normal in $G$. Then one has $P\le Q$ if and only if $P$ acts trivially on $M$. 

\smallskip
{\rm (c)} Assume that $M$ is indecomposable and that $P$ is normal in $G$. Then one of the two following must hold:

\quad {\rm (i)} $P$ acts trivially on $M$ and $M(P)=\Mbar$, or

\quad {\rm (ii)} $\calO_P\nmid \Res^G_P(M)$ and $M(P)=\{0\}$.

\smallskip
{\rm (d)} There exists a decomposition $\Res^G_{N_G(P)}(M) = L\oplus N$ into $\calO[N_G(P)]$-submodules with the following property: $P$ acts trivially on $L$ and $\calO_P\nmid \Res^{N_G(P)}_P(N)$. Moreover, $M(P)\cong \Lbar$.
\end{lemma}

\begin{proof}
(a) Let $X$ be a $P$-set such that $\Res^G_P(M) \cong \calO X$ as $\calO P$-modules. The first and second statement are both equivalent to $X^P\neq\emptyset$. Since the trivial $\calO P$-module has vertex $P$, \ref{noth Brauer con}(d) shows that the second statement implies the third. Conversely, since $M$ has trivial source, one has $\calO_Q\mid \Res^G_Q(M)$, and the third statement implies the second.

\smallskip
(b) If $P$ acts trivially on $M$ then $\calO_P\mid \Res^G_P(M)$ and Part~(a) implies $P\le Q$. Conversely, since $M\mid \Ind_Q^G(F_Q)$, the Mackey formula and $P\le Q$ imply that $P$ acts trivially on $M$.

\smallskip
(c) Let $Q$ be a vertex of $M$. If $P\le Q$ then (i) holds by Part~(b). If $Q$ does not contain $P$ then (ii) holds by Part~(a).

\smallskip
(d) This follows by applying Part~(c) and the last statement in \ref{noth Brauer con}(b) to each indecomposable direct summand of $\Res^G_{N_G(P)}(M)$.
\end{proof}

The following Lemma is well-known and an easy exercise.

\begin{lemma}\label{lem Brauer map compatibility}
Let $G$ be a finite group, $M\in\ltriv{FG}$, $P\le G$ a $p$-subgroup, and let $i\in (FG)^H$ be an idempotent, fixed under the conjugation action of a subgroup $H$ of $G$ containing $P$. Then, for each $m\in M^P$, one has $\Br_P^M(im)=\br_P^G(i)\Br_P^M(m)$. In particular, one obtains a canonical isomorphism $(iM)(P)\myiso\br_P^G(i)M(P)$ of $F[N_H(P)]$-modules.
\end{lemma}

The following lemma will be used repeatedly.

\begin{lemma}\label{lem vertex of M(P)}
{\rm (a)} Let $G$ be a finite group, let $M\in\ltriv{\calO G}$ be indecomposable, and let $P$ be a $p$-subgroup of $G$. Then each vertex of each indecomposable direct summand of the $F[N_G(P)]$-module $M(P)$ is contained in a vertex of $M$. (Note that $M(P)=\{0\}$ is possible, in which case the statement is vacuously true.)

\smallskip
{\rm (b)} Let $G$ and $H$ be finite groups and let $M\in\ltriv{\calO G}_{\calO H}$ be indecomposable with twisted diagonal vertices. Let $X$ be any $p$-subgroup of $G\times H$. Then each vertex of each indecomposable direct summand of the $F[N_{G\times H}(X)]$-module $M(X)$ is again twisted diagonal. 
\end{lemma}

\begin{proof}
Part~(b) is an immediate consequence of Part~(a). In order to prove Part~(a), note that $M(P)=\Mbar(P)=\bigl(\Res^G_{N_G(P)}(\Mbar)\bigr) (P)$ and, by Lemma~\ref{lem Brauer con and vertex}(d), the latter $F[N_G(P)]$-module is isomorphic to a direct summand (possibly $\{0\}$) of $\Res^G_{N_G(P)}(\Mbar)$. Thus, Proposition~\ref{prop p-perm equiv}(b) and \ref{noth Brauer con}(d)  imply the result.
\end{proof}

%

The following lemma will be used in the proof of Proposition~\ref{prop local maximal Brauer pair}.

\begin{lemma}\label{lem Brauer con and Ind 2}
Let $G$ be a finite group, let $Q$ be a $p$-subgroup of $G$, let $M\in\ltriv{F[N_G(Q)]}$ be indecomposable, and let $R$ be a vertex of $M$. If $P$ is a $p$-subgroup of $N_G(Q)$ satisfying $N_G(P)\le N_G(Q)$ and $P\not\le \lexp{g}{R}$, for each $g\in G\smallsetminus N_G(Q)$, then $(\Ind_{N_G(Q)}^G(M))(P)\cong M(P)$ as $F[N_G(P)]$-modules.
\end{lemma}

\begin{proof}
By \ref{noth Brauer con}(b), it suffices to show that $\bigl(\Res^G_{N_G(P)}(\Ind_{N_G(Q)}^G(M))\bigr)(P)\cong \bigl(\Res^{N_G(Q)}_{N_G(P)}(M)\bigr)(P)$ as $F[N_G(P)]$-modules. The Mackey formula yields $\Res^G_{N_G(P)}(\Ind_{N_G(Q)}^G(M))\cong \bigoplus_{g\in [N_G(P)\backslash G/N_G(Q)]} L_g$, where
\begin{equation*}
  L_g:=\Ind_{N_G(P)\cap \lexp{g}{N_G(Q)}}^{N_G(P)}
  \bigl(\Res^{\lexp{g}{N_G(Q)}}_{N_G(P)\cap \lexp{g}{N_G(Q)}}(\gM)\bigr)\,,
\end{equation*}
for $g\in G$. Since $L_1=\Res^{N_G(Q)}_{N_G(P)}(M)$, it suffices to show that for each $g\in G\smallsetminus N_G(Q)$ one has $L_g(P)=\{0\}$. Assume that there exists an element $g\in G\smallsetminus N_G(Q)$ such that $L_g(P)\neq \{0\}$. By Lemma~\ref{lem Brauer con and vertex}(a), there exists an indecomposable direct summand of $L_g$ which has a vertex $S$ that contains $P$. Two applications of \ref{noth Brauer con}(d) imply that there exists $x\in N_G(Q)$ and $y\in N_G(P)$ such that $S\le \lexp{ygx}{R}$. Thus $P\le \lexp{ygx}{R}$ and, since $y\in N_G(P)$, we obtain $P\le \lexp{gx}{R}$, with $gx\notin N_G(Q)$, contradicting the hypothesis of the Lemma. The result now follows.
\end{proof}

\section{Block theoretic preliminaries}\label{sec blocks}

Throughout this section let $G$ and $H$ be finite groups and assume that $(\KK,\calO,F)$ is large enough for $G$ and $H$ and that $F=\calO/(\pi)$ is algebraically closed.


\begin{nothing}\label{noth blocks} {\em Blocks, Brauer pairs, defect groups.}\quad
(a) Recall that $\calO G$ has a unique decomposition $\calO G=B_1\oplus\cdots\oplus B_t$ into indecomposable two-sided ideals, the {\em blocks} of $\calO G$. If one decomposes the identity element of $\calO G$ according to the block decomposition, $1=e_1+\cdots+e_t$, then $e_1,\ldots,e_t$ are central, pairwise orthogonal idempotents of $\calO G$ and $e_i$ is an identity element of $B_i$, called the {\em block idempotent} of $B_i$. Thus, $B_i=e_i\calO G$ is an $\calO$-algebra in its own right. The block idempotents of $\calO G$ are precisely the primitive idempotents of $Z(\calO G)$. The identity element of a block $B$ will be denoted by $e_B$. One obtains a bijection $B\mapsto e_B$ between the blocks of $\calO G$ and the primitive idempotents of $Z(\calO G)$. Similar statements hold for $FG$ and the reduction map $\overline{?}\colon\calO G\to FG$, $\sum_{g\in G}\alpha_g g\mapsto\alphabar_g g$ induces bijections $B_i\mapsto \Bbar_i$ (resp.~$e\mapsto \ebar$) between the blocks of $\calO G$ and the blocks of $FG$ (resp.~the block idempotents of $\calO G$ and those of $FG$). More generally, $B$ is called a {\em sum of blocks} of $\calO G$, if $B=e\calO G$ for some central idempotent $e\neq 0$ of $\calO G$, i.e., $B=\bigoplus_{i\in I} B_i$ and $e=\sum_{i\in I}e_i$ for a unique subset $I$ of $\{1,\ldots,t\}$. In this case $e_B:=e$ is an identity of $B$. This way, sums of blocks and non-zero idempotents of $Z(\calO G)$ are in bijective correspondence. Similarly, one defines sums of blocks of $FG$. Every block idempotent of $\calO[G\times H]=\calO G\otimes_\calO \calO H$ is of the form $e\otimes f$ for uniquely determined block idempotents $e$ of $\calO G$ and $f$ of $\calO H$. Note that we identify the $\calO$-algebras $\calO[G\times H]$ and $\calO G\otimes_\calO \calO H$ via $(g,h)\mapsto g\otimes h$.

\smallskip
(b) Recall that a {\em Brauer pair} of $FG$ is a pair $(P,e)$, where $P$ is a $p$-subgroup of $G$ and $e$ is a block idempotent of $F [C_G(P)]$. Note that the block idempotents of $F[C_G(P)]$ coincide with those of $F[PC_G(P)]$. The group $G$ acts by conjugation on the set of Brauer pairs of $FG$: $\lexp{g}{(P,e)}:=(\gP,\lexp{g}{e})$, for $g\in G$. We denote the $G$-stabilizer of the Brauer pair $(P,e)$ by $N_G(P,e)$. Note that $PC_G(P)\le N_G(P,e)\le N_G(P)$. For Brauer pairs $(P,e)$ and $(Q,f)$ of $FG$, one writes $(Q,f)\unlhd(P,e)$ if $Q\le P\le N_G(Q,f)$ and $\br_P(f)\cdot e = e$ (or equivalently $\br_P(f)\cdot e\neq 0$). The transitive closure of the relation $\unlhd$ on the set of Brauer pairs of $FG$ is denoted by $\le$. It is a partial order and is respected by $G$-conjugation.

If $e$ is a central idempotent of $FG$ then $\br_P(e)$ is an $N_G(P)$-stable central idempotent of $FC_G(P)$ and also a central idempotent of $FN_G(P)$, see~\ref{noth Brauer con}(c). If $B$ is a sum of blocks of $FG$, one says that a Brauer pair $(P,e)$ is a {\em $B$-Brauer pair} if $\br_P(e_B)e=e$, or equivalently, $\br_P(e_B)e\neq 0$. Every Brauer pair is a $B$-Brauer pair for a unique block $B$ of $FG$, and in this case $(\{1\},e_B)\le(P,e)$; see also Proposition~\ref{prop Brauer pairs}(a) below. Let $B$ be again a sum of blocks of $FG$. The set of $B$-Brauer pairs is closed under $G$-conjugation and if $(Q,f)\le (P,e)$ are Brauer pairs of $FG$ then $(P,e)$ is a $B$-Brauer pair if and only if $(Q,f)$ is a $B$-Brauer pair. This follows also immediately from Proposition~\ref{prop Brauer pairs}(a) below. The set of $B$-Brauer pairs is denoted by $\calBP(B)$.

For a sum $B$ of blocks of $\calO G$ we simply define $\calBP(B):=\calBP(\Bbar)$ and call them Brauer pairs of $B$. Thus, Brauer pairs by default are viewed as pairs $(P,e)$, where $e$ is an idempotent of a group algebra over $F$. Sometimes it is convenient to lift the idempotent $e$ to an idempotent over $\calO$. We denote the set of the resulting pairs by $\calBP_\calO(B)$ or $\calBP_\calO(\Bbar)$.

\smallskip
(c) A {\em defect group} of a block $B$ of $FG$ is a subgroup $D$ of $G$, minimal with respect to the property that $e_B\in\tr_D^G((FG)^D)$, see \cite[Section 5.1]{NagaoTsushima1989}. The defect groups of $B$ form a $G$-conjugacy class of $p$-subgroups of $G$. A subgroup $D$ of $G$ is a defect group of $B$ if and only if $\Delta(D)$ is a vertex of $B$, viewed as indecomposable $F[G\times G]$-module, see~\cite[Theorem~10.8]{NagaoTsushima1989}. If $P$ is a normal $p$-subgroup of $G$ then $P$ is contained in each defect group of each block $B$, see~\cite[Theorem~5.2.8]{NagaoTsushima1989}, and $e_B$ is contained in the $F$-span of the $p'$-elements of $C_G(P)$, see \cite[Theorem~3.6.22(ii)]{NagaoTsushima1989}. Similarly one defines defect groups of blocks of $\calO G$. All the above statements hold again over $\calO$ and defect groups don't change under reduction of blocks modulo $\pi$.
\end{nothing}

The following Proposition recalls more standard facts about Brauer pairs, see Theorem~1.8 and 1.14 in \cite{BrouePuig1980}.

\begin{proposition}\label{prop Brauer pairs}
{\rm (a)} For each Brauer pair $(P,e)$ of $FG$ and each subgroup $Q\le P$, there exists a unique Brauer pair $(Q,f)$ of $FG$ such that $(Q,f)\le(P,e)$. In particular, if $(R,g)\le(P,e)$ are Brauer pairs of $FG$ and $R\le Q\le P$ then there exists a unique Brauer pair $(Q,f)$ of $FG$ satisfying $(R,g)\le(Q,f)\le(P,e)$.

\smallskip
{\rm (b)} Let $(Q,f)$ and $(P,e)$ be Brauer pairs of $FG$. If $(Q,f)\le (P,e)$ and $Q\unlhd P$ then $(Q,f)\unlhd(P,e)$.

\smallskip
{\rm (c)} Let $B$ be a block of $FG$. Then the maximal elements in the poset of $B$-Brauer pairs form a single full conjugacy class. Moreover, a $B$-Brauer pair $(P,e)$ is a maximal $B$-Brauer pair if and only if $P$ is a defect group of $B$.
\end{proposition}

The following Proposition is well-known. We give a proof for the reader's convenience. Note that $G\times G$ acts on $FG$ via its $F[G\times G]$-module structure and that $G$ acts on $FG$ via conjugation. These actions are linked via the diagonal embedding $\Delta\colon G\mapsto G\times G$, $g\mapsto (g,g)$, so that $(FG)^{\Delta(H)}=(FG)^H$ for all subgroups $H\le G$.

\begin{proposition}\label{prop Brauer con of block}
Let $B$ be a block of $FG$ and let $e_B$ denote the identity element of $B$. Furthermore, let $(Q,e)$ be a $B$-Brauer pair and set $I:=N_G(Q,e)$.

\smallskip
{\rm (a)} One has $B(\Delta(Q))\cong F[C_G(Q)]\br_Q(e_B)$ as $F[N_{G\times G}(\Delta(Q))]$-modules.

\smallskip
{\rm (b)} One has $eB(\Delta(Q))e\cong F[C_G(Q)]e$ as $F[N_{I\times I}(\Delta(G))]$-modules.
\end{proposition}

\begin{proof}
(a) Recall that $\br_Q\colon (FG)^Q\to F[C_G(Q)]$ is a surjective $F$-algebra homomorphism and an $F[N_{G\times G}(\Delta(Q))]$-module homomorphism. Since $B^Q=(FG)^Q e_B$ and $\br_Q$ is multiplicative, we obtain $\br_Q(B^Q)=\br_Q((FG)^Q)\br_Q(e_B) = F[C_G(Q)]\br_Q(e_B)$. Moreover, we have $\ker(\br_Q)=\sum_{R<Q}\tr_R^Q((FG)^R)$. Thus, $\ker(\br_Q)\cap B^Q = \sum_{R<Q}\tr_R^Q(B^R)$. Altogether, we obtain
\begin{equation*}
  B(\Delta(Q)) = B^Q\big/\sum_{R<Q}\tr_R^Q(B^R) \cong F[C_G(Q)]\br_Q(e_B)
\end{equation*}
as $F[N_{G\times G}(\Delta(Q))]$-modules.

\smallskip
Part~(b) follows immediately from Part~(a).
\end{proof}

\begin{nothing}\label{noth fusion system of B} {\em The fusion system of a block.}
Next we define the fusion system associated to a block, a structure and block invariant introduced by Puig. See \cite{AKO2011} for the definition of fusion systems, saturated fusion systems and basic facts about them. Let $B$ be a block of $FG$ and let $(P,e)$ be a maximal $B$-Brauer pair. For $Q\le P$, denote by $e_Q$ the unique block idempotent of $FC_G(Q)$ such that $(Q,e_Q)\le (P,e)$, cf.~\ref{prop Brauer pairs}(a).

The {\em fusion system} of $B$, associated with $(P,e)$, is the category $\calF$ whose objects are the subgroups of $P$, and whose morphism set $\Hom_\calF(Q,R)$, for subgroups $Q$ and $R$ of $P$, is defined as the set of group homomorphisms arising as conjugation maps $c_g\colon Q\to R$, where $g\in G$ satisfies $\lexp{g}{(Q,e_Q)}\le (R,e_R)$. The category $\calF$ is a saturated fusion system on $P$, see for instance \cite[IV.3]{AKO2011} for a proof.

Recall that, for a general fusion system $\calF$ on a $p$-group $P$, a subgroup $Q$ of $P$ is called {\em fully $\calF$-centralized} (resp.~{\em fully $\calF$-normalized}) if $|C_P(Q)|\ge|C_P(Q')|$ (resp.~$|N_P(Q)|\ge|N_P(Q')|$) for all subgroups $Q'$ of $P$ that are $\calF$-isomorphic to $Q$. Recall also that a subgroup $Q$ of $P$ is called {\em $\calF$-centric} if $C_P(Q')=Z(Q')$ for all $Q'$ that are $\calF$-isomorphic to $Q$.
\end{nothing}

We will need the following result that goes back to Alperin and Brou\'e, see \cite{AlperinBroue1979}. We'll use the formulation given in \cite[Theorems~2.4 and 2.5]{Linckelmann2006} and in~\cite[Theorem~3.11(i)]{Kessar2007}.

\begin{proposition}\label{prop fully centralized and centric}
Let $B$ be a block of $FG$, let $(P,e)$ be a maximal $B$-Brauer pair, and let $\calF$ be the fusion system associated to $B$ and $(P,e)$. For every subgroup $Q$ of $P$, denote by $e_Q$ the unique primitive idempotent in $Z(F[C_G(Q)])$ such that $(Q,e_Q)\le (P,e)$.

\smallskip
{\rm (a)} A subgroup $Q$ of $P$ is fully $\calF$-centralized if and only if $C_P(Q)$ is a defect group of $F[C_G(Q)]e_Q$. In this case $(C_P(Q), e_{QC_P(Q)})$ is a maximal Brauer pair of the block algebra $F[C_G(Q)]e_Q$. In particular, $Q$ is $\calF$-centric if and only if $Z(Q)$ is a defect group of $F[C_G(Q)]e_Q$.

\smallskip
{\rm (b)} A subgroup $Q$ of $P$ is fully $\calF$-normalized if and only if $N_P(Q)$ is a defect group of the block algebra $F[N_G(Q,e_Q)]e_Q$. In this case $(N_P(Q),e_{N_P(Q)})$ is a maximal $F[N_G(Q, e_Q)]e_Q$-Brauer pair.
\end{proposition}

\begin{nothing}\label{noth KP classes} {\em K\"ulshammer-Puig classes.} 
Let $(P,e)$ be a {\em self-centralizing} Brauer pair of $FG$, i.e., such that $Z(P)$ is the defect group of the block $F[C_G(P)]e$. Set $I:=N_G(P,e)$ and $\Ibar:=I/PC_G(P)$. By \cite[Theorems~5.8.10 and 5.8.11]{NagaoTsushima1989}, $P$ is the defect group of the block $F[PC_G(P)]e$ and hence, by \cite[Lemma~5.8.12]{NagaoTsushima1989}, the block algebra $F[PC_G(P)]e$ has a unique simple module $V$. Since $V$ is $I$-stable, the canonical cohomology class $\kappa\in H^2(\Ibar,F^\times)$, assigned to the data $PC_G(P)\trianglelefteq I$ and $V$ by Schur (see \cite[Theorem~3.5.7]{NagaoTsushima1989}),
is called the {\em K\"ulshammer-Puig} class of $(P,e)$.
\end{nothing}

Recall that a surjective group homomorphism $f\colon G\to \Gbar$ induces two functors over any commutative ring $\kk$: The {\em inflation} functor $\Inf_{\Gbar}^G\colon \lMod{\kk \Gbar}\to\lMod{\kk G}$ which is given by restriction along the homomorphism $f$; and the {\em deflation} functor $\Def_{\Gbar}^G\colon \lMod{\kk G}\to\lMod{\kk \Gbar}$ which assigns to a $\kk G$-module $M$ the largest factor-module on which $N:=\ker(f)$ acts trivially. More explicitly, $\Def^G_{\Gbar}(M)= \kk\Gbar \otimes_{\kk G} M \cong M/I_NM$, where $\kk \Gbar$ is viewed as $(\kk \Gbar,\kk G)$-bimodule using $f$ for the right module structure, and where $I_N$ is the ideal of $\kk G$ generated by the elements $x-1$, $x\in N$. Note that $I_N=\ker(\kk G\to\kk \Gbar)$.

\begin{proposition}\label{prop head and def}
Let $B$ be a block of $FG$ with normal defect group $P\trianglelefteq G$.

\smallskip
{\rm (a)} For any $B$-module $M$ one has an $FG$-module isomorphism $M/J(M)\cong\Inf_{G/P}^G\Def^G_{G/P}(M)$.

\smallskip
{\rm (b)} Let $Q$ be a normal subgroup of $G$ with $Q\le P$. Then $M\mapsto \Inf_{G/P}^G\Def^G_{G/P}(M)$ induces a bijection between the set of isomorphism classes of indecomposable $p$-permutation $B$-modules with vertex $Q$ and the set of isomorphism classes of simple $B$-modules.
\end{proposition}

\begin{proof}
Denote by $\overline{\cdot}\colon FG\to F[G/P]$ the canonical $F$-algebra homomorphism. Then, by \cite[Theorems~5.8.10 and 5.8.7(ii)]{NagaoTsushima1989}, $\Bbar$ is a non-zero sum of blocks of $F[G/P]$ of defect $0$ and therefore a semisimple $F$-algebra.

(a) As above, let $I_P$ denote the ideal of $FG$ generated by the elements $x-1$, $x\in P$. It suffices to show that $I_P M=J(M)$. Since $P$ is normal in $G$, $P$ acts trivially on every simple $B$-module and $I_P$ annihilates every simple $B$-module. Thus, $I_P\subseteq J(B)$ and $I_PM\subseteq J(B)M = J(M)$. For the converse it suffices to show that $M/I_PM$ is semisimple as $FG$-module. But the $FG$-module $M/I_PM$ is the inflation of a semisimple $\Bbar$-module, and therefore semisimple.

(b) By Proposition~\ref{prop p-perm equiv}(c) the Brauer construction $M\mapsto M(Q)$ and $\Inf^G_{G/Q}$ define inverse bijections between the set of isomorphism classes of indecomposable $p$-permutation $B$-modules with vertex $Q$ and the set of isomorphism classes of projective indecomposable $F[G/Q]$-modules which after inflation belong to $B$. Each such projective indecomposable $F[G/Q]$-module is the projective cover $P_V$ in $\lmod{F[G/Q]}$ of a simple $B$-module $V$, viewed as $F[G/Q]$-module. Thus, using Part~(a), it suffices to show that for $M=\Inf_{G/Q}^G(P_V)$ one has $M/J(M)\cong V$ as $FG$-modules. But this is clear, since the $FG$-submodules $U$ of $M$ are the same as the $F[G/Q]$-submodules of $P_V$ and the factor module $M/U$ is semisimple as $FG$-module if and only if $P_V/U$ is semisimple as $F[G/Q]$-module.
\end{proof}

\begin{proposition}\label{prop central defect group}
Let $B$ be a block of $FG$ with central defect group $P$, let $Q\le P$, and denote by $\overline{\cdot}\colon FG\to F[G/Q]$ the natural surjective $F$-algebra homomorphism.

\smallskip
{\rm (a)} The image $\Bbar\subseteq F[G/Q]$ of $B$ is a block of $F[G/Q]$ with defect group $P/Q$. Up to isomorphism, $\Bbar$ has a unique simple module $V$.

\smallskip
{\rm (b)} Up to isomorphism, there exists a unique indecomposable $p$-permutation $B$-module $M$ with vertex $Q$. It is isomorphic to the inflation of the unique indecomposable projective $\Bbar$-module $P_V$. Moreover, one has an $F[G/Q]$-module isomorphism $V\cong \Inf_{G/P}^{G/Q}\Def_{G/P}^G(M)$ and $\Inf_{G/Q}^G(V)$ is the unique simple $B$-module.
\end{proposition}

\begin{proof}
(a) This follows from Theorems~5.8.10 and 5.8.11 and from Lemma~5.8.12 in \cite{NagaoTsushima1989}. 

\smallskip
(b) This follows from Proposition~\ref{prop head and def}(b), Part~(a), and Proposition~\ref{prop p-perm equiv}(c).
\end{proof}

\section{Brauer pairs for $p$-permutation modules}\label{sec Brauer pairs for M}

Throughout this section, $G$ denotes a finite group. 
In analogy to Brauer pairs for blocks, we introduce Brauer pairs for $p$-permutation modules.

\begin{definition}\label{def Brauer pairs for M}
Let $M\in\ltriv{\calO G}$ or $M\in\ltriv{FG}$. We call a Brauer pair $(P,e)$ of $FG$ an {\em $M$-Brauer pair} if $M(P,e):=e\cdot M(P)\neq\{0\}$. Note that $M(P,e)$ is an $FIe$-module, where $I:=N_G(P,e)$. For $M\in \ltriv{\calO G}$ the set of $M$-Brauer pairs coincides with the set of $\Mbar$-Brauer pairs. It is denoted by $\calBP(M)$ or $\calBP(\Mbar)$. The corresponding set of Brauer pairs over $\calO$ will be denoted by $\calBP_\calO(M)$ or $\calBP_\calO(\Mbar)$.
\end{definition}

This generalizes the notion of $B$-Brauer pairs for a block $B$ of $\calO G$ in the following sense: A Brauer pair $(P,e)$ of $FG$ is a $B$-Brauer pair as defined in \ref{noth blocks}(b) if and only if $(\Delta(P), e\otimes e^*)$ is a $B$-Brauer pair of the indecomposable $\calO[G\times G]$-module $B$ as defined above.

\begin{nothing}\label{noth covering} We will use the following Morita equivalence between block algebras in the proof of the next two propositions. Recall from \cite[Theorem~5.5.12]{NagaoTsushima1989} that if $(Q,f)$ is a Brauer pair of $\calO G$, if $I:=N_G(Q,f)$, and if $e:=\tr_I^{N_G(Q)}(f)$ is the block idempotent of $\calO [N_G(Q)]$ {\em covering} $f$, i.e., the unique block idempotent $e$ of $\calO[N_G(Q)]$ such that $ef\neq0$, then one has a Morita equivalence between $\lmod{\calO[N_G(Q)]e}$ and $\lmod{\calO If}$ given by tensoring from the left with the $(\calO If,\calO[N_G(Q)]e)$-bimodule $f\calO[N_G(Q)]=f\calO[N_G(Q)]e$. This functor is naturally isomorphic to the functor $f\cdot\Res^{N_G(Q)}_I$. Its inverse is given by tensoring with the $(\calO[N_G(Q)]e,\calO If)$-bimodule $\calO[N_G(Q)]f=e\calO[N_G(Q)]f$. It is naturally isomorphic to the functor $\Ind_{I}^{N_G(Q)}$. This Morita equivalence induces a similar equivalence over $F$.
 \end{nothing}

\smallskip
The following proposition generalizes standard facts on Brauer pairs for blocks to Brauer pairs for $p$-permutation modules. Part~(c) can also be derived from Theorem~2.5 in \cite{Sibley1990}, but we give an independent proof. 

\begin{proposition}\label{prop Brauer pairs for M}
Let $M$ be a $p$-permutation $FG$-module.

\smallskip
{\rm (a)} Assume that $M$ belongs to a sum of blocks $B$ of $FG$, i.e., $e_BM=M$. Then  $\calBP(M)\le \calBP(B)$.

\smallskip
{\rm (b)} $\calBP(M)$ is a $G$-stable ideal in the poset $\calBP(FG)$, i.e., it is stable under $G$-conjugation and if $(Q,f)\le(P,e)$ are Brauer pairs of $FG$ such that $(P,e)$ is an $M$-Brauer pair then also $(Q,f)$ is an $M$-Brauer pair. 

\smallskip
{\rm (c)} Assume that $M$ is indecomposable. Then the maximal $M$-Brauer pairs are precisely the $M$-Brauer pairs $(P,e)$, where $P$ is a vertex of $M$. Moreover, any two maximal $M$-Brauer pairs are $G$-conjugate.
\end{proposition}

\begin{proof}
(a) By Lemma~\ref{lem Brauer map compatibility} we have $\{0\}\neq eM(P) = e((e_B M)(P)) = e\,\br_P(e_B)M(P)$, and therefore $e\,\br_P(e_B)\neq 0$. Thus, $(P,e)$ is a $B$-Brauer pair.

\smallskip
(b) For any Brauer pair $(P,e)$ of $FG$ and any $g\in G$ one has $\lexp{g}{e}\cdot M(\gP)\cong \lexp{g}{(e\cdot M(P))}$ as $F[N_G(\gP,\lexp{g}{e})]$-modules. Thus, $(P,e)$ is an $M$-Brauer pair if and only if $\lexp{g}{(P,e)}$ is an $M$-Brauer pair. 

Now let $(Q,f)\le(P,e)$ be Brauer pairs of $FG$ and assume that $(P,e)$ is an $M$-Brauer pair. In order to show that $(Q,f)$ is an $M$-Brauer pair, we may assume that $(Q,f)\unlhd(P,e)$. Then $P\le N_G(Q,f)=: I$ and $\br_P(f)e=e$. For the FI-module $f\cdot M(Q)$, Lemma~\ref{lem Brauer map compatibility} and Proposition~\ref{prop more on p-perm}(b) imply the following isomorphisms of $FC_G(P)$-modules: $\bigl(f\cdot M(Q)\bigr)(P)\cong \br_P(f)\cdot\bigl(M(Q)(P)\bigr)\cong \br_P(f)\cdot M(P)$. With this we obtain $\{0\}\neq e\cdot M(P) = e\cdot\br_P(f)\cdot M(P) \cong e\cdot\bigl(f\cdot M(Q)\bigr)(P)$, which implies $f\cdot M(Q)\neq\{0\}$.

\smallskip
(c) First we claim that any two $M$-Brauer pairs of the form $(P,e)$, where $P$ is a vertex of $M$, are $G$-conjugate. As the vertices of $M$ are $G$-conjugate, it suffices to fix a vertex $P$ of $M$ and to show that any two $M$-Brauer pairs of the form $(P,e)$ are $N_G(P)$-conjugate. By Proposition~\ref{prop p-perm equiv}(c), the $F[N_G(P)]$-module $M(P)$ is the Green correspondent of $M$. Thus, by Lemma~5.5.4 in \cite{NagaoTsushima1989}, $eM(P)\neq\{0\}$ if and only if the block $eF[C_G(P)]$ is covered by the block of $FN_G(P)$ to which $M(P)$ belongs. But all these blocks $eF[C_G(P)]$ are $N_G(P)$-conjugate, see Lemma~5.5.3 in \cite{NagaoTsushima1989}, and the claim is proved.

In order to prove the statements in (c) it suffices now to show the following claim: Each $M$-Brauer pair $(Q,f)$ is contained in some $M$-Brauer pair $(P,e)$, where $P$ is a vertex of $M$. First note that, since $(Q,f)$ is an $M$-Brauer pair, we have $M(Q)\neq \{0\}$. This implies that $Q$ is contained in a vertex $P$ of $M$ (see Lemma~\ref{lem Brauer con and vertex}(a)). We proceed by induction on the index $[P:Q]$. If $Q=P$, the claim is trivially true. Assume now that $Q<P$. Set $I:=N_G(Q,f)$. Since $(Q,f)$ is an $M$-Brauer pair, we have $fM(Q)\neq \{0\}$ and there exists an indecomposable direct summand $N$ of the $F[N_G(Q)]$-module $M(Q)$ such that $fN\neq \{0\}$. By \ref{noth covering}, the $FIf$-module $fN$ is indecomposable. Assume first that $fN$ has vertex $Q$. Since, by \ref{noth  covering}, $fN\mid \Res^{N_G(Q)}_I(N)$ and $N\cong \Ind_{I}^{N_G(Q)}(fN)$, also $N$ has vertex $Q$ (see \ref{noth Brauer con}(d)). Since $N\mid M(Q)\mid \Res^G_{N_G(Q)}(M)$ (see Lemma~\ref{lem Brauer con and vertex}(d)), also $M$ has vertex $Q$ by the Burry-Carlson-Puig Theorem (see \cite[Theorem 4.4.6]{NagaoTsushima1989}), a contradiction. Thus, the $FIf$-module $fN$ has a vertex $R$ with $Q<R$. By Lemma~\ref{lem Brauer con and vertex}(a) we have $(fN)(R)\neq \{0\}$. Since $fN\mid fM(Q)$, we also have $(fM(Q))(R)\neq \{0\}$. Since $\{0\}\neq(fM(Q))(R)\cong \br_R(f)M(R)$ as $F[C_G(R)]$-modules (see Lemma~\ref{lem Brauer map compatibility} and Proposition~\ref{prop more on p-perm}(b)) we obtain $\br_R(f)\neq 0$,  and since $\br_R(f)$ is central in $F[C_G(R)]$, there exists a block idempotent $e$ of $FC_G(R)$ such that $e\,\br_R(f)=e$ and $e\cdot M(R)\neq\{0\}$. As $R\le I$ and $e\,\br_R(f)=e$, we have $(Q,f)\lhd(R,e)$. Since $e\cdot M(R)\neq\{0\}$, also $(R,e)$ is an $M$-Brauer pair. Applying the induction hypothesis to $(R,e)$ we have $(R,e)\le(P',e')$ for some $M$-Brauer pair $(P',e')$ such that $P'$ is a vertex of $M$. This concludes the proof.
\end{proof}

\begin{proposition}\label{prop isomorphic test}
Let $M$ and $N$ be indecomposable $p$-permutation $FG$-modules, suppose that $(P,e)\in\calBP(FG)$ is both a maximal $M$-Brauer pair and a maximal $N$-Brauer pair, and set $I:=N_G(P,e)$. Then $M(P,e)$ and $N(P,e)$ are indecomposable $p$-permutation $FIe$-modules. Moreover, $M\cong N$ if and only if $M(P,e)\cong N(P,e)$ as $FIe$-modules.
\end{proposition}

\begin{proof}
This is an immediate consequence of Proposition~\ref{prop Brauer pairs for M}(c), the Green correspondence, and the Morita equivalence from \ref{noth covering} (with $(Q,f)$ replaced by $(P,e)$). 
\end{proof}

\section{Extended tensor products and homomorphisms}\label{sec tensor products}

Throughout this section, $G$, $H$, and $K$ denote finite groups and $\kk$ denotes a commutative ring. We consider extended versions of tensor products and homomorphism sets of bimodules for group algebras and prove several basic facts about this construction.

\begin{nothing}\label{noth ext tp and hom}
(a) Let $X\le G\times H$ and $Y\le H\times K$. Further, let $M\in\lmod{\kk X}$ and $N\in\lmod{\kk Y}$. Since $k_1(X)\times k_2(X)\le X$, the $\kk$-module $M$ can be viewed as $(\kk[k_1(X)],\kk[k_2(X)])$-bimodule. Similarly, $N$ can be considered as $(\kk[k_1(Y)],\kk[k_2(Y)])$-bimodule, and $M\otimes_{\kk[k_2(X)\cap k_1(Y)]} N$ is a $(\kk[k_1(X)],\kk[k_2(Y)])$-bimodule. Note that $k_1(X)\times k_2(Y)\le X*Y$ and that this bimodule structure can be extended to a $\kk[X*Y]$-module structure such that, for $(g,k)\in X*Y$, $m\in M$, and $n\in N$, one has
\begin{equation}\label{eqn ext tp}
  (g,k)\cdot (m\otimes n) = (g,h)m\otimes (h,k)n\,,
\end{equation}
where $h\in H$ is chosen such that $(g,h)\in X$ and $(h,k)\in Y$. To the best of our knowledge, this construction was  first used in \cite{Bouc2010b}. We will denote this extended tensor product by $M\tenskXYH N\in \lmod{\kk[X*Y]}$ and obtain a functor $-\tenskXYH-\colon \lmod{\kk X}\times\lmod{\kk Y} \to \lmod{\kk[X*Y]}$. A quick calculation shows that this construction is associative: If also $L$ is a finite group, $Z\le K\times L$, and $P\in\lmod{\kk Z}$ then $(M\tenskXYH N)\tens{X*Y}{Z}{\kk K} P$ and $M\tens{X}{Y*Z}{\kk H}(N\tens{Y}{Z}{\kk K} P)$ are canonically isomorphic under $(m\otimes n)\otimes p\mapsto m\otimes(n\otimes p)$, for $m\in M$, $n\in N$ and $p\in P$. Clearly, $\tenskXYH$ also behaves distributively with respect to direct sums. Moreover, for any $g\in G$, $h\in H$, and $k\in K$ one has an isomorphism
\begin{equation}\label{eqn gen tens and conj}
  \lexp{(g,k)}{(M\tenskXYH N)} \cong \lexp{(g,h)}{M} \tens{\lexp{(g,h)}{X}}{\lexp{(h,k)}{Y}}{\kk H} \lexp{(h,k)}{N}
\end{equation}
of $\kk[\lexp{(g,k)}{(X*Y)}]$-modules, cf.~Lemma~\ref{lem comp form}(d).

\smallskip
(b) Let $X\le H\times G$, $Y\le H\times K$, $M\in\lmod{\kk X}$, and $N\in\lmod{\kk Y}$. Then, $M\in\lmod{\kk[k_1(X)]}_{\kk[k_2(X)]}$, $N\in\lmod{\kk[k_1(Y)]}_{\kk[k_2(Y)]}$ and consequently, $\Hom_{\kk[k_1(X)\cap k_1(Y)]}(M,N)\in\lmod{\kk[k_2(X)]}_{\kk[k_2(Y)]}$. This bimodule structure can be extended to a $\kk[X^\circ*Y]$-module structure satisfying
\begin{equation*}
  \bigl((g,k)\cdot f\bigr)(m) = (h,k)f\bigl((h,g)^{-1}m\bigr)\,,
\end{equation*}
for $(g,k)\in X^\circ*Y$, $f\in \Hom_{\kk[k_1(X)\cap k_1(Y)]}(M,N)$, and $m\in M$, where $h\in H$ is chosen such that $(h,g)\in X$ and $(h,k)\in Y$. We leave the details of this straightforward verification to the reader. We will denote the resulting $\kk[X^\circ*Y]$-module by $\LHom_{\kk H}^{X,Y}(M,N)$. The symbol $\LHom$ is used, since often $G$, $H$ and $K$ will coincide and it might not be clear if one uses homomorphisms of left modules or right modules.

\smallskip
(c) Let $X\le G\times H$, $Y\le K\times H$, $M\in \lmod{\kk X}$, and $N\in\lmod{\kk Y}$. Similarly as in (b), considering homomorphisms with respect to right module structures, we obtain $\Hom_{\kk[k_2(X)\cap k_2(Y)]}(M,N)\in\lmod{\kk[k_1(Y)]}_{\kk[k_1(X)]}$. This bimodule structure can be extended to a $\kk[Y*X^\circ]$-module structure satisfying
\begin{equation*}
  \bigl((k,g)\cdot f\bigr)(m)=(k,h)f\bigl((g,h)^{-1}m\bigr)\,,
\end{equation*}
for $(k,g)\in Y*X^\circ$, $f\in\Hom_{\kk[k_2(X)\cap k_2(Y)]}(M,N)$, and $m\in M$, were $h\in H$ is chosen such that $ (k,h)\in Y$ and $(g,h)\in X$. We denote the resulting $\kk[Y*X^\circ]$-module by $\RHom^{X,Y}_{\kk H}(M,N)$.

\smallskip
(d) Let $X$, $Y$, $M$, and $N$ be as in (c). Note that, by restriction along the flip isomorphism $\tau\colon X\to X^\circ$, $(g,h)\mapsto (h,g)$, we obtain a $\kk X^\circ$-module $M^\tau$ and similarly, a $\kk Y^\circ$-module $N^\tau$. With these operations, one has the equality $\RHom^{X,Y}_{\kk H}(M,N)^\tau = \LHom^{X^\circ, Y^\circ}_{\kk H}(M^\tau,N^\tau)$ of $\kk[X*Y^\circ]$-modules.
\end{nothing}

Note that the tensor product construction in \ref{noth ext tp and hom}(a) generalizes both the tensor product of bimodules (when $X=G\times H$ and $Y=H\times K$) and the internal tensor product of $\kk G$-modules (when $X=Y=\Delta(G)$).

\smallskip
For later use, we will state the following theorem due to Serge Bouc, see \cite{Bouc2010b}.

\begin{theorem}\label{thm Bouc-Mackey}
Let $X\le G\times H$, $Y\le H\times K$, $M\in\lmod{\kk X}$ and $N\in\lmod{\kk Y}$. Then one has an isomorphism
\begin{equation*}
  \Ind_{X}^{G\times H}(M) \otimes_{\kk H} \Ind_{Y}^{H\times K}(N) \cong
  \bigoplus_{t\in [p_2(X)\backslash H/p_1(Y)]} \Ind_{X*\lexp{(t,1)}{Y}} ^{G\times K} ( M\tens{X}{\lexp{(t,1)}{Y}}{\kk H} \lexp{(t,1)}{N} )
\end{equation*}
of $(\kk G,\kk H)$-bimodules.
\end{theorem}

\smallskip
In the sequel we need to establish a list of functorial properties of the extended construction of tensor products and homomorphism functors. The following proposition generalizes the usual adjunction between the tensor product and homomorphism functors.

\begin{proposition}\label{prop ext tp and hom adj 1}
Let $G$, $H$, $K$ and $L$ be finite groups, let $X\le G\times H$, $Y\le H\times K$ and $Z\le G\times L$, and let $M\in\lmod{\kk X}$, $N\in\lmod{\kk Y}$ and $P\in\lmod{\kk Z}$. Then there exists an isomorphism
\begin{equation*}
  \LHom_{\kk G}^{X*Y,Z}\bigl(M\tenskXYH N, P\bigr)\cong \LHom_{\kk H}^{Y,X^\circ*Z}\bigl(N,\LHom_{\kk G}^{X,Z}(M,P)\bigr)
\end{equation*}
of $\kk[(X*Y)^\circ*Z]$-modules which is functorial in $M$, $N$ and $P$.
\end{proposition}

\begin{proof}
It is a straightforward verification that the functions
\begin{align*}
  \Hom_{\kk[k_1(X*Y)\cap k_1(Z)]} \bigl(M\otimes_{\kk[k_2(X)\cap k_1(Y)]} N, P\bigr) 
  & \leftrightarrow \Hom_{\kk[k_1(Y)\cap k_1(X^\circ*Z)]}\bigl(N,\Hom_{\kk[k_1(X)\cap k_1(Z)]}(M,P)\bigr)\\
  f & \mapsto \bigl(n\mapsto \bigl(m\mapsto f(m\otimes n)\bigr)\bigr)\\
  \bigl(m\otimes n\mapsto (f'(n))(m)\bigr) & \lmapsto f'
\end{align*}
are well-defined, mutually inverse homomorphisms of $\kk[(X*Y)^\circ *Z]$-modules and natural in $M$, $N$ and $P$.
\end{proof}

A similar adjunction isomorphism exists for $\RHom$ and can be deduced from the above proposition via the functor $-^\tau$, see \ref{noth ext tp and hom}(d).

\smallskip
In the special case where $Y=X^\circ$ and $Z=X*X^\circ$, the following proposition gives a different type of adjunction. Recall from Lemma~\ref{lem comp form}(b) that $X*X^\circ*X=X$ and $X^\circ*X*X^\circ=X^\circ$.

\begin{proposition}\label{prop ext tp and hom adj 2}
Let $X\in G\times H$, $M\in\lmod{\kk X}$, $N\in\lmod{\kk X^\circ}$, and $P\in\lmod{\kk[X*X^\circ]}$. 

\smallskip
{\rm (a)} There exists an isomorphism
\begin{equation*}
  \Hom_{\kk[X*X^\circ]}(M\tens{X}{X^\circ}{\kk H} N, P) \cong 
  \Hom_{\kk X^\circ}\bigl(N,\LHom^{X,X*X^\circ}_{\kk G}(M,P)\bigr)
\end{equation*}
of $\kk$-modules which is natural in $M$, $N$ and $P$.

\smallskip
{\rm (b)} There exists an isomorphism
\begin{equation*}
  \Hom_{\kk[X*X^\circ]}(M\tens{X}{X^\circ}{\kk H} N, P) \cong 
  \Hom_{\kk X}\bigl(M,\RHom^{X^\circ,X*X^\circ}_{\kk G}(N,P)\bigr)
\end{equation*}
of $\kk$-modules which is natural in $M$, $N$ and $P$.
\end{proposition}

\begin{proof}
(a) Again, it is straightforward to verify that  the maps
\begin{align*}
  \Hom_{\kk[X*X^\circ]} \bigl(M\otimes_{\kk k_2(X)} N, P\bigr) 
  & \leftrightarrow \Hom_{\kk X^\circ}\bigl(N,\Hom_{\kk k_1(X)}(M,P)\bigr)\\
  f & \mapsto \bigl(n\mapsto \bigl(m\mapsto f(m\otimes n)\bigr)\bigr)\\
  \bigl(m\otimes n\mapsto (f'(n))(m)\bigr) & \lmapsto f'
\end{align*}
are well-defined, mutually inverse $\kk$-module homomorphisms which are natural in $M$, $N$ and $P$.

\smallskip
(b) The maps analogous to the ones in (a) give again the desired isomorphisms.
\end{proof}

The following lemma generalizes a well-known compatibility of the tensor product and induction in the special case that $G=H=K$, $X=Y=\Delta(G)$.

\begin{lemma}\label{lem ext tp and ind}
Let $X'\le X\le G\times H$ and  $Y'\le Y\le H\times K$.

\smallskip
{\rm (a)} Let $M'\in \lmod{\kk X'}$ and $N\in\lmod{\kk Y}$. There exists a $\kk[X*Y]$-module homomorphism
\begin{equation*}
   \alpha_1 \colon \Ind_{X'*Y}^{X*Y}(M'\tens{X'}{Y}{\kk H} N) \to \Ind_{X'}^{X}(M')\tenskXYH N
\end{equation*}
which maps $(g,k)\otimes(m'\otimes n)$ to $\bigl((g,h)\otimes m'\bigr)\otimes (h,k)n$, for $(g,k)\in X*Y$, $m'\in M'$ and $n\in N$, where $h\in H$ is chosen such that $(g,h)\in X$ and $(h,k)\in Y$. The homomorphism $\alpha_1$ is functorial in $M'$ and $N$, and if $p_2(X)\le p_1(Y)$ then $\alpha_1$ is an isomorphism.

\smallskip
{\rm (b)} Let $M\in\lmod{\kk X}$ and $N'\in\lmod{\kk Y'}$. There exists a $\kk[X*Y]$-module homomorphism 
\begin{equation*}
   \alpha_2\colon \Ind_{X*Y'}^{X*Y}(M\tens{X}{Y'}{\kk H} N')\to M\tenskXYH \Ind_{Y'}^{Y}(N')
\end{equation*}
which maps $(g,k)\otimes (m\otimes n')$ to $(g,h)m\otimes\bigl((h,k)\otimes n')$, for $(g,k)\in X*Y$, $m\in M$, and $n'\in N'$, where $h\in H$ is chosen such that $(g,h)\in X$ and $(h,k)\in Y$. The homomorphism $\alpha_2$ is functorial in $M$ and $N'$, and if $p_2(X)\ge p_1(Y)$ then it is an isomorphism.
\end{lemma}

\begin{proof}
We only prove Part~(a). Part~(b) is proved similarly. It is straightforward to verify that the map $\alpha_1$ is well-defined
and a natural $\kk[X*Y]$-module homomorphism. Now assume that $p_2(X)\le p_1(Y)$. It is again straightforward to verify that one obtains a well-defined $\kk[X*Y]$-module homomorphism $\beta_1\colon \Ind_{X'}^{X}(M')\tenskXYH N \to \Ind_{X'*Y}^{X*Y}(M'\tens{X'}{Y}{\kk H} N)$ by mapping $\bigl((g,h)\otimes m'\bigr)\otimes n$ to $(g,k)\otimes\bigl(m'\otimes (h,k)^{-1}n\bigr)$, for $(g,h)\in X$, $m'\in M'$, and $n\in N$, were $k\in K$ is chosen such that $(h,k)\in Y$ (using $p_2(X)\le p_1(Y))$. It is obvious that $\alpha_1$ and $\beta_1$ are inverses.
\end{proof}

\begin{nothing}\label{noth ext hom duals}
Let $X\le G\times H$ and $M\in\lmod{\kk X}$. Recall from Lemma~\ref{lem comp form}(a) that $X*X^\circ=(k_1(X)\times\{1\})\Delta(p_1(X))= \{(g,g')\in p_1(X)\times p_1(X)\mid gk_1(X)=g'k_1(X)\}$.  Moreover, $\kk[k_1(X)]$ can be considered as left $\kk[X*X^\circ]$-module via $(g,g')n:=gng'^{-1}$, for $n\in k_1(X)$ and $(g,g')\in X*X^\circ$. Thus, we obtain a $\kk[X]$-module $\kk[k_1(X)]\tens{X*X^\circ}{X}{\kk G}M$ and a $\kk[X^\circ]$-module $\LHom^{X,X*X^\circ}_{\kk G}(M,\kk[k_1(X)])$, since $X*X^\circ*X=X$ and $X^\circ*X*X^\circ=X^\circ$ (see Lemma~\ref{lem comp form}(b)).

\smallskip
Similarly, $\kk[k_2(X)]$ is a left $\kk[X^\circ*X]$-module via $(h,h')n=hnh'^{-1}$ for $n\in k_2(X)$ and $(h,h')\in X^\circ * X$. Thus, one obtains a $\kk[X]$-module $M\tens{X}{X^\circ*X}{\kk H}\kk[k_2(X)]$ and a $\kk[X^\circ]$-module $\RHom^{X,X^\circ*X}_{\kk H}(M,\kk[k_2(X)])$.
\end{nothing}

We leave the straightforward proof of the following proposition to the reader.

\begin{proposition}\label{prop ext hom duals}
Let $X\le G\times H$ and $M\in\lmod{\kk X}$. 

\smallskip
{\rm (a)} The map
\begin{equation*}
  M^\circ\to \LHom_{\kk G}^{X,X*X^\circ}(M,\kk[k_1(X)])\,, \quad
  \lambda\mapsto\bigl(m\mapsto\sum_{g\in k_1(X)} \lambda(g^{-1}m)g\bigr)\,,
\end{equation*}
is a well-defined isomorphism of $\kk X^\circ$-modules. Its inverse maps the homomorphism $f$ to $\bigl(m\mapsto t(f(m))\bigr)$, where $t$ denotes the $\kk$-linear extension of $k_1(X)\to\kk$, $g\mapsto\delta_{g,1}$.

\smallskip
{\rm (b)} The map
\begin{equation*}
  M^\circ\to \RHom_{\kk H}^{X,X^\circ*X}(M,\kk[k_2(X)])\,,\quad
  \lambda\mapsto \bigl(m\mapsto\sum_{h\in k_2(X)} \lambda(mh^{-1}) h\bigr)\,,
\end{equation*}
is a well-defined isomorphism of $\kk X^\circ$-modules. Its inverse maps $f$ to $\bigl(m\mapsto t(f(m))\bigr)$.

\smallskip
{\rm (c)} The maps $a\otimes m\mapsto am$ and $m\otimes a\mapsto ma$ define $\kk[X]$-module isomorphisms
\begin{equation*}
  \kk[k_1(X)]\tens{X*X^\circ}{X}{\kk G}M\to M \quad \text{and} \quad M\tens{X}{X^\circ*X}{\kk H}\kk[k_2(X)]\to M
\end{equation*}
with inverses given by $m\mapsto 1\otimes m$ and $m\mapsto m\otimes 1$.

\smallskip
{\rm (d)} Assume that $K\le k_1(X)$ with $K\trianglelefteq p_1(X)$, that $e\in Z(\kk[k_1(X)])^{p_1(X)}$ is an idempotent, and that $M\in\lmod{\kk X(e\otimes 1)}$. Then one has an isomorphism of $\kk X$-modules
\begin{equation*}
  \Inf_{X/(K\times\{1\})}^X\Def^X_{X/(K\times\{1\})}(M)
  = \kk[X/(K\times\{1\})]\otimes_{\kk X} M
  \cong \kk[k_1(X)/K]\ebar \tens{X*X^\circ}{X}{\kk G} M\,,
\end{equation*}
given by $(g,h)(K\times\{1\})\otimes m\mapsto \ebar\otimes (g,h)m$ with inverse $(gK)\ebar\otimes m\mapsto 1\otimes (g,1)m$,
where $\ebar$ is the image of $e$ under the canonical $\kk$-algebra homomorphism $\kk[k_1(X)]\to\kk[k_1(X)/K]$ and where $\kk[k_1(X)/K]\ebar$ is a $\kk[X*X^\circ]$-module via $(g_1,g_2)\cdot a= g_1ag_2^{-1}$ for $(g_1,g_2)\in X*X^\circ$ and $a\in \kk[k_1(X)/K]\ebar$.
\end{proposition}

\begin{proposition}\label{prop ext hom proj}
{\rm (a)} Let $X\le G\times H$, $Y\le G\times K$, $M\in\lmod{\kk X}$, and $N\in\lmod{\kk Y}$. The map
\begin{equation*}
  \alpha_1\colon\LHom_{\kk G}^{X,X*X^\circ}(M,\kk[k_1(X)])\tens{X^\circ}{Y}{\kk G}N\to
  \LHom_{\kk G}^{X,Y}(M,N)\,, 
  \quad f\otimes n\mapsto \bigl(m\mapsto t(f(m))n\bigr)\,,
\end{equation*}
is a well-defined $\kk [X^\circ*Y]$-module homomorphism and functorial in $M$ and $N$. Here $t$ denotes the canonical projection map $\kk[k_1(X)]\to\kk[k_1(X)\cap k_1(Y)]$. If $M$ is projective as left $\kk[k_1(X)]$-module then $\alpha_1$ is an isomorphism.

\smallskip
{\rm (b)} Let $X\le G\times H$, $Y\le K\times H$, $M\in\lmod{\kk X}$, and $N\in\lmod{\kk Y}$. The map
\begin{equation*}
  \alpha_2\colon N \tens{Y}{X^\circ}{\kk H} \RHom_{\kk H}^{X,X^\circ*X}(M,\kk[k_2(X)])\to
  \RHom_{\kk H}^{X,Y}(M,N)\,,
  \quad n\otimes f\mapsto\bigl(m\mapsto nt(f(m))\bigr)\,,
\end{equation*}
is a well-defined $\kk[Y*X^\circ]$-module homomorphism and functorial in $M$ and $N$. Here, $t$ denotes the canonical projection map $\kk[k_2(X)]\to\kk[k_2(X)\cap k_2(Y)]$. If $M$ is projective as right $\kk[k_2(X)]$-module then $\alpha_2$ is an isomorphism.
\end{proposition}  

\begin{proof}
(a) Set $A:=k_1(X)\cap k_1(Y)$ and $B:=k_1(X)$. It is straightforward to verify that
\begin{equation*}
  \alpha_1\colon \Hom_{\kk B}\bigl(M,\kk B\bigr)\otimes _{\kk A} N \to \Hom_{\kk A}(M,N)\,,\quad
  (f\otimes n)\mapsto \bigl(m\mapsto t(f(m))n\bigr)\,,
\end{equation*}
is well-defined and functorial in $M$ and $N$. A careful but straightforward computation also shows that $\alpha_1$ is a $\kk[X^\circ*Y]$-module homomorphism. Next assume that $M$ is projective as left $\kk B$-module. Note that, for fixed $N$, the map $\alpha_1$ is also a natural transformation between two additive contravariant functors $\lmod{\kk B}\to\lmod{\kk}$. Thus, in order to show that $\alpha_1$ is an isomorphism it suffices to show this when $M=\kk B$. Using the natural isomorphism $\kk B\to \Hom_{\kk B}(\kk B,\kk B)$, $b\mapsto \rho_b$, with $\rho_b(b')=b'b$ for $b,b'\in B$, it suffices to show that the $\kk$-linear map
\begin{equation*}
  \alphatilde_1\colon \kk B\otimes_{\kk A} N\to \Hom_{\kk A}(\kk B,N)\,,\quad
  b\otimes n\mapsto \bigl(b'\mapsto t(b'b)n\bigr)\,,
\end{equation*}
is bijective. But this is easily verified: If $b_1,\ldots,b_d\in B$ are representatives of $B/A$, then the map
\begin{equation*}
  \betatilde_1\colon \Hom_{\kk A}(\kk B,N)\to \kk B\otimes_{\kk A} N\,,\quad
  f'\mapsto\sum_{i=1}^d b_i\otimes f'(b_i^{-1})\,,
\end{equation*}
is an inverse to $\alphatilde_1$.

\smallskip
(b) This is proved in a similar way as Part~(a).
\end{proof}

\begin{corollary}\label{cor ext tp and hom}
Let $X\le G\times H$, $M\in\lmod{\kk X}$, $N\in\lmod{\kk X^\circ}$, and $P\in\lmod{\kk[X*X^\circ]}$. 

\smallskip
{\rm (a)} If $M$ is projective as left $\kk[k_1(X)]$-module then one has a $\kk$-module isomorphism
\begin{equation*}
  \Hom_{\kk[X*X^\circ]}\bigl(M\tens{X}{X^\circ}{\kk H} N, P\bigr) \cong 
  \Hom_{\kk X^\circ}\bigl(N,M^\circ\tens{X^\circ}{X*X^\circ}{\kk G} P\bigr)\,.
\end{equation*}

\smallskip
{\rm (b)} If $N$ is projective as right $\kk[k_1(X)]$-module then  one has a $\kk$-module isomorphism
\begin{equation*}
  \Hom_{\kk[X*X^\circ]}\bigl(M\tens{X}{X^\circ}{\kk H} N, P\bigr) \cong 
  \Hom_{\kk X}\bigl(M, P\tens{X*X^\circ}{X}{\kk G} N^\circ\bigr)\,.
\end{equation*}
\end{corollary}

\begin{proof}
This follows immediately from Propositions~\ref{prop ext tp and hom adj 2}, \ref{prop ext hom proj}, and \ref{prop ext hom duals}.
\end{proof}

The following lemma will be used in the proof of Lemma~\ref{lem unique Lambdahat lr}. Recall the definition of $N_{(S,\phi,T)}$ from \ref{not Nphi}.

\begin{lemma}\label{lem ext tp special}
Let $\phi\colon Q\myiso P$ be an isomorphism between subgroups $Q\le H$ and $P\le G$. Furthermore, let $C_G(P)\le S\le N_G(P)$ and $C_H(Q)\le T\le N_H(Q)$ be intermediate groups. Set $X:=N_{S\times S}(\Delta(P))$, $X':=\Delta(N_{(S,\phi,T)})(C_G(P)\times\{1\})\le G\times G$, and $Y:=N_{S\times T}(\Delta(P,\phi,Q))\le G\times H$.

\smallskip
{\rm (a)} One has $Y*Y^\circ=X'\le X=X*X$, $X*X'=X'$ and $X*Y=X'*Y=Y$. 

\smallskip
{\rm (b)} Let $M\in\lmod{\kk Y}$, $N\in\lmod{\kk Y^\circ}$,  and $V,W\in\lmod{\kk X}$. If $N$ is projective as right $\kk[C_G(P)]$-module then one has a $\kk$-linear isomorphism
\begin{equation*}
  \Hom_{\kk Y}\bigl(V\tens{X}{Y}{\kk G} M, W\tens{X}{Y}{\kk G} N^\circ\bigr) \cong 
  \Hom_{\kk X}\bigl(V\tens{X}{X}{\kk G} \Ind_{X'}^X(M\tens{Y}{Y^\circ}{\kk H} N), W\bigr)\,.
\end{equation*}
\end{lemma}

\begin{proof}
(a) All assertions follow immediately from Lemma~\ref{lem comp form} and Proposition~\ref{prop Nphi}.

\smallskip
(b) By Part~(a), Lemma~\ref{lem ext tp and ind}(b), and the associativity property in \ref{noth ext tp and hom}(a) we obtain isomorphisms
\begin{equation*}
  V\tens{X}{X}{\kk G}\Ind_{X'}^X(M\tens{Y}{Y^\circ}{\kk H} N) \cong 
  \Ind_{X*X'}^{X*X}\bigl(V\tens{X}{X'}{\kk G} (M\tens{Y}{Y^\circ}{\kk H} N)\bigr) \cong
  \Ind_{X'}^X\bigl((V\tens{X}{Y}{\kk G} M) \tens{X*Y}{Y^\circ}{\kk H} N\bigr)
\end{equation*}
of $\kk {X}$-modules. Thus, using again Part~(a), the usual adjunction of $\Ind$ and $\Res$, and Corollary~\ref{cor ext tp and hom}(b), we obtain isomorphisms
\begin{align*}
  & \Hom_{\kk X}\bigl(V\tens{X}{X}{\kk G} \Ind_{X'}^X(M\tens{Y}{Y^\circ}{\kk H} N), W\bigr)
   \cong \Hom_{\kk X}\bigl(\Ind_{X'}^X\bigl((V\tens{X}{Y}{\kk G} M) \tens{X*Y}{Y^\circ}{\kk H} N\bigr), W\bigr)   
                                              \cong \\
  & \cong \Hom_{\kk X'}\bigl((V\tens{X}{Y}{\kk G} M) \tens{X*Y}{Y^\circ}{\kk H} N, \Res^{X}_{X'}(W)\bigr)
  \cong \Hom_{\kk Y}\bigl(V\tens{X}{Y}{\kk G} M, \Res^X_{X'}(W)\tens{X'}{Y}{\kk G} N^\circ\bigr)
\end{align*}
of $\kk$-modules, since $N$ is projective as right module for $\kk[C_G(P)]=\kk[k_1(Y^\circ)]$, see Proposition~\ref{prop Nphi}(c). Finally, $\Res^X_{X'}(W)\tens{X'}{Y}{\kk G} N^\circ = W \tens{X}{Y}{\kk G} N^\circ$ as $\kk[Y]$-modules, since $X'*Y=X*Y=Y$ by Part~(a) and since $k_2(X')=C_G(P)=k_2(X)$ by Proposition~\ref{prop Nphi}(c). This completes the proof.
\end{proof}


\section{Tensor products of $p$-permutation bimodules}\label{sec p-perm bimodules}

Throughout this section, $G$, $H$ and $K$ denote finite groups.

\begin{nothing}\label{noth ext tp and res}
Let $\kk$ be a commutative ring. The following observation will allow us to create situations where $\alpha_1$ and $\alpha_2$ in Lemma~\ref{lem ext tp and ind} are isomorphisms. Let $X\le G\times H$, $Y\le H\times K$, $M\in\lmod{\kk X}$, and $N\in\lmod{\kk Y}$. We define
\begin{equation*}
  \Xtilde:=\{(g,h)\in X\mid \exists k\in K\colon (h,k)\in Y\}\,.
\end{equation*}
Then $\Xtilde\le X$, $p_2(\Xtilde)\le p_1(Y)$, $k_2(\Xtilde)\cap k_1(Y)=k_2(X)\cap k_1(Y)$, and $\Xtilde*Y=X*Y$. Moreover, one has $\Res^X_{\Xtilde}(M)\tens{\Xtilde}{Y}{\kk H} N = M\tenskXYH N$ as $\kk[X*Y]$-modules.  Similarly, one can define $\Ytilde\le Y$ with the analogous properties. The use of this method is illustrated in the proof of the following lemma.
\end{nothing}

For the remainder of this section we assume that $\calO$ is a complete discrete valuation ring of characteristic $0$ with residue field $F$ of characteristic $p>0$.

\begin{lemma}\label{lem ext tp and vertices}
Let $X\le G\times H$ and $Y\le H\times K$.

\smallskip
{\rm (a)} If $M\in\ltriv{\calO X}$ and $N\in\ltriv{\calO Y}$ then $M\tensOXYH N\in\ltriv{\calO[X*Y]}$.

\smallskip
{\rm (b)} If $M\in\lmod{\calO X}$ and $N\in\lmod{\calO Y}$ are indecomposable with twisted diagonal vertices then each indecomposable direct summand of the $\calO[X*Y]$-module $M\tensOXYH N$ has twisted diagonal vertices.
\end{lemma}

\begin{proof}
(a) This can easily be seen by using a tensor product construction on bisets. We give a different proof to illustrate the method from \ref{noth ext tp and res}. Since $\Res^X_{\Xtilde}(M)$ is again a $p$-permutation module, we may assume that $p_2(X)\le p_1(Y)$. Next, since $\tensOXYH$ respects direct sums, we may assume that $M=\Ind_{X'}^X(\calO_{X'})$ for some subgroup $X'$ of $X$. Using Lemma~\ref{lem ext tp and ind}(a), we obtain $M\tensOXYH N \cong \Ind_{X'*Y}^{X*Y}(\calO_{X'}\tens{X'}{Y}{\calO H} N)$. Since the class of $p$-permutation modules is stable under induction, we may assume that $X=X'$ and $M=\calO_{X}$. Similar arguments for $Y$ and $N$ reduce further to the case that $N=\calO_Y$. But in this case we have $M\tensOXYH N= \calO_{X*Y}$, the trivial module, which is a $p$-permutation module.

\smallskip
(b) Since $\tensOXYH$ respects direct sums and by \ref{noth Brauer con}(d), we may use \ref{noth ext tp and res} to reduce to the case where $p_2(X)\le p_1(Y)$. Since $\tensOXYH$ respects direct sums, we may also assume that $M=\Ind_{X'}^X(M')$ for some twisted diagonal subgroup $X'$ of $X$ and some indecomposable module $M'\in \lmod{\calO X'}$. Lemma~\ref{lem ext tp and ind}(a) now implies that $M\tensOXYH N\cong \Ind_{X'*Y}^{X*Y}(M'\tens{X'}{Y}{\calO H} N)$, and by \ref{noth Brauer con}(d) we may assume that $X$ is twisted diagonal. Similar considerations for $Y$ and $N$ reduce further to the situation that also $Y$ is twisted diagonal. But then also $X*Y$ and its subgroups are twisted diagonal. This completes the proof.
\end{proof}

Next we recall and improve Theorem~3.3 from \cite{BoltjeDanz2012}. For a fixed twisted diagonal $p$-subgroup $\Delta(P,\sigma,R)\le G\times K$ we denote by $\Gamma=\Gamma_H(P,\sigma,R)$ the set of triples $(\phi,Q,\psi)$, where $Q\le H$ and $\psi\colon R\myiso Q$ and $\phi\colon Q\myiso P$ are isomorphisms with $\phi\circ\psi=\sigma$. The group $H$ acts on $\Gamma_H(P,\sigma,R)$ by $\lexp{h}{(\phi,Q,\psi)}:=(\phi c_h^{-1}, \hQ,c_h\psi)$. Note that $\stab_H(\phi,Q,\psi)=C_H(Q)$. For $M\in\lmod{\calO G}_{\calO H}$, $N\in\lmod{\calO H}_{\calO K}$, and $(\phi,Q,\psi)\in\Gamma_H(P,\sigma,Q)$, one has an $(F[C_G(P)],F[C_K(R)])$-bimodule homomorphism 
\begin{align}\label{eqn BD homomorphism}
   \Phi_{(\phi,Q,\psi)}\colon \
   M(\Delta(P,\phi,Q))\otimes_{F[C_H(Q)]}N(\Delta(Q,\psi,R)) & \ \to\  (M\otimes_{\calO H} N)(\Delta(P,\sigma,R))\,, \\ \notag
   \Br_{\Delta(P,\phi,Q)}(m)\otimes \Br_{\Delta(Q,\psi,R)}(n) & \ \mapsto\ \Br_{\Delta(P,\sigma,R)}(m\otimes n)\,, 
\end{align}
see \cite[3.1(h)]{BoltjeDanz2012}, which is natural in $M$ and $N$. The following theorem describes $(M\otimes_{\calO H} N)(\Delta(P,\sigma,R))$ as $(F[C_G(P)],F[C_K(R)])$-bimodule, for a particular class of modules $M$ and $N$.

\begin{theorem}(\cite[Theorem~3.3]{BoltjeDanz2012})\label{thm BD theorem}
Let $\Delta(P,\sigma,R)\le G\times K$ be a twisted diagonal $p$-subgroup, let $\Gamma=\Gamma_H(P,\sigma,R)$ be as above and let $\Gammatilde\subseteq \Gamma$ be a set of representatives of the $H$-orbits of $\Gamma$. Furthermore let $M\in\ltriv{\calO G}_{\calO H}$ and $N\in\ltriv{\calO H}_{\calO K}$ be $p$-permutation bimodules all of whose indecomposable direct summands have twisted diagonal vertices. Then the direct sum of the homomorphisms $\Phi_{(\phi,Q,\psi)}$, $(\phi,Q,\psi)\in\Gammatilde$ yields an isomorphism
\begin{equation}\label{eqn BD}
  \Phi\colon\ \bigoplus_{(\phi,Q,\psi)\in\Gammatilde} 
  M(\Delta(P,\phi,Q))\otimes_{F[C_H(Q)]}N(\Delta(Q,\psi,R))\ \to\ (M\otimes_{\calO H} N)(\Delta(P,\sigma,R))
\end{equation}
of $(F[C_G(P)],F[C_K(R)])$-bimodules which is natural in $M$ and $N$.
\end{theorem}

Note that not only $C_G(P)\times C_K(R)$, but also the bigger group $N_{G\times K}(\Delta(P,\sigma,R))$ acts on the right hand side of the isomorphism (\ref{eqn BD}). The corollary below describes the $F[N_{G\times K}(\Delta(P,\sigma,R))]$-module structure of $(M\otimes_{\calO H} N)(\Delta(P,\sigma,R))$. First note that the domain of the homomorphism $\Phi_{(\phi,Q,\psi)}$ in (\ref{eqn BD homomorphism}) carries an $F[N_{G\times H}(\Delta(P,\phi,Q))*N_{H\times K}(\Delta(Q,\psi,R))]$-module structure via the extended tensor product construction in~\ref{noth ext tp and hom}. 
This module structure extends the $F[C_G(P)]\times F[C_K(R)]$-module structure from~(\ref{eqn BD homomorphism}), since $k_1(N_{G\times H}(\Delta(P,\phi,Q)))=C_G(P)$, $k_2(N_{G\times H}(\Delta(P,\phi,Q)))=C_H(Q)=k_1(N_{H\times K}(\Delta(Q,\psi,R)))$ and $k_2(N_{H\times K}(\Delta(Q,\psi,R)))=C_K(R)$ by Proposition~\ref{prop Nphi}(b).
 It is a straightforward verification that $\Phi_{(\phi,Q,\psi)}$ in~(\ref{eqn BD homomorphism}) actually defines a homomorphism
\begin{equation}\label{eqn extended BD homomorphism}
   \Phi_{(\phi,Q,\psi)}\colon\ M(\Delta(P,\phi,Q)) \tens{N_{G\times H}(\Delta(P,\phi,Q))}{N_{H\times K}(\Delta(Q,\psi,R))}{FH}\ \to\ 
   (M\otimes_{\calO H} N)(\Delta(P,\sigma,R))
\end{equation}
of $F[N_{G\times H}(\Delta(P,\phi,Q))*N_{H\times K}(\Delta(Q,\psi,R))]$-modules.

\begin{corollary}\label{cor BD corollary}
Let $M\in\ltriv{\calO G}_{\calO H}$, $N\in\ltriv{\calO H}_{\calO K}$, $(P,\sigma,R)\le G\times K$, and $\Gammatilde\subseteq\Gamma=\Gamma_H(P,\sigma,R)$ be as in Theorem~\ref{thm BD theorem}. 

\smallskip
{\rm (a)} The group $N_{G\times K}(\Delta(P,\sigma,R))\times H$ acts on $\Gamma$ via $\lexp{((g,k),h)}{(\phi,Q,\psi)}:= (c_g\phi c_h^{-1}, \hQ, c_h\psi c_k^{-1})$. 
For the induced action of $N_{G\times K}(\Delta(P,\sigma,R))$ on the $H$-orbits $[\phi,Q,\psi]_H$ of $\Gamma$ one has $\stab_{N_{G\times K}(\Delta(P,\sigma,R))}([\phi,Q,\psi]_H) = N_{G\times H}(\Delta(P,\phi,Q)) * N_{H\times K}(\Delta(Q,\psi,R))$, for each $(\phi,Q,\psi)\in\Gamma$.

\smallskip
{\rm (b)} Let $\Gammahat\subseteq\Gammatilde$ be a set of representatives of the $N_{G\times K}(\Delta(P,\sigma,R))\times H$-orbits of $\Gamma$. The homomorphisms $\Phi_{(\phi,Q,\psi)}$, $(\phi,Q,\psi)\in\Gammahat$, in (\ref{eqn extended BD homomorphism}) induce an isomorphism
\begin{equation*}
   \Phi\colon\ \bigoplus_{\gamma=(\phi,Q,\psi)\in\Gammahat} 
   \Ind_{X(\gamma)*Y(\gamma)}^{N_{G\times K}(\Delta(P,\sigma,R))}
   \Bigl(M(\Delta(P,\phi,Q))\tens{X(\gamma)}{Y(\gamma)}{FH} N(\Delta(Q,\psi,R))\Bigr)\ \myiso\
   (M\otimes_{\calO H} N)(\Delta(P,\sigma,R))
\end{equation*}
of $F[N_{G\times K}(\Delta(P,\sigma,R))]$-modules which is natural in $M$ and $N$, with $X(\gamma):=N_{G\times H}(\Delta(P,\phi,Q))$ and $Y(\gamma):=N_{H\times K}(\Delta(Q,\psi,R))$, for $\gamma=(\phi,Q,\psi)\in\Gamma$.
\end{corollary}

\begin{proof}
(a) The first statement is clear and the second statement is a straightforward verification.

\smallskip
(b) First note that, for $(\phi,Q,\psi)\in\Gamma$, one has $X(\gamma)*Y(\gamma)\le N_{G\times K}(\Delta(P,\sigma,R))$, by Part~(a). 
Now let $(\phi,Q,\psi)\in\Gammatilde$, $(g,k)\in N_{G\times K}(\Delta(P,\sigma,R))$, $m\in M^{\Delta(P,\phi,Q)}$, and $n\in N^{\Delta(P,\psi,R)}$. 
Then there exists $(\phi',Q',\psi')\in\Gammatilde$ and $h\in H$ such that $\lexp{((g,k),h)}{(\phi,Q,\psi)}=(\phi',Q',\psi')$, and we have
\begin{align*}
   &(g,k)\cdot  \Phi_{(\phi,Q,\psi)}\bigl(\Br_{\Delta(P,\phi,Q)}(m)\otimes Br_{\Delta(Q,\psi,R)}(n)\bigr) =
           (g,k)\cdot \Br_{\Delta(P,\sigma,R)}(m\otimes n) = \Br_{\Delta(P,\sigma,R)}(gm\otimes nk^{-1}) \\
    &\quad  = \Br_{\Delta(P,\sigma,R)}(gmh^{-1}\otimes hnk^{-1}) 
              = \Phi_{(\phi',Q',\psi')}\bigl( \Br_{\Delta(P,\phi',Q')}(gmh^{-1}) \otimes \Br_{\Delta(Q',\psi',R)}(hnk^{-1}) \bigr)\,.
\end{align*}
This equation shows that if one transports the $F[N_{G\times K}(\Delta(P,\sigma,R))]$-module structure from the right hand side of the isomorphism $\Phi$ in (\ref{eqn BD}) via $\Phi^{-1}$ to the left hand side, then this action permutes the direct summands according to the action of $N_{G\times K}(\Delta(P,\sigma,R))$ on $H\backslash\Gamma\cong \Gammatilde$. Moreover, the stabilizer of the $(\phi,Q,\psi)$-component equals $N_{G\times H}(\Delta(P,\phi,Q))*N_{H\times K}(\Delta(Q,\psi,R)$ by Part~(a). Thus, the isomorphism in~(\ref{eqn BD}) defines the desired $F[N_{G\times K}(\Delta(P,\sigma,R))]$-module isomorphism of Part~(b).
\end{proof}

In the next theorem we will give block-wise versions of Theorem~\ref{thm BD theorem} and Corollary~\ref{cor BD corollary}. Note that for a twisted diagonal $p$-subgroup $\Delta(P,\phi,Q)$ of $G\times H$ one has $C_{G\times H}(\Delta(P,\phi,Q))=C_G(P)\times C_H(Q)$ and that, for Brauer pairs $(P,e)$ and $(Q,f)$ of $FG$ and $FH$, respectively, the pair $(\Delta(P,\phi,Q),e\otimes f^*)$ is a Brauer pair of $F[G\times H]\cong FG\otimes_F FH$.

\begin{theorem}\label{thm BP decomp}
Let $(P,e)$ be a Brauer pair of $FG$ and let $(R,d)$ be a Brauer pair of $FK$. Suppose that $\sigma\colon R\myiso P$ is an isomorphism and let $C_G(P)\le S\le N_G(P,e)$ and $C_K(R)\le T\le N_K(R,d)$ be  intermediate subgroups. Furthermore, let $\Omega:=\Omega_H((P,e), \sigma,(R,d))$ denote the set of triples $(\phi,(Q,f),\psi)$, where $(Q,f)$ is a Brauer pair of $FH$ and $\psi\colon R\myiso Q$ and $\phi\colon Q\myiso P$ are isomorphisms such that $\sigma=\phi\circ\psi$. Finally, let $M\in\ltriv{\calO G}_{\calO H}$ and $N\in\ltriv{\calO H}_{\calO K}$ be p-permutation modules all of whose indecomposable direct summands have twisted diagonal vertices. 

\smallskip
{\rm (a)} The group $N_{G\times K}(\Delta(P,\sigma,R))\times H$ acts on $\Omega$ via $\lexp{((g,k),h)}{(\phi,(Q,f),\psi)}= (c_g \phi c_h^{-1},\lexp{h}{(Q,f)}, c_h\psi c_k^{-1})$ and $\stab_H(\phi,(Q,f),\psi)=C_H(Q)$.

\smallskip
{\rm (b)} Let $\Omegatilde\subseteq\Omega$ be a set of representatives of the $H$-orbits of $\Omega$. One has an isomorphism. The restrictions $\Phi_{(\phi,(Q,f),\psi)}$ of $\Phi_{(\phi,Q,\psi)}$ in (\ref{eqn BD homomorphism}) to $eM(\Delta(P,\phi,Q))f\otimes_{F[C_H(Q)]}fN(\Delta(Q,\psi,R))d$ define an isomorphism
\begin{equation}\label{eqn BP1}
 \Phi\colon\ \bigoplus_{(\phi,(Q,f),\psi)\in\Omegatilde} eM(\Delta(P,\phi,Q))f\otimes_{F[C_H(Q)]}fN(\Delta(Q,\psi,R))d\ \myiso\ 
  e(M\otimes_{\calO H} N)(\Delta(P,\sigma,R))d
\end{equation}
of $(F[C_G(P)]e,F[C_K(R)]d)$-bimodules, which is natural in $M$ and $N$.

\smallskip
{\rm (c)} Let $\Omegahat\subseteq\Omegatilde\subseteq \Omega$ be a set of representatives of the $N_{S\times T}(\Delta(P,\sigma,R))\times H$-orbits of $\Omega$. Then the homomorphisms $\Phi_{(\phi, (Q,f),\psi)}\colon eM(\Delta(P,\phi,Q))f\otimes_{F[C_H(Q)]}fN(\Delta(Q,\psi,R))d \to
  e(M\otimes_{\calO H} N)(\Delta(P,\sigma,R))d$, $(\phi,(Q,f),\psi)\in\Omegahat$, induce an isomorphism
\begin{equation}\label{eqn BP2}
  \Phi\colon\
  \bigoplus_{\omega\in\Omegahat}
  \Ind_{X(\omega)*Y(\omega)}^{N_{S\times T}(\Delta(P,\sigma,R))} 
  \bigl(M(\omega)\tens{X(\omega)}{Y(\omega)}{F H} N(\omega)\bigr)\  
  \myiso\ e\bigl(M\otimes_{\calO H} N\bigr)\bigl(\Delta(P,\sigma,R)\bigr)d
\end{equation}
of $N_{S\times T}(\Delta(P,\sigma,R))$-modules which is natural in $M$ and $N$. Here, for $\omega=(\phi,(Q,f),\psi)\in \Omegahat$, we set $X(\omega):=N_{S\times H}(\Delta(P,\phi,Q),e\otimes f^*)$, $Y(\omega):=N_{H\times T}(\Delta(Q,\psi,R), f\otimes d^*)$, $M(\omega):=eM(\Delta(P,\phi,Q))f\in\ltriv{FX(\omega)}$, $N(\omega):=fN(\Delta(Q,\psi,R))d\in\ltriv{FY(\omega)}$.

\smallskip
{\rm (d)} Assume that $G=K$, $(P,e)=(R,d)$, $S=T$, and $\sigma=\id_P$. Let $\Lambda:=\Lambda_H(P)$ be the set of pairs $(\phi,(Q,f))$, where $(Q,f)$ is a Brauer pair of $FH$ and $\phi\colon Q\myiso P$ is an isomorphism. The group $S\times H$ acts on $\Lambda$ via $\lexp{(g,h)}{(\phi,(Q,f))}=(c_g\phi c_h^{-1},\lexp{h}{(Q,f)})$, for $(g,h)\in S\times H$ and $(\phi,(Q,f))\in\Lambda$. Let $\Lambdahat\subseteq\Lambdatilde\subseteq\Lambda$ be such that $\Lambdatilde$ (resp.~$\Lambdahat$) is a set of representatives of the $H$-orbits (resp.~$S\times H$-orbits) of  $\Lambda$. Then one has an isomorphism
\begin{equation}\label{eqn BP1'}
  e(M\otimes_{\calO H} N)(\Delta(P))e\ \cong\ 
  \bigoplus_{(\phi,(Q,f))\in\Lambdatilde} eM(\Delta(P,\phi,Q))f\otimes_{F[C_H(Q)]}fN(\Delta(Q,\phi^{-1},P))e 
\end{equation}
of $(F[C_G(P)]e,F[C_G(P)]e)$-bimodules, and an isomorphism
\begin{equation}\label{eqn BP2'}
  e(M\otimes_{\calO H} N)(\Delta(P))e\ \cong\
  \bigoplus_{\lambda\in\Lambdahat}
  \Ind_{\Delta(I(\lambda))(C_G(P)\times\{1\})}^{N_{S\times S}(\Delta(P))} 
  \bigl(M(\lambda)\tens{X(\lambda)}{Y(\lambda)}{F H} N(\lambda)\bigr)
\end{equation}
of $F[N_{S\times S}(\Delta(P))](e\otimes e^*)$-modules. Here, for $\lambda=(\phi,(Q,f))\in\Lambda$, we set $X(\lambda):=N_{S\times H}(\Delta(P,\phi,Q),e\otimes f^*)$, $Y(\lambda):=N_{H\times S}(\Delta(Q,\phi^{-1},P),f\otimes e^*)$, $M(\lambda):=eM(\Delta(P,\phi,Q))f\in \ltriv{FX(\lambda)}$, $N(\lambda):=fN(\Delta(Q,\phi^{-1},P)e\in\ltriv{FY(\lambda)}$, and $I(\lambda):=N_{(S,\phi,N_H(Q,f))}$.
\end{theorem}

\begin{proof}
(a) This follows immediately from the definitions.

\smallskip
(b) Multiplication of the isomorphism in (\ref{eqn BD}) with $e$ from the left and $d$ from the right yields again an isomorphism of $(F[C_G(P)]e,F[C_K(R)]d)$-bimodules. For $(\phi,Q,\psi)\in\Gammatilde$, the corresponding summand on the left hand side of this isomorphism can be written as the obvious direct sum over the block idempotents $f$ of $F[C_H(Q)]$. Moreover, if $(\phi,Q,\psi)$ runs through $\Gammatilde$ and, for each such $(\phi,Q,\psi)$, $f$ runs through the block idempotents of $FC_H(Q)$, then $(\phi, (Q,f),\psi)$ runs through a set of representatives of the $H$-orbits of $\Omega$. This proves the claim for this particular set of representatives $\Omegatilde$ derived from $\Gammatilde$. If $\widetilde{\Omegatilde}$ is an arbitrary set of representatives of the $H$-orbits of $\Omega$, then for each $(\phi,(Q,f),\psi)\in\Omegatilde$ there exists a unique $(\phi',(Q',f'),\psi')\in\widetilde{\Omegatilde}$ and an element $h\in H$ such that $\lexp{h}{(\phi, (Q,f),\psi)}=(\phi', (Q',f'), \psi')$. One obtains a well-defined isomorphism 
\begin{equation*}
   \zeta_h\colon eM(\Delta(P,\phi,Q))f\otimes_{F[C_H(Q)]}fN(\Delta(Q,\psi,R))d \myiso
   eM(\Delta(P,\phi',Q'))f'\otimes_{F[C_H(Q')]}f'N(\Delta(Q',\psi',R))d
\end{equation*}
by mapping $\Br_{\Delta(P,\phi,Q)}(m)\otimes \Br_{\Delta(Q,\psi,R)}(n)$ to $\Br_{\Delta(P,\phi',Q')}(mh^{-1})\otimes \Br_{\Delta(Q',\psi',R)}(hn)$, which is independent of the choice of $h$, such that $\Phi_{(\phi', (Q',f'),\psi')}\circ\zeta_h=\Phi_{(\phi,(Q,f),\psi)}$. This implies the statement in Part~(b).

\smallskip
(c) Note that the right hand side of (\ref{eqn BP1}) is an $F[N_{S\times T}(\Delta(P,\sigma,R))]$-module in a natural way.  
For $\omega=(\phi,(Q,f),\psi)\in\Omega$, we set $X'(\omega):=\Delta(P,\phi,Q)$ and $Y'(\omega):=\Delta(Q,\psi,R)$. Let $(g,k)\in N_{S\times T}(\Delta(P,\sigma,R))$ and $\omega=(\phi,(Q,f),\psi)\in\Omegatilde$. Then there exists a unique $\omega'=(\phi', (Q',f'),\psi')\in \Omegatilde$ and an element $h\in H$ such that $\lexp{((g,k),h)}{\omega}=:\omega'\in\Omegatilde$. One verifies as in the proof of Corollary~\ref{cor BD corollary}(b) that
\begin{equation}\label{eqn Phi}
  (g,k)\Phi_{\omega}\bigl(e\Br^M_{X'(\omega)}(m)f\otimes f\Br^N_{Y'(\omega)}(n)d\bigr)) = 
  \Phi_{\omega'}\bigl(e\Br^M_{X'(\omega')}(gmh^{-1})f' \otimes f'\Br^N_{X'(\omega')}(hnk^{-1})d\bigr)\,,
\end{equation} 
for $m\in M^{X'(\omega)}$ and $n\in N^{Y'(\omega)}$. This implies that if one transports the $FN_{S\times T}(\Delta(P,\sigma,R))$-module structure of the right hand side of (\ref{eqn BP1}) via $\Phi^{-1}$ to the left hand side then $N_{S\times T}(\Delta(P,\sigma,R))$ permutes the components of the left hand side according to its action on the $H$-orbits of $\Omega$. Moreover, it is straightforward to verify that $(g,k)\in N_{S\times T}(\Delta(P,\sigma,R))$ stabilizes the $H$-orbit of $\omega=(\phi,(Q,f),\psi)$ if and only if $(g,k)\in X(\omega)*Y(\omega)$. Equation~(\ref{eqn Phi}) also implies that, for $\omega\in\Omegatilde$, the $F[X(\omega)*Y(\omega)]$-module structure of the $\omega$-component of the left hand side of (\ref{eqn BP1}) coincides with the extended tensor product structure introduced in \ref{noth ext tp and hom}(a). 

\smallskip
(d) Consider the bijection $\alpha\colon \Omega=\Omega_H((P,e),\id_P,(P,e))\to\Lambda_H(P)$, $(\phi,(Q,f),\phi^{-1})\mapsto (\phi,(Q,f))$ and the group homomorphism $\kappa\colon N_{S\times S}(\Delta(P))\times H\mapsto S\times H$, $((s_1,s_2),h)\mapsto (s_1,h)$. Recall from Proposition~\ref{prop Nphi}(a) that $N_{S\times S}(\Delta(P))=\Delta(S)\cdot(\{1\}\times C_G(P))$ and note that $\{1\}\times C_G(P)$ acts trivially on $\Omega$. Therefore, one has
\begin{equation*}
  \alpha(\lexp{((s_1,s_2),h)}{\omega}) = (c_{s_1}\phi c_h^{-1},\lexp{h}{(Q,f)}) = 
  \lexp{\kappa((s_1,s_2),h)}{\alpha(\omega)}\,,
\end{equation*}
for $\omega=(\phi,(Q,f),\phi^{-1})\in\Omega$ and $((s_1,s_2),h)\in N_{S\times S}(\Delta(P))\times H$. Thus, the bijection $\alpha$ maps $\Omegatilde$ (resp.~$\Omegahat$) to a set of representatives of the $H$-orbits (resp.~$S\times H$-orbits) of $\Lambda$. Now, the isomorphisms in (\ref{eqn BP1'}) and (\ref{eqn BP2'}) are immediate consequences of the isomorphisms in (\ref{eqn BP1}) and (\ref{eqn BP2}), after noting that $X(\lambda)*Y(\lambda)=\Delta(I(\lambda))(C_G(P)\times\{1\})$, since $Y(\lambda)=X(\lambda)^\circ$, see Proposition~\ref{prop Nphi}(c) and Lemma~\ref{lem comp form}(a).
\end{proof}

\section{Character groups and perfect isometries}\label{sec character groups}

Throughout this section, $G$, $H$, $K$ denote finite groups. We assume that the $p$-modular system $(\KK,\calO,F)$ is large enough for $G$, $H$, $K$, and the groups $H_1,\ldots, H_n$ appearing in Lemma~\ref{lem isometry separation 1} and Corollary~\ref{cor isometry separation 2}. In this section, we recall and introduce notation, concepts, and basic results related to character groups and perfect isometries and we prove some results on perfect isometries that will be used in later sections.

For more details on the character group concepts of this section we refer the reader to \cite[Section~3.6]{NagaoTsushima1989}.

\begin{notation}\label{not character groups}
(a) Let $A$ be a finite-dimensional algebra over a field $\kk$. Recall that the {\em Grothendieck group} $R(A)$, with respect to short exact sequences, is a free abelian group with $\ZZ$-basis given by elements $[S]$, where $S$ runs through a set of representatives of the isomorphism classes of simple left $A$-modules. For any $M\in\lmod{A}$ one sets $[M]:=[S_1]+\cdots+[S_n]$, where $S_1,\ldots,S_n$ are the composition factors of $M$, (repeated according to their multiplicities). If also $B$ is a finite-dimensional algebra over the same field then we set $R(A,B):=R(A\otimes B^\circ)$, where $B^\circ$ denotes the opposite algebra of $B$. This notation is motivated by the canonical category isomorphism $\lmod{A}_B\cong \lmod{A\otimes B^\circ}$. Thus, each $M\in\lmod{A}_B$ defines an element $[M]\in R(A,B)$. If $B$ is a group algebra $\kk H$, then we always identify $(\kk H)^\circ$ with $\kk H$ using the isomorphism $h^\circ\mapsto h^{-1}$. If additionally $A=\kk G$ is a group algebra, we consequently identify $\kk G\otimes (\kk H)^\circ$ with $\kk[G\times H]$ via $g\otimes h^\circ\mapsto (g,h^{-1})$.

\smallskip
(b) For an idempotent $e\in Z(\KK G)$ we identify $R(\KK Ge)$ with the virtual character group of $\KK Ge$, the free $\ZZ$-span of the irreducible characters $\Irr(\KK G e)$ of $\KK Ge$. This way, $R(\KK Ge)\subseteq R(\KK G)$. Similarly, if $e$ is an idempotent in $Z(FG)$ we identify $R(FGe)$ with the group of virtual Brauer characters belonging to $FGe$. This way, $R(FGe)\subseteq R(FG)$. For convenience, we view Brauer characters throughout as class functions on $G$ (rather than on $G_{p'}$) with values in $\KK$ that vanish on $G\smallsetminus G_{p'}$. Here, $G_{p'}$ denotes the set of {\em $p'$-elements} of $G$, i.e., elements whose order is not divisible by $p$. By scalar extension from $\ZZ$ to $\KK$ we view these Grothendieck groups also as embedded into $\KK$-vector spaces, denoted by $\KK R(\KK G)$, etc., and we identify $\KK R(\KK G)$ with the $\KK$-vector space of $\KK$-valued class functions on $G$, or by linear extension also as subspace of the space of $\KK$-linear functions from $\KK G$ to $\KK$. In particular, we identify $\KK R(FG)$ with the space of $\KK$-valued class functions on $G$ which vanish on $G\smallsetminus G_{p'}$. Note that for any idempotent $e\in Z(\calO G)$ one has $\KK R(FG\ebar)\subseteq \KK R(\KK Ge)$ as $\KK$-vector spaces of function on $G$. In fact, since the determinant of the Cartan matrix of $FG\ebar$ is non-zero, each irreducible Brauer character in $FG\ebar$ is a $\QQ$-linear combination of projective indecomposable characters of $\KK Ge$.

If $e\in Z(\KK G)$ and $f\in Z(\KK H)$ are idempotents then, with the convention in (a), one has a group $R(\KK G e, \KK H f):=R(\KK[G\times H](e\otimes f^*))$. Multiplication of $(\KK G,\KK H)$-bimodules with $e$ from the left and $f$ from the right induces a projection map $R(\KK G, \KK H)\to R(\KK Ge,\KK Hf)$ which we denote by $\mu\mapsto e\mu f$. Taking $\KK$-duals defines a map $R(\KK G e,\KK H f)\to R(\KK Hf,\KK Ge)$, $\mu\mapsto\mu^\circ$, for idempotents $e\in Z(\KK G)$ and $f\in Z(\KK H)$. Similar notations apply with $\KK$ replaced by $F$.

\smallskip
(c) Tensor products of bimodules induce bilinear maps
\begin{equation*}
  R(\KK G, \KK H)\times R(\KK H, \KK K) \to R(\KK G, \KK K)\,, \quad (\mu,\nu)\mapsto \mu\cdotH\nu\,,
\end{equation*}
and extended tensor products (see \ref{noth ext tp and hom}(a)) induce bilinear maps
\begin{equation*}
  R(\KK X)\times R(\KK Y)\to R(\KK [X*Y])\,,\quad (\mu,\nu)\mapsto \mu\dotXYH \nu\,,
\end{equation*}
for $X\le G\times H$ and $Y\le H\times K$. Each $\mu\in R(\KK G, \KK H)$, induces a group homomorphism
\begin{equation*}
  I_\mu\colon R(\KK H)\to R(\KK G)\,, \quad \psi\mapsto \mu\cdotH\psi\,,
\end{equation*}
using the special case $K=\{1\}$ from the beginning of Part~(c). Note that similar constructions do not work for (bi-)modules over $F$, since an $(FG,FH)$-bimodule is not necessarily flat as right $FH$-module.

\smallskip
(d) If $e$ is an idempotent in $Z(\calO G)$ then the {\em decomposition map} $d^e_G\colon R(\KK G)\to R(FG\ebar)\subseteq \KK R(\KK G e)$ is given by
\begin{equation*}
  (d_G^e(\chi))(g) = \begin{cases} \chi(ge), & \text{if $g\in G_{p'}$,}\\ 0, & \text{otherwise,}\end{cases}
\end{equation*}
for $\chi\in R(\KK G)$ and $g\in G$. If $e=1$, one obtains the usual decomposition map $d_G\colon R(\KK G)\to R(FG)$.

More generally, for a $p$-element $u\in G$ and an idempotent $e\in Z(\calO[C_G(u)])$, the {\em generalized decomposition map} $d_G^{(u,e)}\colon \KK R(\KK G) \to \KK R(FC_G(u)\ebar)$ is given by
\begin{equation*}
  \bigl(d_G^{(u,e)}(\chi)\bigr)(g)= 
  \begin{cases} \chi(uge)\,, & \text{if $g\in C_G(u)_{p'}$,} \\ 0\,, & \text{otherwise,}\end{cases}
\end{equation*}
for $\chi\in \KK R(\KK G)$ and $g\in C_G(u)$. If $\chi\in\Irr(\KK G)$ belongs to a sum of blocks $A$ of $\calO G$ and $e$ is a primitive idempotent in $Z(\calO[C_G(u)])$ then Brauer's second main theorem (see \cite[Theorem 5.4.2]{NagaoTsushima1989}) implies that $d_G^{(u,e)}(\chi)=0$ unless $(\langle u\rangle,e)$ is an $A$-Brauer pair. For $u=1$ one recovers the decomposition map $d^e_G$ from above.
\end{notation}

\begin{remark}
(a) Let $M\in\lmod{\KK G}_{\KK H}$, $N\in\lmod{\KK H}_{\KK K}$, and let $\mu\in R(\KK G,\KK H)$ and $\nu\in R(\KK H,\KK K)$ denote their respective characters as left modules for $\KK[G\times H]$ and $\KK[H\times K]$. Then the character $\mu\cdotH \nu\in R(\KK G,\KK K)$ of $M\otimes _{\KK H} N$ viewed as left $\KK[G\times K]$-module is given by
\begin{equation}\label{eqn char tens 1}
  (\mu\cdotH \nu)(g,k)=\frac{1}{|H|} \sum_{h\in H}\mu(g,h)\nu(h,k)\,,
\end{equation}
for $(g,k)\in G\times K$. In the special case that $M=V\otimes_{\KK} W$ and $N=W'\otimes_{\KK} U$ with irreducible modules $V\in\lmod{\KK G}$, $W,W'\in\lmod{\KK H}$,  and $U\in\lmod{\KK K}$, one has $M\otimes_{\KK H}N\cong V\otimes U\in\lmod{\KK[G\times K]}$ if $W^\circ\cong W'$, and $M\otimes_{\KK H}N=\{0\}$ if $W^\circ\not\cong W'$. Thus,
\begin{equation}\label{eqn char tens 2}
  (\chi_V\times \chi_W)\cdotH(\chi_{W'}\times\chi_U) =
  \begin{cases} \chi_V\times\chi_U\,, & \text{if $\chi_W^\circ=\chi_{W'}$,} \\ 0\,, & \text{otherwise.} \end{cases} 
\end{equation}

\smallskip
(b) If $e\in Z(\KK G)$ is an idempotent then the character of $\KK Ge$, viewed as element in $R(\KK Ge,\KK Ge)\subseteq R(\KK[G\times G])$ is given by $\sum_{\chi\in\Irr(\KK Ge)} \chi\times\chi^\circ$.
\end{remark}

Parts~(a) and (c) of the following definition are due to Brou\'e, see \cite{Broue1990}.

\begin{definition}\label{def perfect isometry}
Let $\mu\in R(\KK G,\KK H)$.

\smallskip
(a) The virtual character $\mu$ is called {\em perfect} if it satisfies the following two conditions:

\smallskip
\quad\quad (i) For all $(g,h)\in G\times H$, one has $\mu(g,h)\in|C_G(g)|\calO \cap |C_H(h)|\calO$\,.

\smallskip
\quad\quad (ii) If $(g,h)\in G\times H$ is such that $\mu(g,h)\neq 0$, then $g$ is a $p'$-element if and only if $h$ is a $p'$-element.

\smallskip
(b) We call the virtual character $\mu$ {\em quasi-perfect} if it satisfies condition (ii) in Part~(a).

\smallskip
(c) Assume that $e\in Z(\KK G)$ and $f\in Z(\KK H)$ are idempotents and that $\mu\in R(\KK Ge, \KK H f)$. One calls $\mu$ an {\em isometry} between $\KK Ge$ and $\KK Hf$ if the map $I_\mu\colon R(\KK H f)\to R(\KK Ge)$ is bijective and satisfies $(I_\mu(\psi),I_\mu(\psi'))_G=(\psi,\psi')_H$, for all $\psi,\psi'\in R(\KK Hf)$. If additionally $\mu$ is perfect, then $\mu$ is called a {\em perfect isometry} between $\KK Ge$ and $\KK Hf$.
\end{definition}

Brou\'e's abelian defect group conjecture in its weakest form states that  if $\calO Ge$ is a block of $\calO G$ ($e$ its block idempotent) with abelian defect group $D$ and $\calO[N_G(D)]f$ the block of $\calO[N_G(D)]$ which is in Brauer correspondence with $\calO Ge$, i.e., $\br_D(e)=\fbar$, then there exists a perfect isometry $\mu\in R(\KK Ge,\KK N_G(D)f)$.

\begin{remark}\label{rem isometry}
Let $e\in Z(\KK G)$ and $f\in Z(\KK H)$ be idempotents and let $\mu\in R(\KK Ge, \KK Hf)$. The first two of the following statements are quick consequences of (\ref{eqn char tens 2}).

\smallskip
(a) The following are equivalent:

\smallskip
(i) $\mu$ is an isometry between $\KK Ge$ and $\KK Hf$.

\smallskip
(ii) $\mu\cdotH\mu^\circ=[\KK Ge]\in R(\KK Ge,\KK Ge)$ and $\mu^\circ\cdotG \mu = [\KK Hf]\in R(\KK Hf,\KK Hf)$.

\smallskip
(iii) There exists a bijection $\Irr(\KK Hf)\myiso\Irr(\KK Ge)$, $\psi\mapsto\chi_\psi$, and elements $\varepsilon_{\psi}\in\{\pm1\}$, for $\psi\in\Irr(\KK Hf)$, such that $\mu=\sum_{\psi\in\Irr(\KK Hf)} \varepsilon_\psi\cdot \chi_\psi\times \psi^\circ$.

\smallskip
(b) One has $\mu\neq 0$ if and only if $\mu\cdotH\mu^\circ\neq 0$ in $R(\KK Ge,\KK Ge)$.

\smallskip
(c) The elements in $R(\KK Ge,\KK Hf)$ that satisfy Condition~(i) (resp.~Condition~(ii)) in Definition~\ref{def perfect isometry}(a) form a subgroup of $R(\KK Ge,\KK Hf)$.

\smallskip
(d) If $e\in Z(\calO G)$, $f\in Z(\calO H)$, and $\mu$ is the character of an indecomposable module $M\in\ltriv{\calO Ge}_{\calO Hf}$ with twisted diagonal vertex then $\mu$ satisfies Conditions~(i) and (ii) in Definition~\ref{def perfect isometry}(a); see \cite[Proposition~1.2]{Broue1990}.

\smallskip
(e) If $\mu$ is quasi-perfect then the $\KK$-linear extension $\KK R(\KK Hf)\to\KK R(\KK Ge)$ of $I_\mu$ restricts to a map $\KK R(FHf)\to\KK R(FGe)$. In fact, this follows immediately from the formula in (\ref{eqn char tens 1}) in the special case that $K=\{1\}$.
\end{remark}

\begin{proposition}\label{prop quasi-perfect}
Let $\mu\in R(\KK Ge,\KK Hf)$. The following are equivalent:

\smallskip
{\rm (i)} The virtual character $\mu$ is quasi-perfect.

\smallskip
{\rm (ii)} One has $d_G\circ I_\mu = I_\mu\circ d_H$ as maps $\KK R(\KK H)\to \KK R(\KK G)$.

\smallskip
{\rm (iii)} One has $d_H\circ I_{\mu^\circ} = I_{\mu^\circ}\circ d_G$ as maps $\KK R(\KK G) \to \KK R(\KK H)$.
\end{proposition}

\begin{proof}
Clearly, $\mu$ is quasi-perfect if and only if $\mu^\circ$ is quasi-perfect. Thus it suffices to show the equivalence of (i) and (ii). 

Assume first that (i) holds and let $\psi\in\KK R(\KK H)$ and $g\in G$. Consider the case that $g\notin G_{p'}$. By Equation~(\ref{eqn char tens 1}), we have $I_\mu(d_H(\psi))(g)=|H|^{-1}\sum_{h\in H} \mu(g,h)(d_H(\psi))(h)$. By our assumptions, $\mu(g,h)=0$ for every $h\in H_{p'}$. On the other hand, if $h\in H\smallsetminus H_{p'}$ then $(d_H(\psi))(h)=0$. Thus, we obtain $I_\mu(d_H(\psi))(g)=0=d_G(I_\mu(\psi))(g)$. Now consider the case that $g\in G_{p'}$.
Then, again by (\ref{eqn char tens 1}), we have
\begin{equation*}
  |H|\cdot d_G(I_\mu(\psi))(g)= |H|\cdot I_\mu(\psi)(g) = \sum_{h\in H} \mu(g,h)\psi(h)= \sum_{h\in H_{p'}} \mu(g,h)\psi(h)
  = |H|\cdot I_\mu(d_H(\psi))(g)\,,
\end{equation*}
and (ii) holds.

Now assume that (ii) holds. Let $(g,h)\in G\times H$\ and assume that $\mu(g,h)\neq 0$. Let $\psi\in\KK R(\KK H)$ denote the characteristic function on the conjugacy class of $h$. If $g\in G_{p'}$ and $h\notin H_{p'}$ then $d_H(\psi)=0$ and (\ref{eqn char tens 1}) implies the contradiction $0=I_\mu(d_H(\psi))(g)=d_G(I_\mu(\psi))(g) = I_\mu(\psi)(g) = |C_H(h)|^{-1}\cdot \mu(g,h)\neq 0$. And if $g\notin G_{p'}$ and $h\in H_{p'}$ then we obtain the contradiction $0=d_G(I_\mu(\psi))(g)=I_\mu(d_H(\psi))(g) = |C_H(h)|^{-1}\cdot\mu(g,h)\neq 0$. Thus, (i) holds.
\end{proof}

\begin{lemma}\label{lem isometry separation 1}
Let $H_1,\ldots,H_n$ be finite groups and let $\mu_i\in R(\KK G, \KK H_i)$, $i=1,\ldots,n$, be virtual characters such that $\sum_{i=1}^n \mu_i\cdotHi \mu_i^\circ=\sum_{\chi\in\Omega}\chi\times\chi^\circ$, for some subset $\Omega\subseteq\Irr(\KK G)$. Then $\Omega$ is the disjoint union of subsets $\Omega_i$, $i=1,\ldots,n$, with the property that $\mu_i\cdotHi\mu_i^\circ=\sum_{\chi\in\Omega_i}\chi\times\chi^\circ$.
\end{lemma}

\begin{proof}
For each $i=1,\ldots,n$ we write $\mu_i=\sum_{\chi\in\Irr(\KK G)} \chi\times \psi_{i,\chi}$ with $\psi_{i,\chi}\in R(\KK H_i)$. Equation~(\ref{eqn char tens 2}) implies that, for each $\chi\in\Irr(\KK G)$, the coefficient of $\chi\times\chi^\circ$ in $\sum_{i=1}^n\mu_i\cdotHi \mu_i^\circ$ is equal to $\sum_{i=1}^n(\psi_{i,\chi},\psi_{i,\chi})_{H_i}$. Thus, $\sum_{i=1}^n(\psi_{i,\chi},\psi_{i,\chi})_{H_i}$ is equal to $1$ if $\chi\in\Omega$ and equal to $0$, if $\chi\notin\Omega$. This implies $\psi_{i,\chi}=0$ for all $i=1,\ldots,n$ if $\chi\notin\Omega$. Moreover, if $\chi\in\Omega$, then there exists a unique $i\in\{1,\ldots,n\}$ such that $\psi_{i,\chi}\neq 0$. For $i=1,\ldots,n$ we define $\Omega_i$ as the set of those $\chi\in\Omega$ with $\psi_{i,\chi}\neq 0$. Now the lemma follows.
\end{proof}

%

\begin{lemma}\label{lem isometry crit}
Let $e$ be a block idempotent of $\calO G$, $f$ a block idempotent of $\calO H$ and suppose that $\mu\in R(\KK Ge,\KK Hf)$ is a quasi-perfect virtual character such that there exists a non-empty subset $\Omega$ of $\Irr(\KK Ge)$ with $\mu\cdotH\mu^\circ = \sum_{\chi\in \Omega} \chi\times \chi^\circ$. Then $\Omega=\Irr(\KK Ge)$ and $\mu$ is an isometry between $\KK Ge$ and $\KK Hf$.
\end{lemma}

\begin{proof}
After writing $\mu$ as a $\ZZ$-linear combination of the basis elements $\chi\times \psi^\circ$, $(\chi,\psi)\in\Irr(\KK Ge)\times \Irr(\KK Hf)$, and using Equation~(\ref{eqn char tens 2}), the hypothesis $\mu\cdotH\mu^\circ=\sum_{\chi\in\Omega}\chi\times\chi^\circ$ implies that there exists a subset $\Lambda$ of $\Irr(\KK Hf)$ and a bijection $\alpha\colon\Omega\to \Lambda$ such that $\mu=\sum_{\chi\in\Omega} \varepsilon_\chi\cdot \chi\times \alpha(\chi)^\circ$, with $\varepsilon_\chi\in\{\pm1\}$ for $\chi\in\Omega$. This implies $\mu^\circ\cdotG\mu=\sum_{\psi\in\Lambda}\psi\times\psi^\circ$. As $\mu$ is quasi-perfect, so is $\mu^\circ$. Thus, by symmetry, it suffices now to show that $\Omega=\Irr(\KK Ge)$.

For $\chi,\chi'\in\Irr(\KK Ge)$ set $m_{\chi,\chi'}:=(d_G(\chi),\chi')_G\in\KK$. Then $m_{\chi',\chi}=m_{\chi,\chi'}$ and if $\chi$ has height $0$ then $m_{\chi,\chi'}\neq 0$ for all $\chi'\in\Irr(\KK Ge)$, see~\cite[Lemma~3.6.34(ii)]{NagaoTsushima1989}. Thus, to complete the proof it suffices to show that if $\chi\in\Omega$ and $\chi'\in\Irr(\KK Ge)$ with $m_{\chi,\chi'}\neq 0$ then also $\chi'\in\Omega$. But this holds if and only if  $d_G(\Omega)\subseteq\langle\Omega\rangle_\KK$. So let $\chi\in\Omega$ and set $\psi:=\alpha(\chi)$. Then $I_\mu(\psi)=\varepsilon_\chi\cdot\chi$. Since $\mu$ is quasi-perfect, Proposition~\ref{prop quasi-perfect} implies that 
\begin{equation*}
  d_G(\chi)=\varepsilon_\chi\cdot d_G(I_\mu(\psi)) = \varepsilon_\chi\cdot I_\mu(d_H(\psi))\in I_\mu(\KK R(\KK Hf))
  \subseteq \langle \Omega \rangle_\KK\,,
\end{equation*}
and the proof is complete.
\end{proof}

\begin{corollary}\label{cor isometry separation 2}
Let $e$ be a block idempotent of $\calO G$, let $H_1,\ldots,H_n$ be finite groups, and, for each $i=1,\ldots,n$, let $f_i$ be a block idempotent of $\calO H_i$. Furthermore, for $i=1,\ldots,n$, let $\mu_i\in R(\KK Ge,\KK H_if_i)$ be a quasi-perfect virtual character such that $\sum_{i=1}^n\mu_i\cdotHi\mu_i^\circ=\sum_{\chi\in\Omega}\chi\times\chi^\circ$ in $R(\KK Ge,\KK Ge)$ for some non-empty subset $\Omega\subseteq\Irr(\KK Ge)$. Then there exists a unique $i\in\{1,\ldots,n\}$ such that $\mu_i\neq 0$. Moreover, $\Omega=\Irr(\KK Ge)$ and $\mu_i$ is an isometry between $\KK Ge$ and $\KK H_if_i$.
\end{corollary}

\begin{proof}
Applying Lemma~\ref{lem isometry separation 1}, we see that $\Omega$ is a disjoint union of subsets $\Omega_i$ such that $\mu_i\cdotHi\mu_i^\circ=\sum_{\chi\in\Omega_i}\chi\times\chi^\circ$. Note that $\mu_i\neq 0$ if and only if $\Omega_i\neq\emptyset$. Choose $i\in\{1,\ldots,n\}$ such that $\Omega_i$ is non-empty. Then Lemma~\ref{lem isometry crit} implies that $\Omega_i=\Irr(\KK Ge)$ and that $\mu_i$ is an isometry between $\KK Ge$ and $\KK H_if_i$. This also implies that $\Omega_j=\emptyset$ for all $j\neq i$ in $\{1,\ldots,n\}$ and therefore $\mu_j=0$ for all $j\neq i$ in $\{1,\ldots,n\}$.
\end{proof}

\begin{corollary}\label{cor isometry separation 3}
Let $e$ be a block idempotent of $\calO G$, let $f\in Z(\calO H)$ be an idempotent, and let $\mu\in R(\KK Ge,\KK Hf)$ be a quasi-perfect virtual character satisfying $\mu\cdotH\mu^\circ = [\KK Ge]$ in $R(\KK Ge, \KK Ge)$. Then there exists a unique primitive idempotent $f'$ of $Z(\calO Hf)$ such that $\mu=\mu\cdot f'$. Furthermore, $\mu$ is an isometry between $\KK Ge$ and $\KK Hf'$.
\end{corollary}

\begin{proof}
Let $f'_1,\ldots,f'_n$ denote the primitive idempotents of $Z(\calO Hf)$, and for each $i=1,\ldots,n$, set $\mu_i:=\mu\cdot f'_i\in R(\KK Ge, \KK Hf'_i)$. Then $\mu=\sum_{i=1}^n\mu_i$ and
\begin{equation*}
  \sum_{i=1}^n\mu_i\cdotH\mu_i^\circ = \mu\cdotH\mu^\circ=[\KK Ge]=\sum_{\chi\in\Irr(\KK Ge)}\chi\times\chi^\circ\,.
\end{equation*}
Proposition~\ref{prop quasi-perfect}(ii), together with the fact that $d_H$ respects the block decomposition, implies that with $\mu$ also $\mu_i$ is quasi-perfect. Now Corollary~\ref{cor isometry separation 2} applies and the proof is complete.
\end{proof}

The following corollary is an immediate consequence of Corollary~\ref{cor isometry separation 3}. It also slightly generalizes a result in \cite[Th\'eor\`eme~1.5(2)]{Broue1990}.

\begin{corollary}\label{prop quasi-perfect isometries respect work block-wise}
Let $e\in Z(\calO G)$ and $f\in Z(\calO H)$ be idempotents. Let $\calI$ denote the set of primitive idempotents $e'$ of $Z(\calO Ge)$ with $e'e=e'$ and let $\calJ$ denote the set of primitive idempotents $f'$ of $Z(\calO Hf)$ with $f'f=f'$. Suppose that $\mu\in R(\KK Ge,\KK Hf)$ is a quasi-perfect isometry between $\KK Ge$ and $\KK Hf$. Then, for each $e'\in\calI$ there exists a unique $f'\in\calJ$ such that $e'\mu f'\neq 0$. Conversely, for each $f'\in\calJ$ there exists a unique $e'\in\calI$ such that $e'\mu f'\neq 0$. These conditions define inverse bijections between $\calI$ and $\calJ$. Moreover, if $e'\in\calI$ and $f'\in\calJ$ satisfy $e'\mu f'\neq 0$ then $e'\mu f'$ is an isometry between $\KK Ge'$ and $\KK Hf'$.
\end{corollary}

\begin{proof}
First note that by Proposition~\ref{prop quasi-perfect}(ii), with $\mu$ also $e'\mu f'$ is a quasi-perfect character for every $e'\in\calI$ and $f'\in \calJ$. Next let $e'\in\calI$. Then, Corollary~\ref{cor isometry separation 3} applied to $e'$ and $f$ and the quasi-perfect virtual characters $e'\mu f'$, $f'\in\calJ$, implies that there exists a unique $f'\in \calJ$ with $e'\mu f'\neq 0$ and that $e'\mu f'$ is an isometry between $\KK Ge'$ and $\KK Hf'$. This proves the first statement. Symmetrically, fixing $f'\in \calJ$ and using $\mu^\circ$, we obtain the second statement. The remaining statements are clear from the above.
\end{proof}

The following Lemma will be used in Section~\ref{sec Brauer pairs of ppeqs}.

\begin{lemma}\label{lem non-zero criterion}
Let $e\in Z(\KK G)$ and $f\in Z(\KK H)$ be idempotents and let $\mu\in R(\KK Ge,\KK Hf)$ be quasi-perfect such that $d_G\circ I_\mu\colon R(\KK Hf)\to \KK R(\KK Ge)$ is non-zero. Then $d_{G\times H}(\mu)\neq 0$. In particular, if $e\neq 0\neq f$ and $\mu$ is a quasi-perfect isometry between $\KK Ge$ and $\KK Hf$, then $d_{G\times H}(\mu)\neq 0$.
\end{lemma}

\begin{proof}
Since $\mu$ is quasi-perfect, we have $d_G\circ I_\mu = I_\mu\circ d_H$ as maps from $\KK R(\KK Hf)$ to $\KK R(\KK Ge)$ by Lemma~\ref{prop quasi-perfect}. Since $\mu$ is quasi-perfect, it follows from Equation~(\ref{eqn char tens 1}) and the definition of quasi-perfect that $I_\mu\circ d_H = I_{d_{G\times H}(\mu)}\circ d_H$ as functions from $\KK R(\KK Hf)$ to $\KK R(\KK Ge)$. Thus, we have $I_{d_{G\times H}(\mu)}\circ d_H = d_G\circ I_\mu\neq 0$ and hence $d_{G\times H}(\mu)\neq 0$.
\end{proof}

The following lemma is an immediate consequence of the definition of $\tens{X}{Y}{\KK H}$ and will be used in Section~\ref{sec Brauer pairs of ppeqs}.

\begin{lemma}\label{lem non-zero res}
Let $X\le G\times H$ and $Y\le H\times K$ be subgroups satisfying $k_1(Y)\le k_2(X)$. Further let $\mu\in R(\KK X)$ and $\nu\in R(\KK Y)$ be such that $\mu\dotXYH\nu\in R(\KK[X*Y])$ is equal to $[M]$ or to $-[M]$ for some non-zero module $M\in\lmod{\KK[X*Y]}$. Then $\res^{Y}_{k_1(Y)\times k_2(Y)}(\nu)\neq 0$.
\end{lemma}

\begin{proof}
Clearly, $\res^{X}_{k_1(X)\times k_1(Y)}(\mu)\mathop{\cdot}\limits_{k_1(Y)}\res^{Y}_{k_1(Y)\times k_2(Y)}(\nu) = \res^{X*Y}_{k_1(X)\times k_2(Y)}(\mu\dotXYH\nu)$ and the latter is equal to $\res^{X*Y}_{k_1(X)\times k_2(Y)}([M])$ or its negative, and therefore non-zero. The result now follows.
\end{proof}

\section{Grothendieck groups of $p$-permutation modules and $p$-permutation equivalences}\label{sec p-permutation equivalences}

We assume again that $G$ and $H$ are finite groups and that the $p$-modular system $(\KK,\calO,F)$ is large enough for $G$ and $H$.

 \begin{nothing}\label{noth T(G)} {\em Grothendieck groups of $p$-permutation modules.}\quad
 (a) For an idempotent $e\in Z(\calO G)$, we denote by $T(\calO Ge)$ the Grothendieck group of the category $\ltriv{\calO Ge}$ with respect to direct sums. The group $T(\calO Ge)$ is free as abelian group with {\em standard basis} given by the elements $[M]$, where $M$ runs through a set of representatives of the isomorphism classes of indecomposable $p$-permutation $\calO Ge$-modules. For an arbitrary module $M\in\ltriv{\calO Ge}$ we write $[M]=[M_1]+\cdots+[M_r]\in T(\calO Ge)$ if $M=M_1\oplus\cdots\oplus M_r$ is a decomposition of $M$ into indecomposable submodules. We always view $T(\calO Ge)$ as a subgroup of $T(\calO G)$ in the natural way. Moreover, we say that an indecomposable module $M\in\ltriv{\calO G}$ {\em appears} in an element $\omega\in T(\calO G)$, if $[M]$ occurs with non-zero coefficient in $\omega$ with respect to the above standard basis. Note that multiplying a $p$-permutation $\calO G$-modules with $e$ defines a projection map $T(\calO G)\to T(\calO Ge)$, $\omega\mapsto e\omega$. Similarly, we define the Grothendieck group $T(FGe)$. If additionally $f\in Z(\calO H)$ is an idempotent then we define $T(\calO Ge,\calO Hf):=T(\calO[G\times H](e\otimes f^*))$. If $M\in\ltriv{\calO Ge}_{\calO Hf}$ we denote by $[M]$ the corresponding element in $T(\calO Ge,\calO Hf)$. Similar notations will be used over $F$. The $\ZZ$-span of the elements $[M]\in T(\calO Ge)$, where $M$ is an indecomposable projective $\calO Ge$-module will be denoted by $Pr(\calO G e)$. We also use the notations $Pr(\calO G e, \calO Hf)$, $Pr(FGe)$ and $Pr(FGe, FHf)$ with obvious meanings.
 
 \smallskip
(b) Tensor products of bimodules and generalized tensor products as introduced in Section~\ref{sec tensor products} induce maps on Grothendieck group levels that we denote again by $\cdotH$ and $\dotXYH$, as in \ref{not character groups}. Similarly, the Brauer construction with respect to a $p$-subgroup $P$ of $G$ induces a homomorphism $T(\calO G)\to T(F[N_G(P)/P])$, $\omega\mapsto \omega(P)$. Often we will also consider $\omega(P)$ as element of $T(F[N_G(P)])$ after applying inflation. Note that for a Brauer pair $(P,e)$ of $FG$, one obtains a homomorphism $-(P,e)\colon T(FG)\to T(FIe)$, $\omega\mapsto \omega(P,e)=e\omega(P)$, where $I=N_G(P,e)$. Similarly, one obtains a homomorphism $-(P,e)\colon T(\calO G)\to T(FIe)$. 

 \smallskip
 (c) For each idempotent $e\in Z(\calO G)$ we have a commutative diagram
 \begin{diagram}[70]
  \movevertex(-20,0){Pr(\calO Ge)} & \movevertex(-10,0){\quad \subseteq\quad}  &   T(\calO Ge) & 
               \movevertex(10,0){\Ear[40]{\kappa_G}} & \movevertex(20,0){R(\KK Ge)} &&
   \movevertex(-10,0){\saR{\wr}} & & \saR{\wr} & & \movevertex(20,0){\saR{d_G}} &&
   \movevertex(-20,0){Pr(FGe)} & \movevertex(-10,0){\quad \subseteq \quad} &T(FGe) & 
               \movevertex(10,0){\Ear[40]{\eta_G}} & \movevertex(20,0){R(FGe)} &&
\end{diagram}
whose top horizontal map $\kappa_G$ is induced by the scalar extension functors $\KK\otimes_{\calO}-$, whose left vertical maps are induced by the scalar extension functor $F\otimes_{\calO}-$, whose right vertical map is the decomposition map, and whose bottom horizontal map $\eta_G$ sends $[M]$ to $[M]$ for any $M\in\ltriv{FGe}$. In other words, if $M$ is indecomposable (i.e., $[M]\in T(FGe)$ a standard basis element) then $[M]$ is mapped to the sum of its composition factors (in terms of the standard basis in $R(FGe)$). Recall from Proposition~\ref{prop p-perm equiv}(b) that the left vertical maps are indeed isomorphisms preserving the standard basis elements and vertices. Recall also from \cite[Theorem~3.6.15(i)]{NagaoTsushima1989} that the map $\kappa_G$ is injective on $Pr(\calO G e)$.

For an element $\omega\in T(FG\ebar)$ we will denote the image under $\kappa_G$ of the corresponding element in $T(\calO Ge)$ by $\omega^{\KK}\in R(\KK Ge)$.
 \end{nothing}

The following proposition is well-known to specialists. We state it for easy reference.

\begin{proposition}\label{prop ghost group}
Let $A$ be a block of $\calO G$ and let $\calBPtilde(A)$ denote a set of representatives of the $G$-orbits of $A$-Brauer pairs.

\smallskip
{\rm (a)} The map
\begin{equation*}
  T(A)\mapsto \prod_{(P,e)\in\calBPtilde(A)} R(F[N_G(P,e)]e)\,, \quad 
  \omega\mapsto \bigl(\eta_{N_G(P,e)}(\omega(P,e))\bigr)_{(P,e)\in\calBPtilde(A)}\,,
\end{equation*}
is an injective group homomorphism and has finite cokernel.

\smallskip
{\rm (b)} The map
\begin{equation*}
  T(A)\mapsto \prod_{(P,e)\in\calBPtilde(A)} R(\KK[N_G(P,e)]e)\,, \quad 
  \omega\mapsto \bigl((\omega(P,e)^\KK)\bigr)_{(P,e)\in\calBPtilde(A)}\,,
\end{equation*}
is an injective group homomorphism.
\end{proposition}

\begin{proof}
(a) By Conlon's Theorem (see \cite[Theorem~5.5.4]{Benson1998}), one has an injective map
\begin{equation*}
  T(\calO G)\mapsto \prod_{P} R(F[N_G(P)])\,, \quad 
  \omega\mapsto \bigl(\eta_{N_G(P)}(\omega(P))\bigr)_{P}\,,
\end{equation*}
where $P$ runs through a set of representatives of the $G$-conjugacy classes of $p$-subgroups of $G$. By Proposition~\ref{prop p-perm equiv}(c) it has finite cokernel. By Lemma~\ref{lem Brauer map compatibility}, this map splits into a direct sum of maps with respect to each block $A$ of $\calO G$. Further, the Morita equivalence in 
\ref{noth covering} gives the statement of Part~(a).

\smallskip
(b) This follows from Part~(a) and the commutativity of the diagram in \ref{noth T(G)}(c).
\end{proof}

The following Lemma will be used in the proof of Lemma~\ref{lem M Bp implies gamma Bp}.

\begin{lemma}\label{lem equal coefficients}
Let $0\neq\omega\in T(FG)$ and let $(P,e)$ be a Brauer pair of $FG$ which is maximal among all the Brauer pairs $(Q,f)$ of $FG$ satisfying $M(Q,f)\neq \{0\}$ for some indecomposable $FG$-module $M$ appearing in $\omega$.

\smallskip
{\rm (a)} Let $M$ be an indecomposable $p$-permutation $FG$-module appearing in $\omega$ and satisfying $M(P,e)\neq \{0\}$. Then $(P,e)$ is a maximal $M$-Brauer pair, $M(P,e)$ is an indecomposable $p$-permutation $F[N_G(P,e)]$-module, and the coefficient of $[M]$ in $\omega\in T(FG)$ equals the coefficient of $[M(P,e)]$ in $\omega(P,e)\in T(F[N_G(P,e)])$. 

\smallskip
{\rm (b)} One has $0\neq\omega(P,e)\in Pr(F[N_G(P,e)/P])$.
\end{lemma}
 
\begin{proof}
(a)  The maximality of $(P,e)$ and Proposition~\ref{prop Brauer pairs for M} imply that $P$ is a vertex of $M$ and that $(P,e)$ is a maximal $M$-Brauer pair. The rest follows from Proposition~\ref{prop isomorphic test}.

\smallskip
(b) This follows immediately from Part~(a) and Proposition~\ref{prop p-perm equiv}.
\end{proof}

\begin{remark}\label{rem BX result}
We will need to use the following result from \cite[Corollary~2.6]{BoltjeXu2008}, stating that the generalized decomposition map on characters of $p$-permutation $\calO G$-modules is an element of the Brauer character ring (without extending scalars) and can be expressed via the Brauer construction: Let $M\in\ltriv{\calO G}$ and let $u\in G$ be a $p$-element. Then 
 \begin{equation}\label{eqn BX result}
    d^u_G(\kappa_G([M]) )=   \eta_{C_G(u)}(\res^{N_G(\langle u \rangle)}_{C_G(u)}([M(\langle u\rangle)]))
 \end{equation}
 in $R(F [C_G(u)])$.
\end{remark}

\begin{lemma}\label{lem going down criterion}
Let $\omega\in T(FG)$. For every Brauer pair $(P,e)$ of $FG$ set  $\chi_{(P,e)}:=(\omega(P,e))^{\KK}\in R(\KK [N_G(P,e)])$ and $\psi_{(P,e)}:=\res^{N_G(P,e)}_{C_G(P)}(\chi_{(P,e)})\in R(\KK[C_G(P)])$. Assume that, for every Brauer pair $(P,e)$ of $FG$, the following two conditions are satisfied:

\smallskip
{\rm (i)} If $\chi_{(P,e)}\neq 0$ in $R(\KK[N_G(P,e)])$ then $\psi_{(P,e)}\neq 0$ in $R(\KK[C_G(P)])$.

\smallskip
{\rm (ii)} If $\psi_{(P,e)}\neq 0$ in $R(\KK[C_G(P)])$ then $d_{C_G(P)}(\psi_{(P,e)})\neq 0$ in $R(F[C_G(P)])$.

\smallskip
Then, for any two Brauer pairs $(P,e)\le (Q,f)$ of $\calO G$, one has: If $\psi_{(Q,f)}\neq 0$ in $R(\KK[C_G(Q)])$ then $\psi_{(P,e)}\neq 0$ in $R(\KK[C_G(P)])$.
\end{lemma}
 
\begin{proof}
Arguing by induction on $[Q:P]$ we may assume that $(P,e)\trianglelefteq (Q,f)$ and that $Q/P$ is cyclic. Thus, there exists $u\in Q$ such that $P\langle u\rangle = Q$. Assume that $\psi_{(P,e)}=0$ in $R(\KK[C_G(P)])$ and set $I:=N_G(P,e)$. Then, by (i), also $\chi_{(P,e)}=0$ in $R(\KK I)$. By Lemma~\ref{lem Brauer map compatibility} applied to the $I$-stable idempotent $e$, the element $\omega(P)\in T(FN_G(P))$ and the $p$-subgroup $\langle u\rangle\le I$, we obtain $(e\omega(P))(\langle u\rangle) = \br_{\langle u\rangle}(e)(\omega(P)(\langle u\rangle))$ in $T(N_I(\langle u\rangle))$. Note that $\br_{\langle u\rangle}(e)= \br_Q(e)$. Thus, by Proposition~\ref{prop more on p-perm}(b), after further restriction (omitted in the notation) we obtain
\begin{equation*}
  (e\omega(P))(\langle u \rangle) = \br_Q(e) \omega(Q)
\end{equation*}
in $T(F[N_I(\langle u\rangle)\cap N_G(Q)])$. Applying Equation~(\ref{eqn BX result}) to $e\omega(P)\in T(FI)$ and noting that $C_I(u)\le N_I(\langle u\rangle)\cap N_G(Q)$, we obtain further
\begin{equation*}
  \eta_{C_I(u)}(\br_Q(e) \res^{N_G(Q)}_{C_I(u)}(\omega(Q)) )
  = \eta_{C_I(u)}\res^{N_I(\langle u \rangle)}_{C_I(u)}((e\omega(P))(\langle u\rangle))
  = d_I^u((e\omega(P))^{\KK}) = d_I^u ( \chi_{(P,e)}) = 0
\end{equation*}
Restricting further to $C_G(Q)$, multiplying by $f$, and using that $f\br_Q(e)=f$, we finally have
\begin{align*}
  0 & =\eta_{C_G(Q)}\bigl(f\br_Q(e)\res^{N_G(Q)}_{C_G(Q)}(\omega(Q))\bigr) 
  = \eta_{C_G(Q)}\bigl(\res^{N_G(Q,f)}_{C_G(Q)}(f\omega(Q))\bigr) \\
  & = d_{C_G(Q)}(\res^{N_G(Q,f)}_{C_G(Q)}(f\omega(Q))^\KK)
  = d_{C_G(Q)}(\psi_{(Q,f)})\,.
\end{align*}
Condition~(ii) now implies that $\psi_{(Q,f)}=0$, and the result follows.
\end{proof}

The following definition is similar to Definition~\ref{def Brauer pairs for M}. It will be used extensively in Section~\ref{sec Brauer pairs of ppeqs}.

\begin{definition}\label{def Brauer pairs for T(FG)}
Let $\omega\in T(FG)$ or $\omega\in T(\calO G)$. A Brauer pair $(P,e)$ of $FG$ is called an {\em $\omega$-Brauer pair} if $\omega(P,e) = e\omega(P)\neq 0$ in $T(F[N_G(P,e)])$. The set of $\omega$-Brauer pairs is denoted by $\calBP(\omega)$. Its corresponding set of Brauer pairs over $\calO$ is denoted by $\calBP_\calO(\omega)$.
\end{definition}

\begin{notation}
Let $X$ be a subgroup of $G\times H$ and let $d\in Z(\calO[G\times H])$. We denote by $T^\Delta(\calO X d)$ the subgroup of $T(\calO Xd)$ which is spanned by all standard basis elements $[M]$, where $M$ is an indecomposable $p$-permutation $\calO Xd$-module with twisted diagonal vertices (as subgroups of $G\times H$). If $e\in Z(\calO G)$ and $f\in Z(\calO H)$ are idempotents then we set $T^\Delta(\calO Ge,\calO Hf):= T^\Delta(\calO[G\times H](e\otimes f^*))$. Similarly we define groups $T^\Delta(FXd)$ and $T^\Delta(FGe,FHf)$.
\end{notation}

\begin{definition}
Let $e\in Z(\calO G)$ and $f\in Z(\calO H)$ be non-zero idempotents. A {\em $p$-permutation equivalence} between $\calO Ge$ and $\calO Hf$ is an element $\gamma\in T^\Delta(\calO Ge, \calO Hf)$ satisfying
\begin{equation}\label{eqn ppeq def}
  \gamma\cdotH\gamma^\circ = [\calO Ge] \ \text{in $T^\Delta(\calO Ge,\calO Ge)$}\quad\text{and}\quad
  \gamma^\circ\cdotG\gamma = [\calO Hf]\ \text{in $T^\Delta(\calO Hf, \calO Hf)$}.
 \end{equation}
The set of $p$-permutation equivalences between $\calO Ge$ and $\calO Hf$ will be denoted by $T_o^\Delta(\calO Ge,\calO Hf)$ (\lq o\rq\ for \lq orthogonal\rq). Similarly, we define $p$-permutation equivalences between $FGe$ and $FHf$, and denote the resulting set by $T^\Delta_o(FGe,FHf)$. Clearly, an element $\gamma\in T^\Delta(\calO Ge,\calO Hf)$ is a $p$-permutation equivalence between $\calO Ge$ and $\calO Hf$ if and only if $\gammabar\in T^\Delta(FGe,FHf)$ is a $p$-permutation equivalence between $FGe$ and $FHf$. Moreover, we denote the set of elements $\gamma\in T^\Delta(\calO Ge,\calO Hf)$ satisfying the first (resp.~second) equation in (\ref{eqn ppeq def}) by $T_l^\Delta(\calO Ge,\calO Hf)$ (resp.~$T_r^\Delta(\calO Ge,\calO Hf)$) and call them {\em left} (resp.~{\em right}) $p$-permutation equivalences between $\calO Ge$ and $\calO Hf$. We will see later that $T^\Delta_l(\calO G e,\calO Hf) = T^\Delta_o(\calO Ge,\calO Hf) = T^\Delta_r(\calO Ge,\calO Hf)$.
 \end{definition}

 \begin{proposition}\label{prop p-perm implies perfect isom}
 Let $e\in Z(\calO G)$ and $f\in Z(\calO H)$ be non-zero idempotents and let 
 $\gamma\in T^\Delta_o(\calO Ge, \calO Hf)$. Then $\mu:=\kappa_{G\times H}(\gamma)\in R(\KK Ge,\KK Hf)$ is a perfect isometry between $\KK Ge$ and $\KK Hf$.
 \end{proposition}
 
 \begin{proof}
 The virtual character $\mu$ is perfect by Remark~\ref{rem isometry}(c) and (d). Moreover, applying $\kappa_{G\times G}$ to the equation $\gamma\cdotH\gamma^\circ = [\calO Ge]$ in $T(\calO Ge,\calO Ge)$ implies $\mu\cdotH\mu^\circ = [\KK Ge]\in R(\KK Ge,\KK Ge)$. Similarly, $\gamma^\circ\cdotG\gamma=[\calO Hf]$ implies $\mu^\circ\cdotG\mu=[\KK Hf]\in R(\KK Hf,\KK Hf)$. By Remark~\ref{rem isometry}(a), $\mu$ is an isometry between $\KK Ge$ and $\KK Hf$.
\end{proof}
  
\section{Brauer pairs of $p$-permutation equivalences}\label{sec Brauer pairs of ppeqs}

Throughout this section we assume that $G$ and $H$ are finite groups and that the $p$-modular system $(\KK,\calO,F)$ is large enough for $G$ and $H$. Moreover, we fix a non-zero sum $A=\calO Ge_A$ of blocks of $\calO G$ and a non-zero sum $B=\calO He_B$ of blocks of $\calO H$, with $e_A$ and $e_B$ their respective identity elements. Furthermore, we assume throughout this section that $\gamma\in T^\Delta_l(A,B)$, i.e., $\gamma\in T^\Delta(A,B)$ and $\gamma$ satisfies
\begin{equation}\label{eqn left ppeq}
  \gamma\cdotH\gamma^\circ=[A]\quad\text{in $T^\Delta(A,A)$.}
\end{equation}
Instead of requiring that $\gamma$ is a $p$-permutation equivalence between $A$ and $B$, we prove as much as we can under the weaker assumption in (\ref{eqn left ppeq}) and we will finally show in Section~\ref{sec char criterion} that this implies that $\gamma$ is a $p$-permutation equivalence between $A$ and the sum of the blocks $B'$ of $\calO H$ that {\em support} $\gamma$ from the right, i.e., the sum of those blocks $B'$ with $\gamma e_{B'}\neq 0$ in $T(\calO G,\calO H)$. Note that Equation~(\ref{eqn left ppeq}) implies that $A$ is precisely the sum of those blocks of $\calO G$ that support $\gamma$ from the left.

\smallskip
The main result in this section is Theorem~\ref{thm gamma is uniform}, especially Parts~(a) and (b), which show that the $\gamma$-Brauer pairs behave exactly as the Brauer pairs of an indecomposable $p$-permutation module.

\begin{notation}\label{not gamma mu nu}
Let $\Delta(P,\phi,Q)$ be a twisted diagonal $p$-subgroup of $G\times H$. We will abbreviate the element $\gamma(\Delta(P,\phi,Q))\in T(F[N_{G\times H}(\Delta(P,\phi,Q))])$ by $\gammabar(P,\phi,Q)$ and will denote the resulting elements in the other representation groups of $N_{G\times H}(\Delta(P,\phi,Q))$ from the diagram in \ref{noth T(G)}(c) by
\begin{diagram}[50]
  \movevertex(-20,0){\gamma(P,\phi,Q)} & \mapsto & \movevertex(20,0){\mu(P,\phi,Q)} &&
  \movearrow(-20,-3){\begin{sideways}{\small $\lmapsto$} \end{sideways}} & & 
  \movearrow(20,-3){\begin{sideways}{\small $\lmapsto$} \end{sideways}} &&
  \movevertex(-20,0){\gammabar(P,\phi,Q)} & \mapsto & \movevertex(20,0){\nu(P,\phi,Q)} &&
\end{diagram}
Via restriction, we will also view these elements in the corresponding Grothendieck groups of any subgroup of $N_{G\times H}(\Delta(P,\phi,Q))$, in particular of the subgroup $C_G(P)\times C_H(Q)=C_{G\times H}(\Delta(P,\phi,Q))$. Note that the four maps commute with restrictions, so that the restricted elements still are related through these maps. 
\end{notation}

\begin{remark}\label{rem emuf} 
(a) Let $\Delta(P,\phi,Q)$ be a twisted diagonal $p$-subgroup of $G\times H$. If $(P,e)$ is a Brauer pair of $FG$ and $(Q,e)$ is a Brauer pair of $FH$ then $(\Delta(P,\phi,Q), e\otimes f^*)$ is a Brauer pair of $F[G\times H]$, using that $C_{G\times H}(\Delta(P,\phi,Q))= C_G(P)\times C_H(Q)$. Conversely, if $(\Delta(P,\phi,Q), e\otimes f^*)$ is a Brauer pair of $F[G\times H]$ then $(P,e)$ is a Brauer pair of $FG$ and $(Q,f)$ is a Brauer pair of $FH$. In this case, with $I:=N_G(P,e)$ and $J:=N_H(Q,f)$, one has $N_{G\times H}(\Delta(P,\phi,Q),e\otimes f^*)=N_{I\times J}(\Delta(P,\phi,Q))$ and $e\gamma(P,\phi,Q)f\in T^\Delta(\calO[N_{I\times J}(\Delta(P,\phi,Q))](e\otimes f^*))$. Moreover, $(\Delta(P,\phi,Q),e\otimes f^*)$ is an $A\otimes B^*$-Brauer pair if and only if $(P,e)$ is an $A$-Brauer pair and $(Q,f)$ is a $B$-Brauer pair. We will denote the set of $A\otimes B^*$-Brauer pairs $(X,d)$, where $X\le G\times H$ is a twisted diagonal $p$-subgroup, by $\calBP^\Delta(A,B)$, or $\calBP_\calO^\Delta(A,B)$ if lifted to $\calO$.

\smallskip
(b) Note that, for every $(\Delta(P,\phi,Q),e\otimes f^*)\in\calBP^\Delta_\calO(\calO G,\calO H)$, the element $e\mu(P,\phi,Q)f\in R(\KK[C_G(P)]e,\KK[C_H(Q)]f)$ is perfect. In fact, $e\gamma(P,\phi,Q)f\in T^\Delta(\calO[C_G(P)]e,\calO[C_H(Q)]f)$, by Lemma~\ref{lem vertex of M(P)}(b) and then Remark~\ref{rem isometry}(d) applies.  Moreover, the restriction of $e\gamma(P,\phi,Q)f$ to $C_G(P)\times \{1\}$ and to $\{1\}\times C_H(Q)$ yields elements of $Pr(\calO[C_G(Q)])$ and $Pr(\calO[C_H(Q)])$, respectively, see Lemma~\ref{lem vertex of M(P)}(b). Note that $e\gamma(P,\phi,Q)f=0$ unless $(P,e)$ is an $A$-Brauer pair and $(Q,f)$ is a $B$-Brauer pair. Since every module appearing in $\gamma$ has twisted diagonal vertex, one has $\gamma(X)=0\in T(\calO[N_{G\times H}(X)])$ for every $p$-subgroup $X\le G\times H$ which is {\em not} twisted diagonal. Thus, every $\gamma$-Brauer pair has a twisted diagonal subgroup as first component.
 
\smallskip
(c) Let $\Delta(P',\phi',Q')\le\Delta(P,\phi,Q)$ be twisted diagonal subgroups of $G\times H$, i.e., $Q'\le Q$,  $P'=\phi(Q)\le P$, and $\phi'=\phi|_{Q'}$. Furthermore, let $(\Delta(P,\phi,Q),e\otimes f^*)$ and $(\Delta(P',\phi',Q'), e'\otimes f'^*)$ be $\calO[G\times H]$-Brauer pairs. Then $(\Delta(P',\phi',Q'),e'\otimes f'^*)\le (\Delta(P,\phi,Q), e\otimes f^*)$ if and only if $(P',e')\le(P,e)$ and $(Q',f')\le (Q,f)$.
\end{remark}

The proofs of the following two Lemmas use the two statements and the notation in Theorem~\ref{thm BP decomp}(d).

\begin{lemma}\label{lem unique Lambdatilde lr}
Let $(P,e)\in\calBP_\calO(A)$. Consider the set $\Lambda_B\subseteq\Lambda$ of pairs $(\phi,(Q,f))$, where $(Q,f)\in\calBP_\calO(B)$ and $\phi\colon Q\myiso P$ is an isomorphism, together with its $H$-action from Theorem~\ref{thm BP decomp}(d). There exists a unique $H$-orbit of pairs $(\phi,(Q,f))\in\Lambda_B$ such that 
\begin{equation*}
  e\mu(P,\phi,Q)f\neq 0\quad\text{in $R(\KK[C_G(P)]e,\KK[C_H(Q)]f)$.}
\end{equation*} 
Moreover, for each pair $(\phi,(Q,f))\in\Lambda_B$ which satisfies this condition, the element $e\mu(P,\phi,Q)f$ is a perfect isometry between $\KK[C_G(P)]e$ and $\KK[C_H(Q)]f$, and $e\nu(P,\phi,Q)f\neq0$ in $R(F[C_G(P)]e,F[C_H(Q)]f)$.
\end{lemma}

\begin{proof}
We apply the Brauer construction with respect to $\Delta(P)$ to Equation~(\ref{eqn left ppeq}). By Proposition~\ref{prop Brauer con of block}(b), we have $eA(\Delta(P))e\cong F[C_G(P)]e$ as $(F[C_G(P)],F[C_G(P)])$-bimodules. Thus, by Theorem~\ref{thm BP decomp}(d), the following equation holds in $T^\Delta(F[C_G(P)]e,F[C_G(P)]e)$:
\begin{equation*}
  [F[C_G(P)]e] = [eA(\Delta(P))e] = e(\gamma\cdotH\gamma^\circ)(\Delta(P))e = \sum_{(\phi,(Q,f))\in\Lambdatilde}
  e\gammabar(P,\phi, Q)f\mathop{\cdot}\limits_{C_H(Q)}f\gammabar^\circ(Q,\phi^{-1},P)e\,,
\end{equation*}
where $\Lambdatilde$ denotes a set of representatives of the $H$-orbits of $\Lambda$, as in Theorem~\ref{thm BP decomp}(d). If $(Q,f)\in \calBP_\calO(\calO G)\smallsetminus\calBP_\calO(B)$ then $e\gammabar(P,\phi,Q)f=0$, since $\gamma\in T(A,B)$. Thus, we may replace $\Lambdatilde$ in the above summation by $\Lambdatilde_B:=\Lambda_B\cap\Lambdatilde$.
Lifting the last equation from $F$ to $\calO$ and extending scalars from $\calO$ to $\KK$, we obtain the equation
\begin{equation*}
  [\KK[C_G(P)e]] = \sum_{(\phi, (Q,f))\in\Lambdatilde_B} e\mu(P,\phi, Q)f\mathop{\cdot}\limits_{C_H(Q)}(e\mu(P,\phi,Q)f)^\circ
\end{equation*}
in $R(\KK[C_G(P)]e, \KK[C_G(P)]e)$. Since each $e\mu(P,\phi,Q)f\in R(\KK[C_G(P)]e,\KK[C_H(Q)]f)$ is perfect, we can apply Corollary~\ref{cor isometry separation 2} and obtain that there exists a unique element $\lambda=(\phi,(Q,f))\in\Lambdatilde_B$ such that $e\mu(P,\phi,Q)f\neq 0$ in $R(\KK[C_G(P)]e,\KK[C_H(Q)]f)$ and that this element is a perfect isometry between $\KK[C_G(P)]e$ and $\KK[C_H(Q)]f$. The last statement of the lemma follows from Lemma~\ref{lem non-zero criterion}. 
\end{proof}

\begin{lemma}\label{lem unique Lambdahat lr}
Let $(P,e)\in\calBP_\calO(A)$ and set $I:=N_G(P,e)$ and $X:=N_{I\times I}(\Delta(P))=\Delta(I)(C_G(P)\times \{1\})$. Consider the set $\Lambda_B$ of pairs $(\phi,(Q,f))$, where $(Q,f)\in\calBP_\calO(B)$ and $\phi\colon Q\myiso P$ is an isomorphism, together with its $I\times H$-action from Theorem~\ref{thm BP decomp}(d). For each $\lambda=(\phi,(Q,f))\in\Lambda_B$ we set
\begin{equation*}
  J(\lambda):=N_H(Q,f),\quad I(\lambda):= N_{(I,\phi,J(\lambda))}\le I,\quad\text{and}\quad 
  X(\lambda):=N_{I\times J(\lambda)}(\Delta(P,\phi,Q))\,.
\end{equation*}
Then, $X*X(\lambda)=X(\lambda)$, and for each $\chi\in\Irr(\KK X(e\otimes e^*))$, there exists a unique $I\times H$-orbit of pairs $\lambda=(\phi,(Q,f))\in\Lambda_B$ such that
\begin{equation*}
  \chi\mathop{\cdot}\limits_{G}^{X,X(\lambda)} e\mu(P,\phi,Q)f\neq 0\quad\text{in}\ R(\KK[X(\lambda)](e\otimes f^*))\,.
\end{equation*}
Moreover, for each $\lambda=(\phi,(Q,f))\in\Lambda_B$ satisfying this condition, one has
\begin{equation*}
  \chi\mathop{\cdot}\limits_{G}^{X,X(\lambda)} e\mu(P,\phi,Q)f\in\pm\Irr(\KK[X(\lambda)](e\otimes f^*))\,.
\end{equation*}
\end{lemma}

\begin{proof}
First note that $X*X(\lambda)=X(\lambda)$ for each $\lambda\in\Lambda_B$, by Lemma~\ref{lem  ext tp special}(b) with $S=I$, $T=J(\lambda)$, and $Y=X(\lambda)$. Next, we apply the Brauer construction with respect to $\Delta(P)$ to Equation (\ref{eqn left ppeq}). By Proposition~\ref{prop Brauer con of block}(b), we have $eA(\Delta(P))e\cong F[C_G(P)]e$ in $\lmod{FX}$. Thus, by Theorem~\ref{thm BP decomp}(d) applied to $S=I$, we have \begin{equation*}
  [F[C_G(P)]e] = \sum_{\lambda=(\phi,(Q,f))\in\Lambdahat_B} \ind_{X'(\lambda)}^X\Bigl(e\gammabar(P,\phi,Q)f
  \mathop{\cdot}\limits_{H}^{X(\lambda),X(\lambda)^\circ} f\gammabar^\circ(Q,\phi^{-1},P)e\Bigr)
\end{equation*}
in $T^\Delta(FX(e\otimes e^*))$, where $X'(\lambda):=X(\lambda)*X(\lambda)^\circ=\Delta(I(\lambda))\cdot(C_G(P)\times \{1\})\le X$ and $\Lambdahat_B$ is a set of representatives of the $I\times H$-orbits of $\Lambda_B$. Lifting this equation from $F$ to $\calO$ and extending scalars from $\calO$ to $\KK$, we obtain
\begin{equation*}
  [\KK[C_G(P)]e] = \sum_{\lambda=(\phi,(Q,f))\in\Lambdahat_B} \ind_{X'(\lambda)}^X\Bigl( e\mu(P,\phi,Q)f
  \mathop{\cdot}\limits_{H}^{X(\lambda),X(\lambda)^\circ} (e\mu(P,\phi,Q)f)^\circ\Bigr)
\end{equation*}
in $R(\KK X(e\otimes e^*))$. Now let $\chi\in\Irr(\KK X(e\otimes e^*))$ and apply the group homomorphism
\begin{equation*}
  \bigl(\chi\mathop{\cdot}\limits_G^{X,X}-,\chi\bigr)_X\colon R(\KK X(e\otimes e^*))\to \ZZ
\end{equation*}
to the last equation. We obtain
\begin{equation*}
  \bigl(\chi\mathop{\cdot}\limits_G^{X,X} [\KK C_G(P)e],\chi\bigr)_X 
  = \sum
  \Bigl(\chi\mathop{\cdot}\limits_G^{X,X}\ind_{X'(\lambda)}^X\Bigl(e\mu(P,\phi,Q)f
  \mathop{\cdot}\limits_{H}^{X(\lambda),X(\lambda)^\circ} (e\mu(P,\phi,Q)f)^\circ\Bigr)\ ,\ \chi\Bigr)_X\,,
\end{equation*}
where the sum runs over elements $\lambda=(\phi,(Q,f))\in\Lambdahat_B$. Applying Proposition~\ref{prop ext hom duals}(c) to the left hand side and Lemma~\ref{lem ext tp special}(b) to the right hand side of the last equation, we obtain
\begin{equation*}
  1 = \sum_{\lambda=(\phi,(Q,f))\in\Lambdahat_B} \Bigl(\chi\mathop{\cdot}\limits_G^{X,X(\lambda)}
  e\mu(P,\phi,Q)f, \chi\mathop{\cdot}\limits_G^{X,X(\lambda)}e\mu(P,\phi,Q)f\Bigr)_{X(\lambda)}\,.
\end{equation*}
The statements in the lemma are now immediate.
\end{proof}

Note that in general restrictions of virtual non-zero characters can vanish. However:

\begin{corollary}\label{cor mu non-zero}
Let $(\Delta(P,\phi,Q), e\otimes f^*)\in\calBP_\calO^\Delta(A,B)$ and set $I:=N_G(P,e)$ and $J:=N_H(Q,f)$. Then
\begin{equation*}
  e\mu(P,\phi,Q)f\neq0\quad\text{in $R(\KK[N_{I\times J}(\Delta(P,\phi,Q))](e\otimes f^*))$}
\end{equation*}
if and only if, after restriction,
\begin{equation*}
  e\mu(P,\phi,Q)f\neq 0 \quad \text{in $R(\KK [C_G(P)]e,\KK[C_H(Q)]f)$.}
\end{equation*}
\end{corollary}

\begin{proof}
Set $Y:=N_{I\times J}(\Delta(P,\phi,Q))$ and $\mu:=e\mu(P,\phi,Q)f\in R(\KK Y(e\otimes f^*))$. We need to show that if $\mu\neq 0$ then also $\res^Y_{C_G(P)\times C_H(Q)}(\mu)\neq 0$. Set $X:=N_{I\times I}(\Delta(P))=\Delta(I)\cdot(C_G(P)\times \{1\})$. Then, following the construction in \ref{noth ext tp and res}, we have $\Xtilde= \Delta(N_{(I,\phi,J)})\cdot(C_G(P)\times \{1\})$ (with respect to $Y$), since $p_1(Y)=N_{(I,\phi,J)}$ by Proposition~\ref{prop Nphi}(c). Moreover, Lemma~\ref{lem comp form}(a) applied to $Y$ implies $\Xtilde=Y*Y^\circ$. By \ref{noth ext tp and res} and Proposition~\ref{prop ext hom duals}(c)  we obtain
\begin{equation*}
  [\KK[C_G(P)]e]\mathop{\cdot}\limits_{H}^{X,Y}\mu = 
  \res^{X}_{\Xtilde}([\KK[C_G(P)]e])\mathop{\cdot}\limits_{H}^{\Xtilde,Y} \mu = \mu
\end{equation*}
in $R(\KK Y(e\otimes f^*))$, since $X*Y=\Xtilde*Y=Y*Y^\circ *Y=Y$. Therefore, $\KK[C_G(P)]e$ has an irreducible constituent $\chi\in\Irr(\KK X(e\otimes e^*))$ such that
\begin{equation*}
  \chi\mathop{\cdot}\limits_{H}^{X,Y}\mu\neq 0 \quad\text{in $R(\KK Y(e\otimes f^*))$.}
\end{equation*}
By Lemma~\ref{lem unique Lambdahat lr}, $\chi\mathop{\cdot}\limits_{H}^{X,Y}\mu\in\pm\Irr(\KK X (e\otimes e^*))$. Now, Lemma~\ref{lem non-zero res} implies that $\res^Y_{C_G(P)\times C_H(Q)}(\mu)\neq 0$.
\end{proof}

\begin{corollary}\label{cor mu non-zero going down}
Let $(\Delta(P',\phi',Q'), e'\otimes f'^*)\le(\Delta(P,\phi,Q), e\otimes f^*)\in\calBP^\Delta_\calO(A,B)$. If $e\mu(P,\phi,Q)f\neq 0$ in $R(\KK[C_G(P)]e,\KK[C_H(Q)]f)$ then $e'\mu(P',\phi',Q')f'\neq 0$ in $R(\KK[C_G(P')]e',\KK[C_H(Q')]f')$.
\end{corollary}

\begin{proof}
We apply Lemma~\ref{lem going down criterion} to $\gammabar\in T(F[G\times H])$ and the inclusion of $A\otimes B^*$-Bauer pairs $(\Delta(P',\phi',Q'),e'\otimes f'^*)\le (\Delta(P,\phi,Q), e\otimes f^*)$. Condition~(i) in Lemma~\ref{lem going down criterion} is satisfied by Corollary~\ref{cor mu non-zero}. And Condition~(ii) in Lemma~\ref{lem going down criterion} is satisfied by Lemma~\ref{lem unique Lambdatilde lr}.
\end{proof}

\begin{lemma}\label{lem M Bp implies gamma Bp}
Let $M$ be an indecomposable $p$-permutation $(A,B)$-bimodule that appears in $\gamma$ and let $(\Delta(P,\phi,Q),e\otimes f^*)\in\calBP_\calO(M)$. Then $e\mu(P,\phi,Q)f\neq 0$ in $R(\KK[C_G(P)]e,\KK[C_H(Q)]f)$.
\end{lemma}

\begin{proof}
We abbreviate $X:=\Delta(P,\phi,Q)$. By Corollary~\ref{cor mu non-zero going down} we may assume that the $A\otimes B^*$-Brauer pair $(X,e\otimes f^*)$ is maximal with respect to the property that there exists an indecomposable $p$-permutation $(A,B)$-bimodule $N$ appearing in $\gamma$ such that $eN(X)f\neq \{0\}$. Then, by Lemma~\ref{lem equal coefficients}(b), we have $0\neq e\gamma(X)f\in Pr(\calO[N_{I\times J}(X)/X])$, where $I:=N_G(P,e)$ and $J:=N_H(Q,f)$. Since the map $\kappa$ in \ref{noth T(G)}(c) is injective on $Pr(\calO[N_{I\times J}(X)/X])$ we have $e\mu(P,\phi,Q)f\neq 0$ in $R(\KK[N_{I\times J}(X)/X])$. Since inflation is injective on character groups and commutes with the map $\kappa$, we obtain that $e\mu(P,\phi,Q)f\neq 0$ in $R(\KK[N_{I\times J}(X)])$. Finally, Corollary~\ref{cor mu non-zero} implies that, after restriction, $e\mu(P,\phi,Q)f\neq 0$ in $R(\KK[C_G(P)]e,\KK[C_H(Q)]f)$.
\end{proof}

The following proposition gives convenient reformulations of being a $\gamma$-Brauer pair, see Definition~\ref{def Brauer pairs for T(FG)}.

\begin{proposition}\label{prop equiv gamma Brpair cond}
Let $(\Delta(P,\phi,Q), e\otimes f^*)\in \calBP^\Delta_{\calO}(A,B)$ and set $I:=N_G(P,e)$ and $J:=N_H(Q,f)$. Then the following are equivalent:

\smallskip
{\rm (i)} $(\Delta(P,\phi,Q),e\otimes f^*)\in\calBP_\calO(\gamma)$, i.e., $e\gammabar(P,\phi,Q)f\neq 0$ in $T(F[N_{I\times J}(\Delta(P,\phi,Q))](e\otimes f^*))$.

\smallskip
{\rm (ii)} $e\gammabar(P,\phi,Q)f\neq0$ in $T(F[C_G(P)]e,F[C_H(Q)]f)$.

\smallskip
{\rm (iii)} $e\mu(P,\phi,Q)f\neq 0$ in $R(\KK[N_{I\times J}(\Delta(P,\phi,Q))](e\otimes f^*))$.

\smallskip
{\rm (iv)} $e\mu(P,\phi,Q)f\neq 0$ in $R(\KK[C_G(P)]e,\KK[C_H(Q)]f)$.

\smallskip
{\rm (v)} $e\nu(P,\phi,Q)f\neq 0$ in $R(F[N_{I\times J}(\Delta(P,\phi,Q))](e\otimes f^*))$.

\smallskip
{\rm (vi)} $e\nu(P,\phi,Q)f\neq 0$ in $R(F[C_G(P)]e,F[C_H(Q)]f)$.

\smallskip
{\rm (vii)} $(\Delta(P,\phi,Q),e\otimes f^*)\in\calBP_\calO(M)$ for some indecomposable module $M\in\ltriv{A}_B$ appearing in $\gamma$.
\end{proposition}

\begin{proof}
Clearly, by the diagram in \ref{not gamma mu nu}, each of the conditions (ii)--(vi) implies (i). Similarly, the condition in (vi) implies each of the conditions (i)--(v). Moreover, (iv) implies (vi) by Lemma~\ref{lem unique Lambdatilde lr}. Finally, (vii) implies (iv) by Lemma~\ref{lem M Bp implies gamma Bp}, and clearly (i) implies (vii). This completes the proof of the lemma.
\end{proof}

\begin{corollary}\label{cor connecting Brauer pairs}
Let $(P,e)$ be an $A$-Brauer pair and define $\Lambda_B$ as in Lemma~\ref{lem unique Lambdatilde lr}. Then there exists a unique $H$-orbit of pairs $(\phi,(Q,f))\in\Lambda_B$ with the property that $(\Delta(P,\phi,Q),e\otimes f^*)$ is a $\gamma$-Brauer pair.
\end{corollary}

\begin{proof}
This follows immediately from Lemma~\ref{lem unique Lambdatilde lr} and the equivalence of (i) and (iv) in Proposition~\ref{prop equiv gamma Brpair cond}.
\end{proof}

The following Theorem shows that every $p$-permutation equivalence between $A$ and $B$ determines a bijection between the block direct summands of $A$ and of $B$, and that it is the sum of $p$-permutation equivalences between corresponding blocks.

\begin{theorem}\label{thm blockwise ppeq}
Assume that $\gamma\in T^\Delta_l(A,B)$. Let $\calI$ denote the set of primitive idempotents of $Z(A)$ and let $\calJ$ denote the set of primitive idempotents of $Z(B)$. For each $e\in\calI$ there exists a unique $f$ in $\calJ$ such that $e\gamma f\neq 0$ in $T^\Delta(\calO Ge, \calO Hf)$. If $e\in\calI$ and $f\in\calJ$ satisfies $e\gamma f\neq 0$ then $e\gamma f\in T^\Delta_l(\calO Ge, \calO Hf)$.

In particular, if $\gamma$ is a $p$-permutation equivalence between $A$ and $B$, then the condition $e\gamma f\neq 0$ defines a bijection between $\calI$ and $\calJ$, and if $e\gamma f\neq 0$ then $e\gamma f$ is a $p$-permutation equivalence between $\calO Ge$ and $\calO Hf$. Moreover, $\gamma=\sum_{e\in\calI} e\gamma f$, where $f\in\calJ$ corresponds to $e$.
\end{theorem}

\begin{proof}
The first statement follows immediately from Corollary~\ref{cor connecting Brauer pairs} applied to the $A$-Brauer pair $(\{1\}, e)$. Thus, $\gamma=\sum_{e\in\calI}e\gamma f$, where $f\in\calJ$ denotes the element corresponding to $e$. Suppose that $e\gamma f\neq 0$. Then $e\gamma=e\gamma f$ and multiplying the equation $\gamma\cdotH\gamma^\circ=[A]$ with $e$ from left and right yields $e\gamma f\cdotH(e\gamma f)^\circ=[\calO Ge]$. Thus, $e\gamma f\in T^\Delta_l(\calO Ge,\calO Hf)$.
\end{proof}

We are now ready to prove the main theorem of this section. It shows that $\gamma$-Brauer pairs behave very similar to $M$-Brauer pairs of an indecomposable $p$-permutation module $M$, cf.~Proposition~\ref{prop Brauer pairs for M}(b).

\begin{theorem}\label{thm gamma is uniform}
{\rm (a)} The set of $\gamma$-Brauer pairs form a $G\times H$-stable ideal in the poset of $A\otimes B^*$-Brauer pairs.

\smallskip
{\rm (b)} If $A$ and $B$ are blocks then any two maximal $\gamma$-Brauer pairs are $G\times H$-conjugate.

\smallskip
{\rm (c)} For $(\Delta(P,\phi,Q), e\otimes f^*)\in\calBP_\calO^\Delta(\gamma)$ the following are equivalent:

\smallskip
\quad {\rm (i)} $(\Delta(P,\phi,Q), e\otimes f^*)$ is a maximal $\gamma$-Brauer pair;

\smallskip
\quad {\rm (ii)} $(P,e)$ is a maximal $A$-Brauer pair;

\smallskip
\quad {\rm (iii)} $(Q,f)$ is a maximal $B$-Brauer pair.
\end{theorem}

\begin{proof}
(a) Clearly, the set of $\gamma$-Brauer pairs is closed under $G\times H$-conjugation. Moreover, by Corollary~\ref{cor mu non-zero going down} and Proposition~\ref{prop equiv gamma Brpair cond}, the set of $\gamma$-Brauer pairs is an ideal in the poset of $A\otimes B^*$-Brauer pairs.

\smallskip
(b) Now assume that $A$ and $B$ are blocks. Let $(D,e_D)$ be a maximal $A$-Brauer pair over $\calO$. By Corollary~\ref{cor connecting Brauer pairs} there exists a $B$-Brauer pair $(E,f_E)$ over $\calO$ and an isomorphism $\psi\colon E\myiso D$ such that $(\Delta(D,\psi,E),e_D\otimes f_E^*)$ is a $\gamma$-Brauer pair. Let $(\Delta(P',\phi',Q'),e'\otimes f'^*)$ be an arbitrary $\gamma$-Brauer pair. We will first prove that
\begin{equation}\label{eqn subcon}
  (\Delta(P',\phi',Q'),e'\otimes f'^*)\le_{G\times H} (\Delta(D,\psi,E),e_D\otimes f_E^*)\,.
\end{equation}
Since any two maximal $A$-Brauer pairs are $G$-conjugate, we may assume that $(P',e')\le (D,e_D)$. Set $R:=\psi^{-1}(P')\le \psi^{-1}(D)=E$ and let $f_R$ denote the unique block idempotent of $\calO[C_H(R)]$ such that $(R,f_R)\le(E,f_E)$. Since $(\Delta(D,\psi,E), e_D\otimes f_E^*)$ is a $\gamma$-Brauer pair, Part~(a) implies that also $(\Delta(P', \psi|_R,R), e'\otimes f_R^*)$ is a $\gamma$-Brauer pair. Since also $(\Delta(P',\phi',Q'), e'\otimes f'^*)$ is a $\gamma$-Brauer pair, Corollary~\ref{cor connecting Brauer pairs} implies that $(\psi|_R,(R,f_R))$ and $(\phi', (Q',f'))$ are $H$-conjugate. Thus,
\begin{equation*}
  (\Delta(P',\phi',Q'), e'\otimes f'^*) =_{\{1\}\times H} (\Delta(P',\psi|_R,R), e'\otimes f_R^*) \le (\Delta(D,\psi,E), e_D\otimes f_E^*)
\end{equation*}
and the claim is proven. This implies that $(\Delta(D,\psi,E), e_D\otimes f_E^*)$ is a maximal $\gamma$-Brauer pair, and also that every other maximal $\gamma$-Brauer pair is $G\times H$-conjugate to $(\Delta(D,\psi,E), e_D\otimes f_E^*)$.

\smallskip
(c) Let $A$ and $B$ be again sums of blocks. First recall from Remark~\ref{rem emuf}(b) that every $\gamma$-Brauer pair has a twisted diagonal subgroup as first component. Thus, (ii) implies (i) and (iii) implies (i), by the last statement in Remark~\ref{rem emuf}(a). 

In order to see that (i) implies (ii) and (iii) we claim that it suffices to show this in the situation where $A$ and $B$ are blocks. In fact, $(\Delta(P,\phi,Q),e\otimes f^*)$ is an $A'\otimes B'^*$-Brauer pair for the unique block direct summands $A'$ of $A$ and $B'$ of $B$, respectively, that satisfy $(\{1\},e_{A'})\le (P,e)$ and $(\{1\},e_{B'})\le (Q,f)$. Applying Corollary~\ref{cor connecting Brauer pairs} to the $A$-Brauer pair $(\{1,\},e_{A'})$, we see that $B'$ is uniquely determined by $A'$. Writing $\gamma$ as the sum of the elements $e'\gamma f'$, where $e'$ and $f'$ run through the block idempotents of $\calO G$ and $\calO H$ with $e_A e'=e'$ and $e_B f'=f'$, then multiplying both sides of (\ref{eqn left ppeq}) by $e_{A'}$ from the left and right, we see that $(e_{A'}\gamma e_{B'})\cdotH(e_{A'}\gamma e_{B'})^\circ = [A']$ in $T^\Delta(A',A')$. Since $(\Delta(P,\phi,Q), e\otimes f^*)$ is a maximal $\gamma$-Brauer pair, it is also a maximal $e_{A'}\gamma e_{B'}$-Brauer pair. Since maximal $A'$-Brauer pairs (resp.~$B'$-Brauer pairs) are also maximal $A$-Brauer pairs (resp.~$B$-Brauer pairs), we may assume from now on that $A$ and $B$ are blocks.

Now let $(D,e_D)$ be a maximal $A$-Brauer pair and let the $\gamma$-Brauer pair $(\Delta(D,\psi,E),e_D\otimes f_E^*)$ be chosen as in the proof of Part~(b). Then, by the claim proved there, we know that $(\Delta(D,\psi,E), e_D\otimes f_E^*)$ is a maximal $\gamma$-Brauer pair. Thus, by Part~(b), $(\Delta(P,\phi,Q), e\otimes f^*) \le_{G\times H} (\Delta(D,\psi,E),e_D\otimes f_E^*)$. This already shows that (i) implies (ii). Finally, in order to show that (i) implies (iii), it suffices to show that $(E,f_E)$ is a maximal $B$-Brauer pair, i.e., that $E$ is a defect group of $B$.

So let $q$ denote the order of a defect group of $B$. It suffices to show that $q\le |E|$, since $(E,f_E)$ is a $B$-Brauer pair. Set $\mu:=\kappa_{G\times H}(\gamma)\in R(\KK G, \KK H)$. Then $\mu\cdotH \mu^\circ=[\KK G e_A]$ in $R(\KK G, \KK G)$. Lemma~\ref{lem isometry crit} now implies that $\mu$ is an isometry between $\KK Ge_A$ and $\KK H e_B$. This implies further that
\begin{equation}\label{eqn block rank equation}
  \rk_{\calO}(B)=\dim_{\KK}(\KK H e_B) = \dim_{\KK}(\mu^\circ\cdotG \mu) = \rk_{\calO}(\gamma^\circ\cdotG\gamma)\,,
\end{equation}
where $\dim_\KK$ and $\rk_\calO$ also denote the integer-valued maps induced on the Grothendieck groups.
Since every $\gamma$-Brauer pair is contained in a $G\times H$-conjugate of $(\Delta(D,\psi,E), e_D\otimes f_E^*)$, Lemma~\ref{lem M Bp implies gamma Bp} and Proposition~\ref{prop equiv gamma Brpair cond} imply that every indecomposable $A\otimes B^*$-module that appears in $\gamma$ has a vertex contained in $\Delta(D,\psi,E)$. Therefore, every indecomposable $B\otimes A^*$-module that appears in $\gamma^\circ$ has a vertex contained in $\Delta(E,\psi^{-1},D)$. The Mackey formula for bimodules, see Theorem~\ref{thm Bouc-Mackey}, implies that each indecomposable $B\otimes B^\circ$-module that appears in $\gamma^\circ\cdotG \gamma$ has a vertex of order dividing $|E|$. By Theorem~4.7.5 in \cite{NagaoTsushima1989}, the $p$-part of the $\calO$-rank of each such indecomposable $B\otimes B^\circ$-module is a multiple of $|H|_p^2/|E|$. Thus, using Equation~(\ref{eqn block rank equation}), $\rk_{\calO}(B)_p$ is a multiple of $|H|_p^2/|E|$. On the other hand, by Theorem~5.10.1 in \cite{NagaoTsushima1989} we know that $\rk_{\calO}(B)_p=|H|^2_p/q$. This implies that $q$ divides $|E|$ and the proof of Part~(c) is complete.
\end{proof}

We conclude this section by proving an inverse to the association constructed in Corollary~\ref{cor connecting Brauer pairs}. This will be used in Section~\ref{sec fusion systems} to show that if $\gamma\in T^{\Delta}(A,B)$ satisfies (\ref{eqn left ppeq}) then the fusion systems of $A$ and $B$ are isomorphic.

\begin{lemma}\label{lem connecting Brauer pairs rl}
Let $(Q,f)\in\calBP_\calO(B')$ for a block $B'$ of $\calO H$ satisfying $\gamma e_{B'}\neq 0$. Then there exists a unique $G$-conjugacy class of pairs $((P,e),\phi)$, where $(P,e)\in\calBP_\calO(A)$ and $\phi\colon Q\myiso P$ is an isomorphism, such that $(\Delta(P,\phi,Q), e\otimes f^*)$ is a $\gamma$-Brauer pair. Here, $G$ acts on the set of such pairs via $\lexp{g}{((P,e),\phi)}=((\lexp{g}{P}, \lexp{g}{e}), c_{g}\circ\phi)$.
\end{lemma}

\begin{proof}
We first show the existence part of the lemma. There exists a block direct summand $A'$ of $A$ such that $e_{A'}\gamma e_{B'}\neq 0$ in $T(\calO G,\calO H)$. In other words, $(\{1\}, e_{A'}\otimes e_{B'}^*)$ is a $\gamma$-Brauer pair. Let $(\Delta(D,\psi,E), e_D\otimes f_E^*)$ be a maximal $\gamma$-Brauer pair containing $(\{1\}, e_{A'}\otimes e_{B'})$, then $(\Delta(D,\psi,E), e_D\otimes f_E^*)$ is also an $A'\otimes B'^*$-Brauer pair and $(E,f_E)$ is a maximal $B'$-Brauer pair by Theorem~\ref{thm gamma is uniform}(c). After conjugating this maximal $\gamma$-Brauer pair by an element in $\{1\}\times H$, if necessary, we may assume that $(Q,f)\le (E,f_E)$. Setting $P:=\psi(Q)$ and $\phi:=\psi|_Q\colon Q\myiso P$, there exists a unique primitive idempotent $e$ of $Z(\calO[C_G(P)])$ such that $(P,e)\le(D,e_D)$. This implies that $(\Delta(P,\phi,Q), e\otimes f^*)\le (\Delta(D,\psi,E), e_D\otimes f_E^*)$. Since the set of $\gamma$ -Brauer pairs is an ideal, see Theorem~\ref{thm gamma is uniform}(a), also $(\Delta(P,\phi,Q), e\otimes f^*)$ is a $\gamma$-Brauer pair.

\smallskip
Next, we show the uniqueness part. Consider $e\gamma(P,\phi,Q)f$ as element in $T^\Delta(\calO[C_G(P)]e,\calO[C_H(Q)]f)$ and $e\mu(P,\phi,Q)f$ as element in $R(\KK[C_G(P)]e,\KK[C_H(Q)]f)$. Since $\gamma\cdotH\gamma^\circ=[A]$ in $T^\Delta(A,A)$, we obtain $\gamma\cdotH\gamma^\circ\cdotG\gamma = \gamma$ in $T^\Delta(A,B)$ and
\begin{equation*}
  e\gamma(P,\phi,Q)f = e(\gamma\cdotH\gamma^\circ\cdotG \gamma)(\Delta(P,\phi,Q))f 
  = e\gamma(P,\phi,Q)f\mathop{\cdot}\limits_{C_H(Q)} f(\gamma^\circ\cdotG\gamma)(\Delta(Q))f\,,
\end{equation*}
in $T^\Delta(\calO[C_G(P)]e,\calO[C_H(Q)]f)$, where the last equation follows from Theorem~\ref{thm BP decomp}(b) and Corollary~\ref{cor connecting Brauer pairs}. Extending scalars from $\calO$ to $\KK$, the last equation implies
\begin{equation*}
  e\mu(P,\phi,Q)f = e\mu(P,\phi|_Q,Q)f \mathop{\cdot}\limits_{C_H(Q)} \bigl(f(\gamma^\circ\cdotG \gamma)(\Delta(Q))f\bigr)^{\KK}
\end{equation*}
in $R(\KK[C_G(P)]e,\KK[C_H(Q)]f)$. By Lemma~\ref{lem unique Lambdatilde lr}, $e\mu(P,\phi|_Q,Q)f$ is an isometry between $\KK[C_G(P)]e$ and $\KK[C_H(Q)]f$. Thus, the last equation implies
\begin{equation*}
  [\KK[C_H(Q)]f] = \bigl(f(\gamma^\circ\cdotG \gamma)(\Delta(Q))f\bigr)^{\KK}
\end{equation*}
in $R(\KK[C_H(Q)]f,\KK[C_H(Q)]f)$. Using again Theorem~\ref{thm BP decomp}(b), we can write
\begin{equation*}
  f(\gamma^\circ\cdotG\gamma)(\Delta(Q))f = \sum_{((P',e'),\phi')}
  f\gamma^\circ(\Delta(Q,\phi'^{-1},P'))e'\cdotG e'\gamma(P',\phi',Q)f\,,
\end{equation*}
in $T^\Delta(\calO[C_H(Q)]f,\calO[C_H(Q)]f)$, where $((P',e'),\phi')$ runs through representatives of the $G$-conjugacy classes of pairs as described in the statement of the lemma. Thus,
\begin{equation*}
  [\KK[C_H(Q)]f] = \sum_{((P',e'),\phi')} f\mu(P',\phi',Q)^\circ e' \cdotG e'\mu(P',\phi',Q) f
\end{equation*}
in $R(\KK[C_H(Q)]f,\KK[C_H(Q)]f)$. Corollary~\ref{cor isometry separation 2} now implies that, up to $G$-conjugacy, there exists a unique pair $((P',e'),\phi')$ such that $e'\mu(P',\phi',Q) f \neq 0$ in $R(\KK[C_G(P)]e',\KK[C_H(Q)]f)$. Together with Proposition~\ref{prop equiv gamma Brpair cond} this implies the desired uniqueness statement.
\end{proof}

\section{Fusion systems, local equivalences and finiteness}\label{sec fusion systems}

Throughout this section we assume again that $G$ and $H$ are finite groups, and that the $p$-modular system $(\KK,\calO, F)$ is large enough for $G$ and $H$. Moreover, we assume that $A=\calO Ge_A$ is a block of $\calO G$, $B=\calO He_B$ is a block of $\calO H$, and that $\gamma\in T^\Delta_l(A,B)$ is a {\em left} $p$-permutation equivalence between $A$ and $B$, i.e, $\gamma\cdotH\gamma^\circ=[A]$ in $T(A,A)$. Furthermore, we assume the notation from \ref{not gamma mu nu} and that $(\Delta(D,\phi,E), e_D\otimes f_E^*)$ is a maximal $\gamma$-Brauer pair. Then, by Theorem~\ref{thm gamma is uniform}(c), $(D,e_D)$ is a maximal $A$-Brauer pair and $(E,f_E)$ is a maximal $B$-Brauer pair. By $\calA$ we denote the fusion system of $A$ associated with $(D,e_D)$ and by $\calB$ we denote the fusion system of $B$ associated with $(E,f_E)$.

\smallskip
In Theorem~\ref{thm isomorphic fusion systems} we show that $\phi$ is an isomorphism between the fusion systems $\calB$ and $\calA$. And in Theorem~\ref{thm local ppeqs} we show that Brauer constructions with respect to subgroups of $\Delta(D,\phi,E)$ applied to $\gamma$ yields again local $p$-permutation equivalences at various levels. Finally in Theorem~\ref{thm finiteness} we show that there can only be finitely many $p$-permutation equivalences between given blocks $A$ and $B$.

\begin{proposition}\label{prop isomorphic inertia quotients}
Let $(\Delta(P,\psi,Q), e\otimes f^*)$ be a $\gamma$-Brauer pair and set $I:=N_G(P,e)$, $J:=N_H(Q,f)$. For every $g\in I$ there exists a unique element $hC_H(Q)\in J/C_H(Q)$ such that $c_g\circ \psi = \psi \circ c_h \colon Q\myiso P$. Similarly, for every $h\in J$ there exists a unique element $gC_G(P)\in I/C_G(P)$ such that $c_g\circ \psi = \psi \circ c_h$. These associations define mutually inverse group isomorphisms between $I/C_G(P)$ and $J/C_H(Q)$. The isomorphism $J/C_H(Q)\myiso I/C_G(P)$ restricts to the isomorphism $QC_H(Q)/C_H(Q)\myiso PC_G(P)/P$, $hC_H(Q)\mapsto \psi(h)C_G(P)$, for $h\in Q$. The group $Y:=N_{G\times H}(\Delta(P,\psi,Q), e\otimes f^*)$ satisfies $p_1(Y)=I$, $p_2(Y)=J$, $k_1(Y)=C_G(P)$, $k_2(Y)=C_H(Q)$ and the resulting isomorphism $\eta_Y\colon J/C_H(Q)\myiso I/C_G(P)$ from \ref{noth dir prod}(a) is equal to the one described above.
\end{proposition}

\begin{proof}
Let $g\in I$. With $(\Delta(P,\psi,Q),e\otimes f^*)$ also $\lexp{(g,1)}{(\Delta(P,\psi,Q),e\otimes f^*)} = (\Delta(P, c_g\circ \psi, Q), e \otimes f^*)$ is a $\gamma$-Brauer pair. Now, the uniqueness statement in Corollary~\ref{cor connecting Brauer pairs} implies that there exists an element $h\in H$ such that $\lexp{h}{(c_g\circ\psi, (Q,f))}=(\psi, (Q,f))$. This implies that $h\in J$ and that $c_g\circ\psi = \psi\circ c_h$. Clearly, $h\in J$ is uniquely determined up to multiplication with elements of $C_H(Q)\trianglelefteq J$ by this condition. It is easy to see that this defines a group homomorphism from $I/C_G(P)$ to $J/C_H(Q)$. A similar argument, now using Lemma~\ref{lem connecting Brauer pairs rl}, implies the second statement and defines a group homomorphism from $J/C_H(Q)$ to $I/C_G(P)$. Clearly, these two homomorphisms are mutually inverses. The following statement is clear, since $\psi c_h \psi^{-1}=c_{\psi(h)}$. The last statements about $Y$ follow immediately from the above and Propostion~\ref{prop Nphi}(c).
\end{proof}

Note that for the following theorem we only assume that $\gamma\in T^\Delta_l(A,B)$. 

\begin{theorem}\label{thm isomorphic fusion systems}
The isomorphism $\phi\colon E\myiso D$ is an isomorphism between the fusion systems $\calB$ and $\calA$.
\end{theorem}

\begin{proof}
Using Alperin's fusion theorem, see \cite[Theorem~I.3.6]{AKO2011}, it suffices to show that for every subgroup $Q\le E$ with $P:=\phi(Q)\le D$ and $\psi:=\phi|_Q\colon Q\myiso P$, one has $\psi^{-1}\circ\Hom_{\calA}(P,P)\circ \psi = \Hom_{\calB}(Q,Q)$, an equation of sets of automorphisms of $Q$. Let $e_P\in Z(\calO[C_G(P)])$ and $f_Q\in Z(\calO[C_H(Q)])$ be the unique primitive idempotents such that $(P,e_P)\le (D,e_D)$ and $(Q,f_Q)\le (E,f_E)$, and set $I:=N_G(P,e_P)$ and $J:=N_H(Q,f_Q)$. Note that $I/C_G(P)\to\Hom_{\calA}(P,P)$, $gC_G(P)\mapsto c_g$, and $J/C_H(Q)\to \Hom_{\calB}(Q,Q)$, $hC_H(Q)\mapsto c_h$, are group isomorphisms. Now the claim follows immediately from Proposition~\ref{prop isomorphic inertia quotients}, noting that $(\Delta(P,\psi,Q), e _P\otimes f_Q^*)\le (\Delta(D,\phi,E), e_D\otimes f_E^*)$ , by Remark~\ref{rem emuf}(c), so that also $(\Delta(P,\psi,Q), e_P\otimes f_Q^*)$ is a $\gamma$-Brauer pair by Theorem~\ref{thm gamma is uniform}(a).
\end{proof}

We can now improve the formulation of Lemma~\ref{lem unique Lambdahat lr}. Recall the definition of $\mu(P,\phi,Q)$ from \ref{not gamma mu nu}.

\begin{lemma}\label{lem unique Lambdahat lr 2}
Let $(\Delta(P,\psi,Q),e\otimes f^*)\in \calBP_\calO^\Delta(\gamma)$ and set $I:=N_G(P,e)$, $J:=N_H(Q,f)$, $X:=N_{I\times I}(\Delta(P))\le G\times G$ and $Y:=N_{I\times J}(\Delta(P,\psi,Q))\le G\times H$. Then $X*Y=Y$ and for every $\chi\in\Irr(\KK X(e\otimes e^*))$ one has
\begin{equation*}
  \chi\mathop{\cdot}\limits_{G}^{X,Y} e\mu(P,\psi,Q)f\in\pm\Irr(\KK Y(e\otimes f^*))\,.
\end{equation*}
\end{lemma}

\begin{proof}
By Proposition~\ref{prop isomorphic inertia quotients}, we have $p_1(Y)=I$. Therefore, Lemma~\ref{lem comp form}(c) implies that $X*Y=Y$. By Lemma~\ref{lem unique Lambdahat lr} there exists a $B$-Brauer pair $(Q',f')$ and an isomorphism $\psi'\colon Q'\myiso P$ such that $0\neq \chi\mathop{\cdot}\limits_{G}^{X,Y'} e\mu(P,\psi',Q')f'\in R(\KK Y(e\otimes {f'}^*))$, where $Y':=N_{I\times J'}(\Delta(P,\psi', Q'))$ with $J':=N_H(Q',f')$. Therefore, by Proposition~\ref{prop equiv gamma Brpair cond}, $(\Delta(P,\psi',Q'), e\otimes{f'}^*)$ is a $\gamma$-Brauer pair. By Lemma~\ref{lem unique Lambdatilde lr} there exists $h\in H$ such that $(\Delta(P,\psi,Q),e\otimes f^*)=\lexp{(1,h)}{(\Delta(P,\psi',Q'), e\otimes {f'}^*)}$. This, together with Equation~(\ref{eqn gen tens and conj}), implies
\begin{equation*}
  \chi\mathop{\cdot}\limits_{G}^{X,Y} e\mu(P,\psi,Q)f = 
  \chi\mathop{\cdot}\limits_{G}^{X,Y} \lexp{(1,h)}{(e\mu(P,\psi',Q')f')} =
  \lexp{(1,h)}{\bigl(\chi\mathop{\cdot}\limits_{G}^{X,Y'} e\mu(P,\psi',Q')f'\bigr)} \neq 0
\end{equation*}
in $R(\KK X(e\otimes f^*))$, since $\lexp{(1,h)}{Y'}=Y$. By Lemma~\ref{lem unique Lambdahat lr}, this virtual character belongs to $\pm\Irr(\KK Y(e\otimes f^*))$.
\end{proof}

\begin{theorem}\label{thm local ppeqs}
Let $(\Delta(P,\psi,Q),e\otimes f^*)\in\calBP^\Delta_\calO(\gamma)$ be a $\gamma$-Brauer pair and set $I:=N_G(P,e)$ and $J:=N_H(Q,f)$. Suppose that $C_G(P)\le S\le I$ and $C_H(Q)\le T\le J$ are intermediate groups related via the isomorphism in Proposition~\ref{prop isomorphic inertia quotients} and set $Y:=N_{S\times H}(\Delta(P,\psi,Q), e\otimes f^*) = N_{S\times T}(\Delta(P,\psi,Q))$. Then, the element $e\gamma(P,\psi,Q)f\in T(\calO Y(e\otimes f^*))$, defined as in \ref{not gamma mu nu} and restricted to $Y$, satisfies
\begin{equation}\label{eqn Delta(P)}
  [\calO[C_G(P)]e] = e\gamma(P,\psi,Q)f \mathop{\cdot}\limits^{Y,Y^\circ}_{H} f\gamma(P,\psi,Q)^\circ e \quad 
  \text{in $T(\calO[N_{S\times S}(\Delta(P))](e\otimes e^*))$.}
\end{equation}
Moreover, the element
\begin{equation*}
  \gammatilde:= \ind_{Y}^{S\times T}\bigl(e\gamma(P,\psi,Q)f\bigr) \in 
  T^\Delta(\calO Se,\calO Tf)
\end{equation*}
satisfies $\gammatilde\mathop{\cdot}\limits_{T}{\gammatilde}^\circ = [\calO Se]$ in $T^\Delta(\calO Se,\calO Se)$.
In particular, if $\gamma$ is a $p$-permutation equivalence between $A$ and $B$ then $\gammatilde$ is a $p$-permutation equivalence between $\calO Se$ and $\calO Tf$.
\end{theorem}

\begin{proof}
We apply the Brauer construction with respect to $\Delta(P)$ to Equation~(\ref{eqn left ppeq}) and obtain after restriction
\begin{equation*}
  [\calO[C_G(P)]e] = e(\gamma\cdotH \gamma^\circ)(\Delta(P))e \quad 
  \text{in $T(\calO[N_{S\times S}(\Delta(P))](e\otimes e^*))$.}
\end{equation*}
Next we apply Theorem~\ref{thm BP decomp}(d) to the right hand side of the last equation. Note that $p_1(Y)=S$ and $p_2(Y)=T$ by Propositions~\ref{prop Nphi}(c) and \ref{prop isomorphic inertia quotients}. Thus, $Y*Y^\circ = \Delta(S)(C_G(P)\times \{1\}) = N_{S\times S}(\Delta(P))$ by Lemma~\ref{lem comp form} and Proposition~\ref{prop Nphi}. By Corollary~\ref{cor connecting Brauer pairs}, the pair $(\psi,(Q,f))$ is, up to $H$-conjugation, the only pair such that $e\gamma(P,\psi,Q)f\neq 0$. Thus, Theorem~\ref{thm BP decomp}(d) implies Equation~(\ref{eqn Delta(P)}).

\smallskip
Next we apply $\ind_{N_{S\times S}(\Delta(P))}^{S\times S}$ to both sides of Equation~(\ref{eqn Delta(P)}). First we show that the left hand side yields $[\calO Se]\in T^\Delta(\calO Se,\calO Se)$. In fact, $\calO[C_G(P)]$ is an $\calO[N_{S\times S}(\Delta(P))]$-permutation module with $\Delta(S)$ as stabilizer of the standard basis element $1\in C_G(P)$. Thus, $\calO[C_G(P)]\cong \Ind_{\Delta(S)}^{N_{S\times S}(\Delta(P))}(\calO)$ in $T(\calO[N_{S\times S}(\Delta(P))])$ and \ref{noth bimodules}(f), applied to the central idempotent $e\otimes e^*$ of $\calO[N_{S\times S}(\Delta(P))]$, implies that \begin{equation*}
  \Ind_{N_{S\times S}(\Delta(P))}^{S\times S} (\calO[C_G(P)]e) 
  \cong e\bigl(\Ind_{N_{S\times S}(\Delta(P))}^{S\times S}(\calO[C_G(P)])\bigr)e
  \cong e\,\Ind_{\Delta(S)}^{S\times S}(\calO) e \cong e\calO S e = \calO Se
\end{equation*}
as $(\calO Se,\calO Se)$-bimodules. Next we show that applying $\ind_{N_{S\times S}(\Delta(P))}^{S\times S}$ to the right hand side of Equation~(\ref{eqn Delta(P)}) yields
\begin{equation}\label{eqn S T expression}
  \ind_Y^{S\times T}(e\gamma(P,\psi,Q)f) \mathop{\cdot}\limits_T 
  \bigl( \ind_{Y}^{S\times T}(e\gamma(P,\psi,Q) f)\bigr)^\circ
\end{equation}
In fact, rewriting $\bigl(\ind_{Y}^{S\times T}(e\gamma(P,\psi,Q) f)\bigr)^\circ = \ind_{Y^\circ}^{T\times S}(f\gamma(P,\psi,Q)^\circ e)$ in the above expression, then applying the Mackey formula in Theorem~\ref{thm Bouc-Mackey} to the resulting expression and noting that $p_2(Y)=T$ and $Y*Y^\circ = N_{S\times S}(\Delta(P))$, one obtains that the expression in (\ref{eqn S T expression}) is equal to
\begin{equation*}
  \ind_{N_{S\times S}(\Delta(P))}^{S\times S}
  \bigl(e\gamma(P,\psi,Q)f \mathop{\cdot}\limits^{Y,Y^\circ}_{H} f\gamma(P,\psi,Q)^\circ e\bigr)
\end{equation*}
in $T(\calO Se,\calO Se)$ as desired. 
\end{proof}

The following proposition investigates maximal Brauer pairs of the local $p$-permutation equivalences from Theorem~\ref{thm local ppeqs} for the two extreme choices of $S$ and $T$. It will use the notions introduced in \ref{noth fusion system of B} and the results from Proposition~\ref{prop fully centralized and centric}.

\begin{proposition}\label{prop local maximal Brauer pair}
For any subgroup $P$ of $D$ let $(P,e_P)$ denote the unique $A$-Brauer pair with $(P,e_P)\le (D,e_D)$ and for any subgroup $Q$ of $E$ let $(Q,f_Q)$ denote the unique $B$-Brauer pair with $(Q,f_Q)\le (E,f_E)$. Let $Q\le E$ and set $P:=\phi(Q)$. Then $(\Delta(P,\phi,Q), e_P\otimes f_Q^*)$ is a $\gamma$-Brauer pair (cf.~Remark~\ref{rem emuf}(c) and Theorem~\ref{thm gamma is uniform}(a)).

\smallskip
{\rm (a)} Set $\gamma':=e_P\gamma(P,\phi,Q)f_Q\in T^\Delta_l(\calO[C_G(P)]e_P,\calO[C_H(Q)]f_Q)$. The $\calO[C_G(P)]e_P\otimes \calO[C_H(Q)]f_Q^*$-Brauer pair $(\Delta(C_D(P),\phi,C_E(Q)), e_{PC_D(P)}\otimes f_{QC_E(Q)}^*)$ is a $\gamma'$-Brauer pair. It is a maximal $\gamma'$-Brauer pair if and only if $P$ is fully $\calA$-centralized. In particular, $(\Delta(Z(P),\phi,Z(Q)),e_P\otimes f_Q^*)$ is a maximal $\gamma'$-Brauer pair if and only if $P$ is $\calA$-centric.

\smallskip
{\rm (b)} Set $I:=N_G(P,e_P)$, $J:=N_H(Q,f_Q)$ and $\gamma'':=\ind_{N_{I\times J}(\Delta(P,\phi,Q))}^{I\times J}(e_P\gamma(P,\phi,Q)f_Q)\in T^\Delta_l(\calO I e_P,\calO J f_Q)$. The $\calO Ie_P\otimes \calO Jf_Q^*$-Brauer pair $(\Delta(N_D(P),\phi,N_E(Q)), e_{N_D(P)}\otimes f_{N_E(Q)}^*)$ is a $\gamma''$-Brauer pair. It is a maximal $\gamma''$-Brauer pair if and only if $P$ is fully $\calA$-normalized.
\end{proposition}

\begin{proof}
By Theorem~\ref{thm local ppeqs}, we have $\gamma'\in T^\Delta_l(\calO[C_G(P)]e_P,\calO[C_H(Q)]f_Q)$ and $\gamma''\in T^\Delta_l(\calO I e_P,\calO J f_Q)$, so that we can apply results from Section~\ref{sec Brauer pairs of ppeqs} to $\gamma'$ and to $\gamma''$.

\smallskip
(a) First note that $(\Delta(C_D(P),\phi,C_E(Q)), e_{PC_D(P)}\otimes f_{QC_E(Q)}^*)$ is an $\calO[C_G(P)]e_P\otimes \calO[C_H(Q)]f_Q^*$-Brauer pair, since $C_{C_G(P)}(C_D(P))=C_G(P)\cap C_G(C_D(P))=C_G(PC_D(P))$ and $C_{C_H(Q)}(C_E(Q)) = C_H(QC_E(Q))$. It is a $\gamma'$-Brauer pair, since 
\begin{equation*}
  e_{PC_D(P)}\gamma'(\Delta(C_D(P),\phi,C_E(Q))) f_{QC_E(Q)} =  
  e_{PC_D(P)}\gamma(\Delta(PC_D(P),\phi,EC_E(Q))) f_{QC_E(Q)}\neq 0
\end{equation*}
in $T^\Delta(\calO[C_G(PC_D(P))]e_{PC_D(P)},\calO[C_H(QC_E(Q))]f_{QC_E(Q)})$ by Proposition~\ref{prop more on p-perm}(b) and Lemma~\ref{lem Brauer map compatibility}, and since $(\Delta(PC_D(P),\phi,EC_E(Q)), e_{PC_D(P)}\otimes f_{QC_E(Q)}^*)$ is a $\gamma$-Brauer pair by Theorem~\ref{thm gamma is uniform}(a) and Proposition~\ref{prop equiv gamma Brpair cond}. By Theorem~\ref{thm gamma is uniform}(c), $(\Delta(C_D(P),\phi,C_E(Q)), e_{PC_D(P)}\otimes f_{QC_E(Q)}^*)$ is a maximal $\gamma'$-Brauer pair if and only if $C_D(P)$ is a defect group of $\calO[C_G(P)]e_P$. But, by Proposition~\ref{prop fully centralized and centric}(a), this is equivalent to $P$ being fully $\calA$-centralized.

\smallskip
(b) Note that $(\Delta(N_D(P),\phi, N_E(Q)), e_{N_D(P)}\otimes f_{N_E(Q)}^*)$ is an $\calO Ie_P\otimes \calO Jf_Q^*$-Brauer pair, since $C_{I}(N_D(P))=C_G(N_D(P))$ and $C_J(N_E(Q))=C_G(N_E(Q))$. It is a straightforward verification that the $p$-subgroups $\Delta(P,\phi,Q)$ and $\Delta(N_D(P),\phi, N_E(Q))$ of $I\times J$ satisfy the hypothesis of Lemma~\ref{lem Brauer con and Ind 2} (in place of the subgroups $Q$ and $P$ of $G$), since every indecomposable $\calO [N_{I\times J}(\Delta(P,\phi,Q))]$-module appearing in $e_P\gamma(P,\phi,Q)f_Q$ has a vertex contained in $\Delta(D,\phi,E)$ by Lemma~\ref{lem M Bp implies gamma Bp}, Proposition~\ref{prop equiv gamma Brpair cond}, and Theorem~\ref{thm gamma is uniform}. 
Thus, Lemma~\ref{lem Brauer con and Ind 2} together with Lemma~\ref{lem Brauer map compatibility} implies
\begin{equation*}
  e_{N_D(P)}\gamma''(N_D(P),\phi,N_E(Q)) f_{N_E(Q)} = 
  e_{N_D(P)}\gamma(N_D(P),\phi,N_E(Q))f_{N_E(Q)}
\end{equation*}
in $T^\Delta\bigl(\calO[C_G(N_D(P))]e_{N_D(P)},\calO[C_H(N_E(Q))]f_{N_H(Q)}\bigr)$. Since $(\Delta(N_D(P),\phi, N_E(Q)), e_{N_D(P)}\otimes f_{N_E(Q)}^*)$ is a $\gamma$-Brauer pair by Theorem~\ref{thm gamma is uniform}(a), the element in the above equation is non-zero. This implies that $(\Delta(N_D(P),\phi, N_E(Q)), e_{N_D(P)}\otimes f_{N_E(Q)}^*)$ is a $\gamma''$-Brauer pair. By Theorem~\ref{thm gamma is uniform}(c), it is a maximal $\gamma''$-Brauer pair if and only if $N_D(P)$ is a defect group of  $\calO I e_P$. But, by Proposition~\ref{prop fully centralized and centric}(b), this is equivalent to $P$ being fully $\calA$-normalized.
\end{proof}

\begin{nothing}\label{noth Schur} {\it Schur classes.}\quad
Let $\kk$ be a field. Recall from the proof of \cite[Theorem~3.5.7]{NagaoTsushima1989} that whenever $N\trianglelefteq G$ and $V$ is an irreducible $G$-stable $\kk N$-module which is {\em $\kk$-split}, i.e., $\End_{\kk N}(V) = \kk$, Schur assigned to these data a canonical cohomology class $\kappa\in H^2(G/N,\kk^\times)$. This construction has the following properties: 

(a) $V$ extends to a $\kk G$-module if and only if $\kappa=1$;

(b) The class assigned to $V^\circ$ is $\kappa^{-1}$;

(c) If $N\le H\le G$ then the canonical class assigned to $N\trianglelefteq H$ and $V$ is $\res^{G/N}_{H/N}(\kappa)$; and

(d) If $\kappa_i$ is the class assigned to $N_i\trianglelefteq G_i$ and $V_i$, for $i=1,2$, then $\kappa_1\times \kappa_2\in H^2((G_1\times G_1)/(N_1\times N_2),F^\times)$ is the canonical class assigned to $N_1\times N_2\trianglelefteq G_1\times G_2$ and $V_1\otimes_{\kk} V_2$.
\end{nothing}

\begin{lemma}\label{lem Schur for Y}
Let $\kk$ be a field, let $G$ and $H$  be finite groups, and let $Y\le G\times H$ be such that $p_1(Y)=G$ and $p_2(Y)=H$. Set $M:=k_1(Y)\trianglelefteq G$, $N:=k_2(Y)\trianglelefteq H$ and let $\eta_Y\colon H/N\myiso G/M$ (see~\ref{noth dir prod}(a)) be the isomorphism induced by $Y$. Suppose that $V\in\lmod{\kk M}$ is irreducible, $G$-stable, and $\kk$-split, and that $W\in\lmod{\kk N}$ is irreducible, $H$-stable and $\kk$-split. Denote by $\kappa\in H^2(G/M,\kk^\times)$ and $\lambda\in H^2(H/N,\kk^\times)$ their respective Schur classes. Suppose further that there exists $U\in\lmod{\kk Y}$ with
$\Res^Y_{M\times N}(U)=V\otimes_{\kk}W^\circ$. Then
\begin{equation*}
  \lambda=\eta_Y^*(\kappa)\in H^2(H/N,\kk^\times)\,.
\end{equation*}
\end{lemma}

\begin{proof}
The Schur class of $V\otimes_{\kk} W^\circ$ with respect to $M\times N\trianglelefteq Y$ is $\res^{(G\times H)/(M\times N)}_{Y/(M\times N)}(\kappa\times \lambda^{-1})$, by \ref{noth Schur}(b), (c), and (d). Since $V\otimes_{\kk} W^\circ$ extends to $Y$, the latter class is trivial. Let $\pbar_2\colon Y/(M\times N)\myiso H/N$ be the isomorphism induced by the projection $p_2\colon G\times H\to H$. Then also $(\pbar_2^{-1})^*(\res^{(G\times H)/(M\times N)}_{Y/(M\times N)}(\kappa\times\lambda^{-1}))=1$ in $H^2(H/N,\kk^\times)$. However, a straightforward cocycle computation shows that $(\pbar_2^{-1})^*(\res^{(G\times H)/(M\times N)}_{Y/(M\times N)}(\kappa\times\lambda^{-1}))=\eta^*(\kappa)\cdot\lambda^{-1}$, and the proof is complete.
\end{proof}

The element $\mu$ in the next proposition will be specified in the follow-up proposition to $e\mu(P,\phi,Q)f$ with the notation as in \ref{not gamma mu nu}.

\begin{proposition}\label{prop diagonal extension}
Let $G$ and $H$ be finite groups, $Y\le G\times H$ with $p_1(Y)=G$ and $p_2(Y)=H$. Set $M:=k_1(Y)\trianglelefteq G$ and $N:=k_2(Y)\trianglelefteq H$, and let $\eta_Y\colon H/N\myiso G/M$ denote the isomorphism induced by $Y$. Suppose that $e$ is a $G$-stable idempotent of $Z(\KK M)$, $f$ is an $H$-stable idempotent of $Z(\KK N)$, and that $\mu\in R(\KK Y(e\otimes f^*))$ is a virtual character such that $\res^Y_{M\times N}(\mu)$ is an isometry between $\KK Me$ and $\KK Nf$.

\smallskip
{\rm (a)} If $\alpha\colon \Irr(\KK Nf)\myiso\Irr(\KK Me)$ denotes the bijection induced by $\res^Y_{M\times N}(\mu)$ then $\alpha(\lexp{hN}{\chi}) = \lexp{\eta_Y(hN)}{\alpha(\chi)}$, for all $\chi\in\Irr(\KK Nf)$ and all $h\in H$. In particular, $\alpha$ induces a bijection $\alphabar\colon\Irr(\KK Nf)/H\myiso \Irr(\KK Me)/G$.

\smallskip
{\rm (b)} The group homomorphisms $\mu\dott{Y}{H}{H} -$ and $\ind_Y^{G\times H}(\mu)\cdotH -$ from $R(\KK Hf)$ to $R(\KK Ge)$ coincide.

\smallskip
{\rm (c)} Set $X_1:=(M\times M)\Delta(G)$ and $X_2:=(N\times N)\Delta(H)$. Then $Y*Y^\circ=X_1$ and $Y^\circ*Y=X_2$. Moreover, the following are equivalent:

\smallskip
\qquad {\rm (i)} $\mu\dott{Y}{Y^\circ}{H} \mu^\circ = [\KK M e]$ in $R(\KK X_1(e\otimes e^*))$.

\smallskip
\qquad {\rm (ii)} $\mu^\circ\dott{Y^\circ}{Y}{G} \mu = [\KK Nf]$ in $R(\KK X_2(f\otimes f^*))$.

\smallskip
\qquad {\rm (iii)} $(\mu,\mu)_Y = |\Irr(\KK Me)/G|$.

\smallskip
\qquad {\rm (iv)} $\mu$ has precisely $|\Irr(\KK Me)/G|$ distinct irreducible constituents and each of them occurs with multiplicity $\pm 1$.

\smallskip
{\rm (d)} Assume that the equivalent conditions (i)--(iv) in Part~(c) hold and let $\calJ$ be a set of representatives of the $H$-orbits of $\Irr(\KK Nf)$. Then there exist irreducible characters $\mu_\chi\in\Irr(\KK Y(e\otimes f^*))$ and signs $\varepsilon_\chi\in\{\pm 1\}$, for $\chi\in\calJ$, such that $\mu=\sum_{\chi\in\calJ}\varepsilon_\chi\cdot \mu_\chi$ and $\res^Y_{M\times N}(\mu_\chi)=\sum_{\chi'\in[\chi]_H} \alpha(\chi')\times {\chi'}^\circ$ for all $\chi\in\calJ$.

\smallskip
{\rm (e)} Assume that the equivalent conditions (i)--(iv) in Part~(c) hold. Let $\chi\in\Irr(\KK N f)$ and set $\zeta:=\alpha(\chi)\in\Irr(\KK M e)$, $G_1:=\stab_G(\zeta)$, and $H_1:=\stab_H(\chi)$. Then, $\eta_Y(H_1)=G_1$ by Part~(a). If $\kappa\in H^2(G_1/M,\KK^\times)$ and $\lambda\in H^2(H/N,\KK^\times)$ denote the respective extension classes of $\zeta$ and $\chi$ then $\kappa=\eta_Y^*(\lambda)$.

\smallskip
{\rm (f)} Assume that the equivalent conditions (i)--(iv) in Part~(c) hold. Then $\ind_Y^{G\times H}(\mu)\in R(\KK Ge,\KK Hf)$ is an isometry between $\KK Ge$ and $\KK Hf$.
\end{proposition}

\begin{proof}
(a) By Remark~\ref{rem isometry}(a) we can write
\begin{equation}\label{eqn res mu}
  \res^Y_{M\times N}(\mu)=\sum_{\chi\in\Irr(\KK Nf)} \varepsilon_\chi\cdot \alpha(\chi)\times \chi^\circ
\end{equation}
with $\varepsilon_\chi\in\{\pm 1\}$, for $\chi\in\Irr(\KK Nf)$. Let $(g,h)\in Y$. Then $\lexp{(g,h)}{\mu}=\mu$ and therefore, for every $\chi\in\Irr(\KK Nf)$, $\varepsilon_\chi\cdot \lexp{g}{\alpha}(\chi)\times\lexp{h}{\chi^\circ}$ is again equal to one of the summands in the above sum. This implies $\lexp{g}{\alpha}(\chi)=\alpha(\lexp{h}{\chi})$ and $\varepsilon_\chi=\varepsilon_{\lexp{h}{\chi}}$, and Part~(a) is proved. For later use, let $\calJ\subseteq\Irr(\KK Nf)$ denote a set of representatives of the $H$-orbits of $\Irr(\KK Nf)$. Then the above also shows that $\chi\mapsto \alpha(\chi)\times \chi^\circ$ induces a bijection
\begin{equation}\label{eqn H-orbit Y-orbit bijection}
   \calJ\myiso \{\alpha(\chi)\times \chi^\circ \mid \chi\in\Irr(\KK Nf)\}/Y\,.
\end{equation}

\smallskip
(b) Note that identifying $H$ with $H\times \{1\}$ and $G$ with $G\times \{1\}$, the first map, given by the generalized tensor product, maps $R(\KK Hf)$ to $R(\KK Ge)$, since $Y*(H\times\{1\})=G\times\{1\}$. The statement follows directly from Lemma~\ref{lem ext tp and ind}(a) applied to $Y\le G\times H$ and $H\times\{1\}\le H\times \{1\}$.

\smallskip
(c) We will show that (i)$\Rightarrow$(iii)$\Rightarrow$(iv)$\Rightarrow$(i). Then also (ii)$\iff$(iii) by symmetry, since $(\mu^\circ,\mu^\circ)_{Y^\circ} = (\mu,\mu)_Y$ and $|\Irr(\KK Ge)/G|=|\Irr(\KK Nf)/H|$, by Part~(a). But first we establish some facts that hold without further hypotheses. Note that  $Y*Y^\circ=X_1$ and $Y^\circ*Y=X_2$ by Lemma~\ref{lem comp form}(a), since $p_1(Y)=G$ and $p_2(Y)=H$. Set $\theta_1:=[\KK Me]\in R(\KK X_1(e\otimes e^*))$. Then Proposition~\ref{prop ext hom duals}(c) and Corollary~\ref{cor ext tp and hom}(b) imply
\begin{equation}\label{eqn (mu,mu)}
  (\mu,\mu)_Y = (\mu,\theta_1\dott{X_1}{Y}{G}\mu)_Y = (\mu\dott{Y}{Y^\circ}{H}\mu^\circ,\theta_1)_{X_1}\,.
\end{equation}

Now write $\mu=a_1\mu_1+\cdots+a_r\mu_r$ with pairwise distinct $\mu_1,\ldots,\mu_r\in \Irr(\KK Y)$ and non-zero integers $a_1,\ldots, a_r$. Since $\res^Y_{M\times N}(\mu_i)$ is a multiple of the sum of a $Y$-orbit of $\Irr(\KK[M\times N])$, (\ref{eqn res mu}) and (\ref{eqn H-orbit Y-orbit bijection}) imply that
\begin{equation}\label{eqn (mu,mu) 2}
  (\mu,\mu)_Y = a_1^2+\cdots+a_r^2 \ge r \ge |\calJ| = |\Irr(\KK Nf)/H|\,.
\end{equation}
Moreover, one has $(\mu,\mu)_Y = |\Irr(\KK Nf)/H|$, if and only if there is a bijection $\{1,\ldots,r\}\myiso \calJ$, $i\mapsto \chi_i$, such that 
\begin{equation}\label{eqn for (d)}
   a_i= \varepsilon_{\chi_i}\quad \text{and} \quad 
   \res^Y_{M\times N}(\mu_i)=\sum_{\chi'\in[\chi_i]_H} \alpha(\chi')\times {\chi'}^\circ\,,
\end{equation}
for all $i=1,\ldots, r$, where $[\chi]_H$ denotes the $H$-orbit of $\chi\in\Irr(\KK Nf)$. This shows that $(iii)$ implies $(iv)$, since $|\Irr(\KK Ge)/G|=|\Irr(\KK Hf)/H|$ by Part~(a).

The same considerations apply to $\theta_1$ with $\res^{X_1}_{M\times M}(\theta_1)=\sum_{\zeta\in\Irr(\KK Me)} \zeta\times\zeta^\circ$. Since $\theta_1$ is the character of a $\KK X_1$-module, each $G$-orbit sum $\sum_{\zeta'} \zeta'\times {\zeta'}^\circ$ extends to an irreducible character and we obtain
\begin{equation}\label{eqn theta1}
  (\theta_1 , \theta_1)_{X_1} = |\Irr(\KK Me)/G|\,.
\end{equation}

Now, (i) implies (iii) by Equations~(\ref{eqn (mu,mu)}) and (\ref{eqn theta1}).

Finally, we assume (iv) and aim to show (i). Since (iv) implies $(\mu,\mu)_Y=|\calJ|$, we obtain (\ref{eqn for (d)}) and can write  $\mu=\sum_{\chi\in\calJ} \varepsilon_\chi \mu_\chi$, with $\mu_\chi\in\Irr(\KK Y)$ satisfying $\res^Y_{M\times N} (\mu_\chi)= \sum_{\chi'\in[\chi]_H} \alpha(\chi')\times{\chi'}^\circ$, for all $\chi\in\calJ$. If $\chi_1,\chi_2\in\calJ$ are distinct then $\mu_{\chi_1}\dott{Y}{Y^\circ}{H}\mu_{\chi_2}^\circ = 0$, by the definition of the extended tensor product (taken over $\KK N$ in this case). Thus
\begin{equation*}
  \mu\dott{Y}{Y^\circ}{H}\mu^\circ 
  = \sum_{\chi\in\calJ} \varepsilon_\chi \mu_\chi\dott{Y}{Y^\circ}{H} \varepsilon_\chi \mu_\chi^\circ
  = \sum_{\chi\in\calJ} \mu_\chi\dott{Y}{Y^\circ}{H} \mu_\chi^\circ\,.
\end{equation*}
Moreover, for each $\chi\in\calJ$, the character $\mu_\chi\dott{Y}{Y^\circ}{H} \mu_\chi^\circ$ of $X_1$ is irreducible, since its restriction to $M\times M$ is the sum $\sum_{\zeta'\in[\alpha(\chi)]_G} \zeta'\times {\zeta'}^\circ$ is the sum of an $X_1$-orbit of elements in $\Irr(\KK[M\times M])$. Furthermore, the irreducible characters $\mu_\chi\dott{Y}{Y^\circ}{H} \mu_\chi^\circ$, $\chi\in\calJ$, are pairwise orthogonal, since their restrictions to $M\times M$ are. Thus, both $\mu\dott{Y}{Y^\circ}{H}\mu^\circ$ and $\theta_1$ are multiplicity-free sums of $|\calJ|$ pairwise distinct irreducible characters of $X_1$. But by Equation~(\ref{eqn (mu,mu)}) and by (iv) we have  $(\mu\dott{Y}{Y^\circ}{H}\mu^\circ,\theta_1)_{X_1}=(\mu,\mu)_Y=|\calJ|$. This implies $\mu\dott{Y}{Y^\circ}{H}\mu^\circ = \theta_1$ and the proof of Part~(c) is complete.

\smallskip
(d) This was shown in the proof of Part~(c).

\smallskip
(e) This follows immediately from Part~(c) and Lemma~\ref{lem Schur for Y}.

\smallskip
(f) By Theorem~\ref{thm Bouc-Mackey} we have
\begin{equation}\label{eqn ind and mu}
   (\ind_Y^{G\times H}(\mu))\cdotH(\ind_Y^{G\times H}(\mu))^\circ =
   (\ind_Y^{G\times H}(\mu))\cdotH(\ind_{Y^\circ}^{H\times G}(\mu^\circ)) =
   \ind_{Y*Y^\circ}^{G\times G}(\mu\dott{Y}{Y^\circ}{H}\mu^\circ)\,,
\end{equation}
since $p_2(Y)=p_1(Y^\circ)=H$. Moreover,  $\mu\dott{Y}{Y^\circ}{H}\mu^\circ=[\KK M e]$ in $R(\KK X_1(e\otimes e^*))$ by hypothesis. Thus, the last expression in (\ref{eqn ind and mu}) equals $\ind_{X_1}^{G\times G}([\KK Me])$ which in turn equals $[\KK Ge]\in R(\KK Ge,\KK Ge)$. Similarly one shows that 
\begin{equation*}
   (\ind_Y^{G\times H}(\mu))^\circ\cdotG(\ind_Y^{G\times H}(\mu)) = [\KK Hf]
\end{equation*}
in $R(\KK Hf,\KK Hf)$. Now Remark~\ref{rem isometry}(a) implies the result.
\end{proof}

\begin{proposition}\label{prop local characters}
Let $A$ be a block of $\calO G$, $B$ a block of $\calO H$, $\gamma \in T^\Delta_l(A,B)$, and let $(\Delta(P,\phi,Q),e\otimes f^*)\in\calBP_\calO(\gamma)$. Set $I:=N_G(P,e)$, $J:=N_H(Q,f)$, and $Y:=N_{I\times J}(\Delta(P,\phi,Q))$. Then the element $\mu:=e\mu(P,\phi,Q)f\in R(\KK Y(e\otimes f^*))$ defined in \ref{not gamma mu nu} has the following properties:

\smallskip
{\rm (a)} The restriction of $\mu$ to $C_G(P)\times C_H(Q)$ is a perfect isometry between $\KK[C_G(P)]e$ and $\KK[C_H(Q)]f$.

\smallskip
{\rm (b)} The conditions (i)--(iv) in Part~(c), and therefore Parts (d), (e), and (f) of Proposition~\ref{prop diagonal extension} hold for the groups $I$ and $J$, the subgroup $Y\le I\times J$, their normal subgroups $C_G(P)$ and $C_H(Q)$, and the character $\mu$. 
\end{proposition}

\begin{proof}
(a) This follows from Theorem~\ref{thm local ppeqs} and Proposition~\ref{prop p-perm implies perfect isom}.

\smallskip
(b) By Part~(a) and Proposition~\ref{prop isomorphic inertia quotients}, all the hypotheses of Proposition~\ref{prop diagonal extension} are satisfied. Moreover, the condition in Proposition~\ref{prop diagonal extension}(c)(i) holds by Equation~(\ref{eqn Delta(P)}) in Theorem~\ref{thm local ppeqs}. Thus, the proof is complete.
\end{proof}

As an immediate consequence of the above proposition, we obtain the following finiteness result.

\begin{theorem}\label{thm finiteness}
Let $A$ be a block of $\calO G$ and $B$ a block of $\calO H$. The set $T^\Delta_l(A,B)$ of left $p$-permutation equivalences between $A$ and $B$ is finite. In particular, the set $T_o^\Delta(A,B)$ of $p$-permutation equivalences between $A$ and $B$ is finite and the group $T_o^\Delta(A,A)$ of $p$-permutation self equivalences of $A$ is finite.
\end{theorem}

\begin{proof}
By Proposition~\ref{prop local characters}, for each $A\otimes B^*$-Brauer pair $(Y, e\otimes f^*)$, there are only finitely many choices for the virtual character of $\gamma(Y,e\otimes f^*)\in T(\calO Y(e\otimes f^*))$, since it has bounded norm by the condition in Proposition~\ref{prop diagonal extension}(c)(iii). On the other hand, Proposition~\ref{prop ghost group} implies that these characters determine $\gamma$.
\end{proof}

\section{A character theoretic criterion and moving from left to right}\label{sec char criterion}

Throughout this section we assume that $G$ and $H$ are finite groups and that the $p$-modular system $(\KK,\calO,F)$ is large enough for $G$ and $H$. Further, we assume that $A$ is a block algebra of $\calO G$ and that $B$ is a block algebra of $\calO H$.

\smallskip
In this section we will prove a character theoretic criterion for an element $\gamma\in T^\Delta(A,B)$ to be in $T^\Delta_l(A,B)$, see Theorem~\ref{thm char criterion}. Since one of the conditions in this criterion is symmetric, we can derive that $T^\Delta_l(A,B)=T^\Delta_o(A,B)=T^\Delta_r(A,B)$, see Theorem~\ref{thm left equals right}.

\begin{lemma}\label{lem fusion consequence}
Suppose that $(D,e_D)$ is a maximal $A$-Brauer pair, $(E,f_E)$ is a maximal $B$-Brauer pair, and $\phi\colon E\myiso D$ is an isomorphism between the fusion systems $\calB$ and $\calA$ associated with $(E,f_E)$ and $(D,e_D)$, respectively. Suppose further that one has $A$-Brauer pairs $(P,e)$ and $(R,d)$ and a $B$-Brauer pair $(Q,f)$ with $(P,e)\le (D,e_D)$, $(R,d)\le (D,e_D)$ and $(Q,f)\le (E,f_E)$, and isomorphisms $\alpha\colon Q\myiso P$ and $\beta\colon Q\myiso R$ such that 
\begin{gather*}
  (\Delta(P,\alpha,Q), e\otimes f^*)\le_{G\times H} (\Delta(D,\phi,E), e_D\otimes f_E)\rlap{\quad\text{and}\quad}\\
  (\Delta(R,\beta,Q), d\otimes f^*)\le_{G\times H} (\Delta(D,\phi,E), e_D\otimes f_E)\,.
\end{gather*}
Then $(\Delta(P,\alpha\beta^{-1},R),e\otimes d^*)=_{\{1\}\times G} (\Delta(P),e\otimes e^*)$ and there exists $g\in G$ such that $c_g\circ\alpha=\beta\colon Q\myiso P$ and $\lexp{g}{(P,e)} = (R,d)$.
\end{lemma}

\begin{proof}
By assumption there exist $(g_1,h_1),(g_2,h_2)\in G\times H$ such that 
\begin{gather*}
  \lexp{(g_1,h_1)}{(\Delta(P,\alpha,Q),e\otimes f^*)} \le (\Delta(D,\phi,E), e_D\otimes f_E^*)\rlap{\quad\text{and}}\\
  \lexp{(g_2,h_2)}{(\Delta(R,\beta,Q), d\otimes f^*)} \le (\Delta(D,\phi,E), e_D\otimes f_E^*)\,,
\end{gather*}
or equivalently,
\begin{gather*}
  \lexp{g_1}{(P,e)}\le (D,e_D)\,,\quad  \lexp{h_1}{(Q,f)}\le (E,f_E)\,, \quad 
  \lexp{g_2}{(R,d)}\le (D,e_D)\,, \quad \lexp{h_2}{(Q,f)}\le (E,f_E)\,, \\
  \lexp{g_1}P= \phi(\lexp{h_1}Q)\,,\quad 
  c_{g_1}\alpha c_{h_1^{-1}} = \phi\colon \lexp{h_1}{Q}\myiso\lexp{g_1}{P}\,,\quad
  \lexp{g_2}{R}=\phi(\lexp{h_2}{Q})\,,\quad
  c_{g_2}\beta c_{h_2^{-1}} = \phi \colon \lexp{h_2}{Q}\myiso\lexp{g_2}{R}\,.
\end{gather*}
Thus, $c_{h_2h_1^{-1}}\in\Hom_{\calB}(\lexp{h_1}{Q},\lexp{h_2}{Q})$. Since $\phi\colon E\myiso D$ is an isomorphism between $\calB$ and $\calA$, there exists $g\in G$ such that 
\begin{equation*}
  \lexp{gg_1}{(P,e)} = \lexp{g_2}{(R,d)}\quad\text{and}\quad 
  \phi c_{h_2h_1^{-1}} = c_g\phi\colon \lexp{h_1}{Q}\myiso \phi(\lexp{h_2}{Q}) = \lexp{g_2}{R}\,.
\end{equation*}
From this we obtain $\lexp{g_2^{-1}gg_1}{(P,e)} = (R,d)$, $c_{g_2^{-1}gg_1}\alpha = \beta\colon Q\myiso R$, which implies
\begin{equation*}
  \lexp{(1,g_1^{-1}g^{-1}g_2)}{\bigl(\Delta(P,\alpha\beta^{-1},R),e\otimes d^*\bigr)} = (\Delta(P),e\otimes e^*)\,,
\end{equation*}
and the proof of the lemma is complete.
\end{proof}

\begin{theorem}\label{thm char criterion}
Let $\gamma\in T^\Delta(A,B)$. Then $\gamma$ belongs to $T^\Delta_l(A,B)$ if and only if the following hold:

\smallskip
{\rm (i)} There exists a  maximal element $(\Delta(D,\phi,E), e_D\otimes f_E^*)$ in $\calBP_\calO(\gamma)$ such that $(D,e_D)$ is a maximal $A$-Brauer pair, $(E,f_E)$ is a maximal $B$-Brauer pair and $\phi\colon E\to D$ is an isomorphism between the fusion systems $\calB$ and $\calA$ of $B$ and $A$ associated with $(E,f_E)$ and $(D,e_D)$, respectively;

\smallskip
{\rm (ii)} Any two maximal elements in $\calBP_\calO(\gamma)$ are $G\times H$-conjugate; and

\smallskip
{\rm (iii)} For every $(\Delta(P,\psi,Q), e\otimes f^*)\in\calBP_\calO(A\otimes B^*)$ with $(\Delta(P,\psi,Q),e\otimes f^*)\le (\Delta(D,\phi,E),e_D\otimes f_E^*)$, setting $I:=N_G(P,e)$, $J:=N_H(Q,f)$, and $Y:=N_{I\times J}(\Delta(P,\psi,Q))$, the element $\mu:=\gamma(\Delta(P,\psi,Q),e\otimes f^*)^{\KK}$ in $R(\KK Y (e\otimes f^*))$ satisfies
\begin{equation*}
  \mu\dott{Y}{Y^\circ}{H} \mu^\circ = [\KK[C_G(P)]e] \quad \text{in $R(\KK[N_{I\times I}(\Delta(P))])$.}
\end{equation*}
\end{theorem}

\begin{proof}
First assume that $\gamma\in T_l^\Delta(A,B)$. Let $(\Delta(D,\phi,E), e_D\otimes f_E)$ be any maximal $\gamma$-Brauer pair. Then the conditions in (i) are satisfied by Theorem~\ref{thm gamma is uniform}(c) and by Theorem~\ref{thm isomorphic fusion systems}. Moreover, Theorem~\ref{thm gamma is uniform}(b), implies the condition in (ii), and Equation~(\ref{eqn Delta(P)}) in Theorem~\ref{thm local ppeqs} implies the condition in (iii). 

Next we assume that Conditions (i)--(iii) hold for $\gamma\in T^\Delta(A,B)$. We aim to show that $\gamma\cdotH \gamma^\circ=[A]$ in $T^\Delta(A,A)$. By Proposition~\ref{prop ghost group}(b), it suffices to show that
\begin{equation}\label{eqn (X,c)}
  \bigl((\gamma\cdotH\gamma^\circ)(X,c)\bigr)^{\KK} = \bigl(A(X,c)\bigr)^\KK\quad
  \text{in } T(\KK[N_{G\times G}(X,c)])\,,
\end{equation}
for all $(X,c)$ running through a set of representatives of the $G\times G$-orbits of $\calBP(A\otimes A^*)$.

\smallskip
{\it Claim 1: Both sides of (\ref{eqn (X,c)}) are $0$ unless $(X,c)=_{G\times G}(\Delta(P), e\otimes e^*)$ for some  Brauer pair $(P,e)$ with $(P,e)\le (D,e_D)$.} This is clear for the right hand side, since $(\Delta(D), e_D\otimes e_D^*)$ is a maximal Brauer pair of the indecomposable $p$-permutation $A\otimes A^*$-module $A$ (see Proposition~\ref{prop Brauer pairs for M}(c)). To prove this for the left hand side of (\ref{eqn (X,c)}), assume that $(X,c)\in\calBP(A\otimes A^*)$ is such that the left hand side of (\ref{eqn (X,c)}) is non-zero. Then also $(\gamma\cdotH\gamma^\circ)(X,c)\neq 0$ in $T(F[N_{G\times G}(X,c)])$. Since every indecomposable $p$-permutation $\calO[G\times G]$-module $M$ appearing in $\gamma\cdotH\gamma^\circ$ has twisted diagonal vertices (see Lemma~\ref{lem ext tp and vertices}(b)) and since every $M$-Brauer pair is also an $A\otimes A^*$-Brauer pair (see Proposition~\ref{prop Brauer pairs for M}(a)), we can conclude that $(X,c)$ is $G\times G$-conjugate to a Brauer pair of the form $(\Delta(P,\sigma,R),e\otimes d^*)$ with $(P,e)\le (D,e_D)$ and $(R,d)\le (D,e_D)$. Using Theorem~\ref{thm BP decomp}(c) with $I:=N_G(P,e)$ for $S$ and $K:=N_G(R,d)$ for $T$, this implies that there exists a Brauer pair $(Q,f)$ of $\calO H$ and isomorphisms $\alpha\colon Q\myiso P$ and $\beta\colon R\to Q$ such that $\alpha\beta=\sigma$, $\gamma(\Delta(P,\alpha,Q), e\otimes f^*)\neq 0$ in $T(F[N_{I\times J}(\Delta(P,\alpha,Q))])$, and $\gamma(\Delta(R,\beta^{-1},Q), d\otimes f^*)\neq 0$ in $T(F[N_{K\times J}(\Delta(R,\alpha,Q)]))$, where $J:=N_H(Q,f)$. By Condition~(ii) this implies $(\Delta(P,\alpha,Q),e\otimes f^*)\le_{G\times H}(\Delta(D,\phi,E),e_D\otimes f_E^*)$ and $(\Delta(R,\beta^{-1},Q), d\otimes f^*)\le_{G\times H}(\Delta(D,\phi,E),e_D\otimes f_E^*)$. Now Lemma~\ref{lem fusion consequence} implies Claim 1.

\smallskip
With Claim~1 it now suffices to show that
\begin{equation}\label{eqn (P,e)}
  \bigl((\gamma\cdotH\gamma^\circ)(\Delta(P),e\otimes e^*)\bigr)^{\KK} = 
  \bigl([A](\Delta(P), e\otimes e^*)\bigr)^{\KK} \quad \text{in }R(\KK[N_{I\times I}(\Delta(P))])\,,
\end{equation}
for all $A$-Brauer pairs $(P,e)$ with $(P,e)\le (D,e_D)$, where $I:=N_G(P,e)$. By Proposition~\ref{prop Brauer con of block}(b), the right hand side of (\ref{eqn (P,e)}) equals $[\KK[C_G(P)]e]$. We use the second part of Theorem~\ref{thm BP decomp}(d) with $S=I$ to compute the left hand side of (\ref{eqn (P,e)}). For $Q\le E$ let $f_Q$ denote the unique block idempotent of $\calO[C_H(Q)]$ such that $(Q,f_Q)\le (E,f_E)$ and set $Q_0:=\phi^{-1}(P)$ and $f_0:=f_{Q_0}$. Then, with the notation in Theorem~\ref{thm BP decomp}(d), the summand for $\lambda_0=(\phi,(Q_0,f_0))$ contributes the element
\begin{equation*}
  \ind_{\Delta(I(\lambda_0))(C_G(P)\times\{1\})}^{N_{I\times I}(\Delta(P))} 
  \bigl(\gamma(\Delta(P,\phi,Q_0), e\otimes f_0^*) \dott{Y_0}{Y_0^\circ}{H} 
  \gamma(\Delta(P,\phi,Q_0), e\otimes f_0^*)^\circ\bigr)\,,
\end{equation*}
where $Y_0=N_{G\times H}(\Delta(P\phi,Q_0), e\otimes f_0^*)$. Since $\phi$ is an isomorphism between the fusion systems $\calB$ and $\calA$ by Condition~(i), we have $I(\lambda_0)=I$ and therefore $\Delta(I(\lambda_0))(C_G(P)\times \{1\})=N_{I\times I}(\Delta(P))$. Thus, after applying $(-)^{\KK}$, the contribution of the summand parametrized by $\lambda_0$ to the left hand side of (\ref{eqn (P,e)}) is equal to $\mu_0\dott{Y_0}{Y_0^\circ}{H} \mu_0^\circ$ with $\mu_0:=\gamma(\Delta(P,\phi,Q_0), e\otimes f_0^*)^{\KK}$. Condition~(iii) now implies that this expression is also equal to $[\KK[C_G(P)]e]\in R(\KK[N_{I\times I}(\Delta(P))])$. Thus, it suffices to show {\it Claim 2: Let $(Q,f)$ be an $\calO H$-Brauer pair and $\alpha\colon Q\myiso P$ an isomorphism with $\gamma(\Delta(P,\alpha,Q),e\otimes f^*)\neq 0$ in $T(F[N_{G\times H}(\Delta(P,\alpha,Q), e\otimes f^*)])$ then $\lambda=(\alpha, (Q,f))$ belongs to the same $I\times H$-orbit as $\lambda_0$.} In order to prove this claim, note first that $\gamma(\Delta(P,\alpha,Q),e\otimes f^*)\neq 0$ implies that $(Q,f)$ is a $B$-Brauer pair. Without loss of generality we may assume that $(Q,f)\le (E,f_E)$. Condition~(ii) now implies that $(\Delta(P,\alpha,Q), e\otimes f^*)\le_{G\times H}(\Delta(D,\phi,E), e_D\otimes f_E^*)$. Applying the symmetric statement of Lemma~\ref{lem fusion consequence} to $(\Delta(P,\alpha,Q), e\otimes f^*)$ and $(\Delta(P,\phi,Q_0), e\otimes f_0^*)$, we obtain $(\alpha,(Q,f))=_H(\phi,(Q_0,f_0))$. This completes the proof of the theorem.
\end{proof}

\begin{theorem}\label{thm left equals right}
Let $A$ be a block of $\calO G$, $B$ a block of $\calO H$ and $\gamma\in T^\Delta(A,B)$. The following are equivalent:

\smallskip
{\rm (i)} $\gamma\in T^\Delta_l(A,B)$.

\smallskip
{\rm (ii)} $\gamma\in T^\Delta_r(A,B)$.

\smallskip
{\rm (iii)} $\gamma\in T^\Delta_o(A,B)$.

\smallskip
{\rm (iv)} $\gamma$ satisfies the conditions (i)--(iii) in Theorem~\ref{thm char criterion}.
\end{theorem}

\begin{proof}
By Theorem~\ref{thm char criterion}, (i) is equivalent with (iv). Morover, clearly (iii) implies (i) and (ii). Next we show that (i) implies (iii). In fact, assuming (i), Theorem~\ref{thm char criterion} shows that the conditions (i)--(iii) in Theorem~\ref{thm char criterion} are satisfied. Since conditions (i) and (ii) are symmetric and since Condition~(iii) for $\gamma$ is equivalent to Condition~(iii) for $\gamma^\circ$, by Proposition~\ref{prop local characters} and the equivalence of (i) and (ii) in Proposition~\ref{prop diagonal extension}(c), we see that the conditions (i)--(iii) also hold for $\gamma^\circ\in T^\Delta(B,A)$. Now Theorem~\ref{thm char criterion} implies that $\gamma^\circ\in T^\Delta_l(B,A)$. But this is equivalent to $\gamma\in T^\Delta_r(A,B)$. Thus, (i) implies (iii). 
Finally, if (ii) holds then $\gamma^\circ\in T_l^\Delta(B,A)$, and, since (i) implies (iii), we obtain $\gamma^\circ\in T_o^\Delta(B,A)$ and $\gamma\in T_o^\Delta(A,B)$. Thus, (ii) implies (iii), and the proof is complete.
\end{proof}

\section{K\"ulshammer-Puig classes}\label{sec KP classes preserved}

Throughout this section we assume that $G$ and $H$ are finite groups and that the $p$-modular system $(\KK,\calO,F)$ is large enough for $G$ and $H$. Further, we assume that $A$ is a block algebra of $\calO G$, that $B$ is a block algebra of $\calO H$, and that $\gamma\in T^\Delta(A,B)$ is a $p$-permutation equivalence.

The main goal of this section is to show that K\"ulshammer-Puig classes are \lq preserved\rq\ by $\gamma$, see Theorem~\ref{thm K-P classes} for the precise statement. In order to prove this we need to take a closer look at the elements $e\gamma(P,\psi,Q)f$, for a $\gamma$-Brauer pair $(\Delta(P,\psi,Q), e\otimes f^*)$ with $P$ and $Q$ centric in the associated fusion systems. This is done in Proposition~\ref{prop modules for centric case}.

\begin{lemma}\label{lem local defect group}
Let $(\Delta(D,\phi,E), e\otimes f^*)$ be a maximal $\gamma$-Brauer pair and let $\calF$ denote the fusion system associated to the maximal $A\otimes B^*$-Brauer pair $(D\times E,e\otimes f^*)$. Then the subgroup $\Delta(D,\phi,E)$ of $D\times E$ is fully $\calF$-normalized. Moreover, setting $Y:=N_{G\times H}(\Delta(D,\phi,E), e\otimes f^*)$, the block $\calO Y(e\otimes f^*)$ has the normal subgroup $(Z(D)\times Z(E))\cdot\Delta(D,\phi,E)$ of $Y$ as defect group.
\end{lemma}

\begin{proof}
Any $G\times H$-conjugate subgroup of $\Delta(D,\phi,E)$ which is contained in $D\times E$ must be again a twisted diagonal subgroup and therefore of the form $\Delta(D,\psi,E)$ for some isomorphism $\psi\colon E\myiso D$. By Lemma~\ref{prop Nphi}(b), one has $N_{D\times E}(\Delta(D,\psi,E))=(Z(D)\times Z(E))\cdot\Delta(D,\psi,E)$, whose order is independent of $\psi$. This proves the first statement. The second statement follows immediately from Proposition~\ref{prop fully centralized and centric}(b).
\end{proof}

\begin{proposition}\label{prop modules for centric case}
Let $(\Delta(P,\psi,Q), e\otimes f^*)\in\calBP_\calO(\gamma)$ and suppose that $Z(P)$ is a defect group of $\calO[C_G(P)]e$. Set
\begin{gather*}
  I:=N_G(P,e),\quad J:=N_H(Q,f),\quad Y:=N_{I\times J}(\Delta(P,\psi,Q)),\quad C:=C_G(P)\times C_H(Q),\\
  Z:=Z(P)\times Z(Q),\quad 
  \text{and}\quad Z':=\Delta(Z(P),\psi, Z(Q))\,.
\end{gather*}

\smallskip
{\rm (a)} One has
\begin{equation*}
  \res^Y_C(e\gamma(P,\psi,Q)f) = \varepsilon[N'] \quad\text{in $T(\calO C(e\otimes f^*))$}\,,
\end{equation*}
where $\varepsilon\in\{\pm 1\}$ and $N'$ is the unique indecomposable $p$-permutation $\calO C(e\otimes f^*)$-module with vertex $Z'$ (cf.~Proposition~\ref{prop central defect group}(b)). Moreover, the $p$-permutation $\calO C(e\otimes f^*)$-module $V':=\Inf_{C/Z}^C\Def^C_{C/Z}(N')$ satisfies:

\smallskip
\quad {\rm (i)} $F\otimes_{\calO}V'$ is the unique simple $FC(e\otimes f^*)$-module;

\smallskip
\quad {\rm (ii)} $\KK \otimes_{\calO}V'$ is a simple $\KK C(e\otimes f^*)$-module.

\smallskip
{\rm (b)} Set $\zeta:=\infl^Y_{Y/Z}\defl^Y_{Y/Z}(e\gamma(P,\psi,Q)f)\in T(\calO Y(e\otimes f^*))$. There exists a simple $\KK Y(e\otimes f^*)$-module $W$ and a simple $FY(e\otimes f^*)$-module $\Wbar$ such that $d_Y([W])=[\Wbar]$ in $R(FY(e\otimes f^*))$ and such that the following hold with $\varepsilon$ from Part~(a):

\smallskip
\quad {\rm (i)} $\kappa_Y(\zeta)=\varepsilon[W]$ in $R(\KK Y(e\otimes f^*))$ and $\Res^Y_C(W)\cong \KK\otimes_{\calO} V'$;

\smallskip
\quad {\rm (ii)} $\eta_Y(\zetabar)=\varepsilon[\Wbar]$ in $R(FY(e\otimes f^*))$ and $\Res^Y_C(\Wbar)\cong F\otimes_{\calO} V'$.

\smallskip
{\rm (c)} If $(\Delta(P,\psi,Q), e\otimes f^*)$ is a maximal $\gamma$-Brauer pair then there exists an indecomposable $p$-permutation $\calO Y(e\otimes f^*)$-module $N$ such that
\begin{equation*}
  e\gamma(P,\psi,Q)f = \varepsilon [N] \text{ in $T(\calO Y(e\otimes f^*))$}\quad\text{and}\quad \Res^Y_C(N)\cong N'\,,
\end{equation*}
where $\varepsilon$ is from Part~(a).
\end{proposition}

We visualize the situation via the diagrams of elements
\begin{equation}\label{eqn zeta diagrams}
\text{
\parbox{4.5cm}{
\begin{diagram}[50]
  \zeta & \mapsto & \varepsilon[W] &&
  {\begin{sideways}{\small $\lmapsto$} \end{sideways}} & & {\begin{sideways}{\small $\lmapsto$} \end{sideways}} &&
  \zetabar & \mapsto & \varepsilon[\Wbar] &&
\end{diagram}}}
\quad\text{and} \quad
\text{
\parbox{5cm}{
\begin{diagram}[50]
  \movevertex(-12,0){\varepsilon[V']} & \mapsto & \movevertex(20,0){\varepsilon[\KK\otimes_\calO V']} &&
  \movearrow(-12,-2){\begin{sideways}{\small $\lmapsto$} \end{sideways}} & & 
  \movearrow(20,-2){\begin{sideways}{\small $\lmapsto$} \end{sideways}} &&
  \movevertex(-17,0){\varepsilon[F\otimes_\calO V']} & \mapsto & \movevertex(20,0){\varepsilon[F\otimes_\calO V']} &&
\end{diagram}}}
\end{equation}
in the diagrams of groups
\begin{equation*}
\text{
\parbox{6cm}{
\begin{diagram}[50]
  \llap{$T(\calO Y(e\otimes f^*))$} & \Ear{\kappa_Y} & \rlap{$R(\KK Y(e\otimes f^*))$} &&
  \movearrow(-40,0){\Sar{\wr}} & & \movearrow(40,0){\saR{d_Y}} &&
  \llap{$T(FY(e\otimes f^*))$} & \Ear{\eta_Y} & \rlap{$R(FY(e\otimes f^*))$} &&
\end{diagram}}}
\qquad\text{and}\qquad 
\text{
\parbox{6cm}{
\begin{diagram}[50]
  \llap{$T(\calO C(e\otimes f^*))$} & \Ear{\kappa_C} & \rlap{$R(\KK C(e\otimes f^*))$} &&
  \movearrow(-40,0){\Sar{\wr}} & & \movearrow(40,0){\saR{d_C}} &&
  \llap{$T(FC(e\otimes f^*))$} & \Ear{\eta_C} & \rlap{$R(FC(e\otimes f^*))$\,,} &&
\end{diagram}}}
\end{equation*}
where the right diagram of elements is the image of the left one under the restriction maps $\res^Y_C$, as shown in the proof. Before we prove the proposition, we mention the following negative result:

\begin{remark}\label{rem non-maximal case}
Assume the situation in Proposition~\ref{prop modules for centric case}. In general, the element $e\gamma(P,\psi,Q)f\in T(\calO Y(e\otimes f^*))$ is not plus or minus the class of an indecomposable $p$-permutation $\calO Y$-module. In fact, let $G=H$ be the dihedral group of order $8$, $P$ the cyclic subgroup of order $4$ and $\gamma:=[\calO[G\times G/\Delta(G)]]-[\calO[G\times G/\Delta(P)]]\in T^\Delta(\calO G,\calO G)$. Then $\gamma$ is a $p$-permutation equivalence, $P=Z(P)=C_G(P)$ is a defect group of $\calO[C_G(P)]$, but $\gamma(\Delta(P))= [\calO[Y/\Delta(G)]]-[\calO[Y/\Delta(P)]]$ in $T(\calO Y)$.
\end{remark}

\begin{proof}{\sl of Proposition~\ref{prop modules for centric case}}.\quad We
set $\omega:= e\gamma(P,\psi,Q)f\in T(\calO Y(e\otimes f^*))$ and $\omega':=\res^Y_C(\omega)\in T(\calO C(e\otimes f^*))$. 

\smallskip
(a) Since $Z(P)$ is a defect group of $\calO[C_G(P)]e$, $Z(Q)$ is a defect group of $\calO[C_H(Q)]f$. In fact, we can choose a maximal $\gamma$-Brauer pair $(\Delta(D,\phi,E), e'\otimes {f'}^*)$ such that $(\Delta(P,\psi,Q), e\otimes f^*)\le (\Delta(D,\phi,E),e'\otimes {f'}^*)$. Then, $\psi$ is the restriction of $\phi$ and $\phi\colon E\myiso D$ is an isomorphism between the fusion systems $\calA$ and $\calB$ associated with $(D,e')$ and $(E,f')$, cf.~Theorem~\ref{thm isomorphic fusion systems}, and since $P$ is centric in the fusion system $\calA$, $Q$ must be centric in the fusion system $\calB$. Applying now Proposition~\ref{prop fully centralized and centric} the claim follows. Thus, $Z$ is a defect group of $\calO C(e\otimes f^*)$.

Let $U$ be an indecomposable $\calO C$-module appearing in $\omega'$. We will show that $Z'$ is a vertex of $U$. In fact, let $X\le C$ denote a vertex of $U$. Since $\Delta(P,\psi,Q)$ acts trivially on every indecomposable $\calO Y$-module appearing in $\omega$, the group $Z'$ acts trivially on $U$. Thus, $Z'\le X$. Moreover, $X$ must be a twisted diagonal subgroup of $Z$, since $Z$ is the unique defect group of $\calO C(e\otimes f^*)$. This implies $X=Z'$.

Since $\calO C(e\otimes f^*)$ has the central defect group $Z$, Proposition~\ref{prop central defect group}(b) implies that there exists a unique $p$-permutation $\calO C(e\otimes f^*)$-module $N'$ with vertex $Z'$. Thus, by the previous paragraph, there exists $\varepsilon\in \ZZ$ with $\omega'=\varepsilon[N']\in T(\calO C)$. Since Lemma~\ref{lem unique Lambdatilde lr} implies that $\kappa_C(\omega')=\varepsilon[\KK\otimes_\calO N']$ is an isometry between $\KK[C_G(P)]e$ and $\KK[C_H(E)]f$, we obtain $\varepsilon\in\{\pm1\}$.
Moreover, by Proposition~\ref{prop central defect group}(b), the $p$-permutation $\calO C(e\otimes f^*)$-module $V':=\Inf_{C/Z}^C\Def^C_{C/Z}(N')$ has the property that $F\otimes_\calO V'$ is the unique simple $FC(e\otimes f^*)$-module. Note also that $F\otimes_\calO V'$ can be viewed as the unique simple $F[C/Z]\pi(e\otimes f^*)$-module, where $\pi\colon\calO C\to \calO[C/Z]$ denotes the canonical map $\calO C\to \calO[C/Z]$. Since $\calO[C/Z]\pi(e\otimes f^*)$ is a block of defect $0$ (see Proposition~\ref{prop central defect group}(a)), the $\KK C(e\otimes f^*)$-module $\KK\otimes_{\calO}V'$ is simple, namely the inflation of the unique simple $\KK[C/Z]\pi(e\otimes f^*)$-module. This establishes all the statements in Part~(a).


\smallskip
(b) Since $Z'\le \Delta(P,\psi,Q)$, $Z'$ acts trivially on any indecomposable $\calO Y$-module appearing in $\omega$. Using $Z=Z'\cdot(Z(P)\times\{1\})$ and Proposition~\ref{prop ext hom duals}(d), we obtain
\begin{equation}\label{eqn inf def}
  \zeta=\infl_{Y/Z}^Y\defl^Y_{Y/Z}(\omega) = \infl_{Y/(Z(P)\times\{1\})}^Y\defl^Y_{Y/(Z(P)\times\{1\})}(\omega) =
  [\calO[C_G(P)/Z(P)]\ebar]\mathop{\cdot}\limits^{Y*Y^\circ, Y}_{G}\omega
\end{equation}
in $T(\calO Y(e\otimes f^*))$, where $\ebar$ denotes the image of $e$ under the canonical map $\calO[C_G(P)]\to\calO[C_G(P)/Z(P)]$. Note that $Y*Y^\circ=\Delta(I)\cdot (C_G(P)\times C_G(P))$ by Lemma~\ref{lem comp form}(a) and Proposition~\ref{prop isomorphic inertia quotients}. We claim that the character of the $\calO[\Delta(I)\cdot (C_G(P)\times C_G(P))]$-module $\calO[C_G(P)/Z(P)]\ebar$ in (\ref{eqn inf def}) is irreducible. In fact, since $\calO[C_G(P)]e$ has central defect group $Z(P)$, $\calO[C_G(P)/Z(P)]\ebar$ is a block of defect $0$. Thus, $\KK[C_G(P)/Z(P)]\ebar$ is an irreducible $(\KK[C_G(P)/Z(P)]\ebar, \KK[C_G(P)/Z(P)]\ebar)$-bimodule and via inflation an irreducible $\KK[C_G(P)\times C_G(P)](e\otimes e^*)$-module, which is a fortiori irreducible as $\KK[\Delta(I)\cdot(C_G(P)\times C_G(P))]$-module. Now, (\ref{eqn inf def}) 
and Lemma~\ref{lem unique Lambdahat lr 2} imply that $\kappa_Y(\zeta)=\varepsilon'\cdot[W]$, for some $\varepsilon'\in\{\pm 1\}$ and some simple $\KK Y(e\otimes f^*))$-module $W$. Since clearly
\begin{equation*}
  \res^Y_C\infl_{Y/Z}^Y\defl^Y_{Y/Z}=\infl_{C/Z}^C\defl^C_{C/Z}\res^Y_C\colon T(\calO Y)\to T(\calO C)\,,
\end{equation*}
we have $\res^Y_C(\zeta)=\infl^C_{C/Z}\defl^C_{C/Z}\res^Y_C(\omega)= \infl_{C/Z}^C\defl^C_{C/Z}(\omega') = \varepsilon[V']$ and therefore $\res^Y_C(\varepsilon'[W])=\res^Y_C(\kappa_Y(\zeta)) = \kappa_C(\res^Y_C(\zeta))=\varepsilon\kappa_C[V'] = \varepsilon[\KK\otimes_{\calO}V']$. This implies that $\varepsilon'=\varepsilon$ and establishes Statement~(i) of Part~(b). Moreover, we have $\res^Y_C(d_Y([W]))= d_C(\res^Y_C([W])) = d_C([\KK\otimes_\calO V']) = [F\otimes_{\calO}V']\in R(FC)$. Since $F\otimes_{\calO}V'$ is an irreducible $FC(e\otimes f^*)$-module, the previous equation implies that $d_Y([W])=[\Wbar]$, for some irreducible $FY(e\otimes f^*)$-module $\Wbar$. This completes the proof of Part~(b).

\smallskip
(c) By Lemma~\ref{lem local defect group}, the block algebra $\calO Y(e\otimes f^*)$ has defect group $Z\cdot\Delta(P,\psi,Q)$, a normal subgroup of $Y$. Thus, by \cite[Theorems~V.8.7(ii) and V.8.10]{NagaoTsushima1989}, $\calO[Y/Z\cdot\Delta(P,\psi,Q)]\pi(e\otimes f^*)$ is a sum of blocks of defect $0$, where $\pi\colon\calO Y\to\calO[Y/Z\cdot\Delta(P,\psi,Q)]$ denotes the canonical homomorphism. Since $\Delta(P,\psi,Q)$ acts trivially on every indecomposable $\calO Y$-module appearing in $\omega$, we have
\begin{equation*}
  \zetabar=\infl_{Y/Z}^Y\defl^Y_{Y/Z}(\omegabar) = 
  \infl_{Y/Z\cdot\Delta(P,\psi,Q)}^Y\defl^Y_{Y/Z\cdot\Delta(P,\psi,Q)}(\omegabar)\,,
\end{equation*}
where $\defl^Y_{Y/Z\cdot\Delta(P,\psi,Q)}(\omegabar)\in T(F[Y/Z\cdot\Delta(P,\psi,Q)])$ is a $\ZZ$-linear combination of classes of simple $p$-permutation modules.This property is preserved under inflation, so that also $\zetabar$ is a $\ZZ$-linear combination of classes of simple $p$-permutation $FY$-modules. But since $\eta_Y(\zetabar)=\varepsilon[\Wbar]$, we obtain that $\Wbar$ is a simple $p$-permutation $FY(e\otimes f^*)$-module and that $\zetabar=\varepsilon[\Wbar]$ in $T(FY(e\otimes f^*))$. 

Since every indecomposable $FY$-module appearing in $\omegabar\in T(FY)$ has vertex $\Delta(P,\psi,Q)$ and $\infl_{Y/Z}^Y\defl^Y_{Y/Z}(\omegabar) = \zetabar = \varepsilon [\Wbar]\in T(FY)$, Proposition~\ref{prop head and def}(b) implies that $\omegabar=\varepsilon[\Nbar]\in T(FY(e\otimes f^*))$ for an indecomposable $p$-permutation $FY(e\otimes f^*)$-module $\Nbar$ with vertex $\Delta(P,\psi,Q)$. Let $N$ be the indecomposable $p$-permutation $\calO Y(e\otimes f^*)$-module corresponding to $\Nbar$. Then $\omega=\varepsilon[N]$ and $N'\cong \Res^Y_C(N)$, since $\omega'=\res^Y_C(\omega)$. This completes the proof of the proposition.
\end{proof}

Next we show that $p$-permutation equivalences \lq preserve\rq\ K\"ulshammer-Puig classes.

\begin{theorem}\label{thm K-P classes}
Let $A$ be a block of $\calO G$, $B$ a block of $\calO H$, and $\gamma\in T^\Delta_o(A,B)$. Further, let $(\Delta(P,\psi,Q),e\otimes f^*)\in\calBP_\calO(\gamma)$ and suppose that $Z(P)$ is a defect group of the block $\calO[C_G(P)]e$. Then $Z(Q)$ is a defect group of the block $\calO[C_H(Q)]f$, by Theorem~\ref{thm isomorphic fusion systems} and the preservation of centric subgroups. Set $I:=N_G(P,e)$, $\Ibar:=I/PC_G(P)$, $J:=N_H(Q,f)$, and $\Jbar:=J/QC_H(Q)$. Furthermore, let $\kappa\in H^2(\Ibar,F^\times)$ and $\lambda\in H^2(\Jbar, F^\times)$ denote the K\"ulshammer-Puig classes of $(P,\ebar)$ and $(Q,\fbar)$, respectively. Then
\begin{equation*}
  \lambda= \etabar^*(\kappa)\in H^2(\Jbar,F^\times)\,,
\end{equation*}
where $\etabar\colon \Jbar\myiso \Ibar$ is induced by the isomorphisms in Proposition~\ref{prop isomorphic inertia quotients} and $\etabar^*:=H^2(\etabar,F^\times)$.
\end{theorem}

\begin{proof}
Let $M$ denote the unique simple $F[PC_G(P)]e$-module and let $N$ denote the unique simple $F[QC_H(Q)]f^*$-module. Then, by \ref{noth KP classes}, $\kappa$ is the Schur class of $M$ with respect to $PC_G(P)\trianglelefteq I$, and $\lambda^{-1}$ is the Schur class of $N$ with respect to $QC_H(Q)\trianglelefteq J$. Set $C:=C_G(P)\times C_H(Q)$, $Y:=N_{I\times J}(\Delta(P,\psi,Q))$ and $\Ytilde:=Y(P\times Q)\le I\times J$. By Proposition~\ref{prop isomorphic inertia quotients} and Lemma~\ref{lem Schur for Y}, it suffices to show that the irreducible $F[C(P\times Q)]$-module $M\otimes_F N$ can be extended to $\Ytilde$.

By Proposition~\ref{prop modules for centric case}(b), there exists an irreducible $FY(e\otimes f^*)$-module $W$ (it is denoted by $\Wbar$ there) such that $\Res^Y_C(W)\cong M\otimes N$. Consider the canonical isomorphism $\Ytilde/(P\times Q)\cong Y/Y\cap(P\times Q)$ and note that $Y\cap(P\times Q)=(Z(P)\times Z(Q))\Delta(P,\psi,Q)$ is a normal $p$-subgroup of $Y$ and therefore acts trivially on $W$. Using the above isomorphism we see that $W$ extends to an $F\Ytilde$-module $\Wtilde$. Then, $\Res^{\Ytilde}_{C(P\times Q)}(\Wtilde)\cong M\otimes N$, since $P\times Q$ acts trivially on $\Wtilde$ and on $M\otimes N$ and since $\Res^{\Ytilde}_C(\Wtilde)=\Res^Y_C(W)\cong M\otimes N$. Thus, the $I\times J$-stable simple $F[C(P\times Q)]$-module $M\otimes N$ can be extended to $\Ytilde$ and the proof is complete.
\end{proof}


\section{The maximal module of a $p$-permutation equivalence}\label{sec maximal module}

Throughout this section we assume that $G$ and $H$ are finite groups and that the $p$-modular system $(\KK,\calO,F)$ is large enough for $G$ and $H$. Further, we assume that $A$ is a block algebra of $\calO G$, that $B$ is a block algebra of $\calO H$, and that $\gamma\in T^\Delta_o(A,B)$ is a $p$-permutation equivalence.

\smallskip
We will take a closer look at the Brauer pairs (and therefore vertices) of the indecomposable modules appearing in $\gamma$. We establish that there is a unique indecomposable module $M$ appearing in $\gamma$ such that $\calBP_{\calO}(M)=\calBP_{\calO}(\gamma)$ and $\calBP_{\calO}(M)\supset \calBP_{\calO}(N)$ for all other indecomposable modules $N$ appearing in $\gamma$. Moreover, we show that the Brauer construction of $M$ with respect to a vertex of $M$ yields a $p$-permutation bimodule that induces a Morita equivalences between the Brauer correspondents of $A$ and $B$.

\smallskip
All results and definitions in this section have obvious analogues if $A$ and $B$ are block algebras over $F$.

\begin{theorem}\label{thm maximal module}
Up to isomorphism, there exists a unique indecomposable $(A,B)$-bimodule $M$ appearing in $\gamma$ whose vertex is of the form $\Delta(D,\phi,E)$ where $D$ is a defect group of $A$. Moreover, the multiplicity of $[M]$ in $\gamma$ equals $1$ or $-1$.
\end{theorem}

\begin{proof}
{\em Existence:} Let $(\Delta(D,\varphi,E),e\otimes f^*)\in\calBP_\calO(\gamma)$ be a maximal $\gamma$-Brauer pair. Then, by Theorem~\ref{thm gamma is uniform}(c), $D$ is a defect group of $A$. Since $e\gamma(\Delta(D,\phi,E))f\neq 0$, there exists an indecomposable $(A,B)$-bimodule $M$ appearing in $\gamma$ with $e M(\Delta(D,\phi,E))f\neq \{0\}$. Thus, $(\Delta(D,\phi,E), e\otimes f^*)\in\calBP_\calO(M)$ and $\Delta(D,\phi,E)$ is contained in a vertex of $M$, by Proposition~\ref{prop Brauer pairs for M}. Since $M$ belongs to the block $A\otimes B^*$ with defect group $D\times E$ and $M$ has twisted diagonal vertex, $\Delta(D,\phi,E)$ must be a vertex of $M$.

\smallskip
{\em Uniqueness:} Let $N$ be an indecomposable $(A,B)$-bimodule appearing in $\gamma$ which has a vertex of the form $\Delta(\Dtilde,\phitilde,\Etilde)$, where $\Dtilde$ is a defect group of $A$. We will show that $N$ is isomorphic to $M$ from above. Choose block idempotents $\etilde$ and $\ftilde$ of $\calO C_G(\Dtilde)$ and $\calO C_H(\Etilde)$, respectively, such that  $(\Delta(\Dtilde,\phitilde,\Etilde),\etilde\otimes \ftilde^*)$ is a maximal $N$-Brauer pair.  By Lemma~\ref{lem M Bp implies gamma Bp} and Proposition~\ref{prop equiv gamma Brpair cond}(i)$\iff$(iv), $(\Delta(\Dtilde,\phitilde,\Etilde),\etilde\otimes \ftilde^*)$ is a $\gamma$-Brauer pair. Moreover, $(\Dtilde,\etilde)$ is a maximal $A$-Brauer pair and Theorem~\ref{thm gamma is uniform}(c) implies that $(\Delta(\Dtilde,\phitilde,\Etilde),\etilde\otimes\ftilde^*)$ is a maximal $\gamma$-Brauer pair. By Theorem~\ref{thm gamma is uniform}(b) and Proposition~\ref{prop Brauer pairs for M}(c), we may choose $(\Delta(\Dtilde,\phitilde,\Etilde),\etilde\otimes\ftilde^*)$ to be equal to $(\Delta(D,\phi,E),e\otimes f^*)$. By Lemma~\ref{lem equal coefficients}(a), the coefficients of $[M]$ (resp. $[N]$) in $\gamma$ equals the coefficient of $[eM(\Delta(D,\phi,E)f]$ (resp.~$eN(\Delta(D,\phi,E))f$) in $e\gammabar(\Delta(D,\phi,E))f\in T(FN_{G\times H}(\Delta(D,\phi,E),e\otimes f^*))$. But, by Proposition~\ref{prop modules for centric case}(c), the element $e\gammabar(D,\phi,E)f\in T(FN_{G\times H}(\Delta(D,\phi,E),e\otimes f^*))$ is of the form $\varepsilon\cdot [L]$ for some $\varepsilon\in\{\pm 1\}$ and some indecomposable $FN_{G\times H}(\Delta(D,\phi,E),e\otimes f^*))$-module $L$. Thus, since both $eM(\Delta(D,\phi,E))f$ and $eN(\Delta(D,\phi,E))f$ appear in $e\gammabar(\Delta(D,\phi,E))f$, we have
\begin{equation*}
  eM(\Delta(D,\phi,E))f \cong L \cong eN(\Delta(D,\phi,E))f\,.
\end{equation*}
Now, Proposition~\ref{prop isomorphic test} implies that $M\cong N$. Finally, the above statement about multiplicities implies that the multiplicity of $M$ in $\gamma$ is equal to $\varepsilon$.
\end{proof}

\begin{definition}\label{def maximal module}
The module $M$ from Theorem~\ref{thm maximal module} is called the {\em maximal module} of $\gamma$ and its multiplicity $\varepsilon\in\{\pm 1\}$ in $\gamma$ is called the {\em sign} of $\gamma$.
\end{definition}

\begin{theorem}\label{thm Brauer pairs of appearing modules}
Let $N$ be an indecomposable $(A,B)$-bimodule appearing in $\gamma$ and let $M$ be the maximal module of $\gamma$. Then $\calBP_\calO(N)\subseteq \calBP_\calO(M)=\calBP_\calO(\gamma)$. If $M\not\cong N$ then $\calBP_\calO(N)\subset \calBP_\calO(M)$. In particular, every vertex of $N$ is contained in a vertex of $M$, with strict containment if $N\not\cong M$.
\end{theorem}

\begin{proof}
By Lemma~\ref{lem M Bp implies gamma Bp}, we have $\calBP_{\calO}(N)\subseteq\calBP_\calO(\gamma)$. To see that $\calBP_\calO(M)=\calBP_\calO(\gamma)$, it suffices to show that the maximal $\gamma$-Brauer pairs and the maximal $M$-Brauer pairs coincide, see Proposition~\ref{prop Brauer pairs for M}(b) and Theorem~\ref{thm gamma is uniform}(a). If $(\Delta(D,\phi,E),e\otimes f^*)$ is a maximal $M$-Brauer pair then it is a $\gamma$-Brauer pair by the first statement, and even a maximal $\gamma$-Brauer pair by Theorem~\ref{thm gamma is uniform}(c). Conversely, if $(\Delta(D,\phi,E), e\otimes f^*)$ is a maximal $\gamma$-Brauer pair then by the existence part of the proof of Theorem~\ref{thm maximal module}, it is also an $M$-Brauer pair, and by the already established inclusion also a maximal $M$-Brauer pair. Finally, if $M\not\cong N$, then, by the definition of $M$, the vertices of $N$ have smaller order than the vertices of $M$ so that the inclusion $\calBP_\calO(N)\subset\calBP_\calO(M)$ is proper.
\end{proof}

\begin{proposition}\label{prop Morita on local level}
Let $(\Delta(D,\phi,E), e\otimes f^*)$ be a maximal $\gamma$-Brauer pair, let $M$ be the maximal module of $\gamma$, and let $C_G(D)\le S\le N_G(D,e)$ and $C_H(E)\le T\le N_H(E,f)$ be intermediate groups that correspond under the isomorphism in Proposition~\ref{prop isomorphic inertia quotients}. 
Let $L\in\ltriv{\calO[N_{G\times H}(\Delta(D,\phi,E))]}$ be the Green correspondent of $M$.
Then the $p$-permutation $(FSe, FTf)$-bimodule
\begin{equation*}
  \Ind_{N_{S\times T}(\Delta(D,\phi,E))}^{S\times T}\bigl((e\otimes f^*)L\bigr)
\end{equation*}
induces a Morita equivalence between $\calO Se$ and $\calO Tf$.
\end{proposition}

\begin{proof}
This follows immediately from Theorem~\ref{thm local ppeqs} noting that $e\gamma(\Delta(D,\phi,E))f=\pm[(e\otimes f^*)L]\in T(\calO N_{G\times H}(\Delta(D,\phi,E),e\otimes f^*))$ by Proposition~\ref{prop p-perm equiv} (c), Theorem~\ref{thm maximal module}, and Theorem~\ref{thm Brauer pairs of appearing modules}.
\end{proof}

\begin{theorem}\label{thm Morita between Brauer correspondents}
Let $\Delta(D,\phi,E)$ be a vertex of the maximal module $M$ of $\gamma$ and let $A'$ (resp.~$B'$) be the block algebra of $\calO N_G(D)$ (resp.~$\calO N_H(E)$) that is in Brauer correspondence with $A$ (resp.~$B$) via Brauer's First Main Theorem. Furthermore, let $L\in\ltriv{\calO[N_{G\times H}(\Delta(D,\phi,E))]}$ be the Green correspondent of $M$.
Then the $p$-permutation $(A',B')$-bimodule
\begin{equation}\label{eqn Morita bimodule}
  \Ind_{N_{G\times H}(\Delta(D,\phi,E))}^{N_G(D)\times N_H(E)} (L)
\end{equation}
induces a Morita equivalence between $A'$ and $B'$.
\end{theorem}

\begin{proof}
There exist block idempotents $e$ of $\calO C_G(D)$ and $f$ of $\calO C_H(E)$ such that $(\Delta(D,\phi,E),e\otimes f^*)$ is a maximal $M$-Brauer pair and therefore a maximal $\gamma$-Brauer pair by Theorem~\ref{thm Brauer pairs of appearing modules}. Set $I:=N_G(D,e)$ and $J:=N_H(E,f)$. By Proposition~\ref{prop Morita on local level}, the $p$-permutation $(\calO N_G(D,e),\calO N_H(E,f))$-bimodule $\Ind_{N_{I\times J}(\Delta(D,\phi,E))}^{I\times J}\bigl((e\otimes f^*)L\bigr)$ induces a Morita equivalence between $\calO Ie$ and $\calO Jf$. Moreover, if $e'$ (resp.~$f'$) denotes the identity element of $A'$ (resp.~$B'$) then the $(\calO N_G(D)e', \calO Ie)$-bimodule $\calO N_G(D)e=e'\calO N_G(D)e$ (resp.~the $(\calO Jf,\calO N_H(E)f')$-bimodule $f\calO N_H(E)=f\calO N_H(E)f'$) induces a Morita equivalence between $\calO N_G(D)e'=A'$ and $\calO Ie$ (resp.~between $\calO Jf$ and $\calO N_H(E)f'=B'$). Thus, the $(A',B')$-bimodule
\begin{equation*}
  \calO N_G(D)e \otimes_{\calO Ie} \Ind_{N_{I\times J}(\Delta(D,\phi,E))}^{I\times J}\bigl((e\otimes f^*)L \bigr) \otimes_{\calO Jf} f\calO N_H(E)
\end{equation*}
induces a Morita equivalence between $A'$ and $B'$. However, the latter $(A',B')$-bimodule is isomorphic to
\begin{equation*}
  \Ind_{N_{I\times J}(\Delta(D,\phi,E))}^{N_G(D)\times N_H(E)} \bigl((e\otimes f^*)L\bigr) 
   \cong \Ind_{N_{G\times H}(\Delta(D,\phi,E))}^{N_G(D)\times N_H(E)} \Ind_{N_{I\times J}(\Delta(D,\phi,E))}^{N_{G\times H}(\Delta(D,\phi,E))}\bigl((e\otimes f^*)L\bigr)
\end{equation*}
which is isomorphic to the $(A',B')$-bimodule in (\ref{eqn Morita bimodule}), since $\Ind_{N_{I\times J}(\Delta(D,\phi,E))}^{N_{G\times H}(\Delta(D,\phi,E))}\bigl((e\otimes f^*)L\bigr)\cong L$.
\end{proof}

\section{Connection with isotypies and splendid Rickard equivalences}\label{sec isotypies}

Throughout this section we assume that $G$ and $H$ are finite groups and that the $p$-modular system $(\KK,\calO,F)$ is large enough for $G$ and $H$. Further, we assume that $A$ is a block algebra of $\calO G$, that $B$ is a block algebra of $\calO H$.
We will establish that a splendid Rickard equivalence between $A$ and $B$ induces a $p$-permutation equivalence between $A$ and $B$ and that a $p$-permutation equivalence between $A$ and $B$ induces an isotypy between $A$ and $B$.

\begin{definition}\label{def splendid Rickard equivalence}
A {\em splendid Rickard equivalence} between $A$ and $B$ is a bounded chain complex $C_\bullet$ of $p$-permutation $(A,B)$-bimodules satisfying the following properties:

\smallskip
(i) For every $n\in\ZZ$, the vertices of indecomposable direct summands of $C_n$ are twisted diagonal subgroups of $G\times H$.

\smallskip
(ii) One has $C_\bullet \otimes_B C_\bullet^\circ\cong A$ in the homotopy category of chain complexes of $(A,A)$-bimodules and $C_\bullet^\circ\otimes_A C_\bullet\cong B$ in the homotopy category of chain complexes of $(B,B)$-bimodules. Here, $A$ denotes the chain complex  with only one non-zero term $A$ in degree $0$, and $C_\bullet^\circ$ denotes the $\calO$-dual of the chain complex $C_\bullet$.
\end{definition}

The proof of the following theorem can be easily adapted from the proof of \cite[Theorem~1.5]{BoltjeXu2008}.

\begin{theorem}\label{thm Rickard implies ppeq}
If $C_\bullet$ is a splendid Rickard equivalence between $A$ and $B$ then
\begin{equation*}
  \gamma:=\sum_{n\in\ZZ} (-1)^n[C_n] \in T^\Delta(A,B)
\end{equation*}
is a $p$-permutation equivalence between $A$ and $B$.
\end{theorem}

Next we will show that $p$-permutation equivalences induce isotypies. The following definition is due to Brou\'e, cf.~Definition~4.6 and the subsequent Remark~2 in \cite{Broue1990} and Definition~2.1 in \cite{Broue1995}.

\begin{definition}\label{def isotypy}
An {\em isotypy} between $A$ and $B$ consists of the following data:

\smallskip
$\bullet$ Maximal Brauer pairs $(D,e)\in\calBP_\calO(A)$ and $(E,f)\in\calBP_\calO(B)$; 

\smallskip
$\bullet$ an isomorphism $\phi\colon E\myiso D$ which is also an isomorphism between the fusion system $\calB$ of $B$ associated to $(E,f)$ and the fusion system $\calA$ of $A$ associated to $(D,e)$; and 

\smallskip
$\bullet$ a family of perfect isometries $\mu_Q\in R\bigl(\KK C_G(\phi(Q)) e_{\phi(Q)},\KK C_H(Q) f_Q\bigr)$, $Q\le E$, where $f_Q$ denotes the unique block idempotent of $\calO C_H(Q)$ with $(Q,f_Q)\le (E,f)$ and, for $P\le D$, $e_P$ denotes the unique block idempotent of $\calO C_G(P)$ with $(P,e_P)\le (D,e)$.

\smallskip
These data are subject to the following conditions:

\smallskip
(i) (Equivariance) For every Brauer pair $(\Delta(P,\psi,Q), e_P\otimes f_Q^*)\le (\Delta(D,\phi,E),e\otimes f^*)$ and every $(g,h)\in G\times H$ such that also $\lexp{(g,h)}{(\Delta(P,\psi,Q), e_P\otimes f_Q^*)}\le (\Delta(D,\phi,E),e\otimes f^*)$, one has $\lexp{(g,h)}\mu_Q=\mu_{(\lexp{h}{Q})}$.

\smallskip
(ii) (Compatibility) For every $Q\le E$ and every $y\in C_E(Q)$, setting $P:=\phi(Q)$, $x:=\phi(y)$, $Q':=Q\langle y\rangle$, and $P':=P\langle x\rangle$, the diagram
\begin{diagram}[75]
  \movevertex(-30,0){\KK R(\KK C_H(Q)f_Q)} & \Ear[40]{I_Q} & \movevertex(30,0){\KK R(\KK C_G(P)e_P)} &&
  \movearrow(-30,0){\Sar{d_{C_H(Q)}^{(y,f_{Q'})}}} & & \movearrow(30,0){\saR{d_{C_G(P)}^{(x,e_{P'})}} }&&
  \movevertex(-30,0){\KK R(FC_H(Q')f_{Q'})} & \Ear[35]{I_{Q'}} & \movevertex(30,0){\KK R(F C_G(P')e_{P'})} &&
\end{diagram}
commutes, where $I_Q$ denotes the $\KK$-linear extension of the group homomorphism $I_{\mu_Q}=\mu_Q\cdotH-\colon R(\KK C_H(Q) f_Q) \to R(\KK C_G(P)e_P)$; cf.~\ref{not character groups}(b),(c),(d) for notation and Remark~\ref{rem isometry}(e) for the bottom map in the above diagram.
\end{definition}

\begin{theorem}\label{thm ppeq implies isotypy}
Let $\gamma\in T^\Delta(A,B)$ be a $p$-permutation equivalence between $A$ and $B$ and let $(\Delta(D,\phi,E),e\otimes f^*)\in\calBP_{\calO}(\gamma)$ be a maximal $\gamma$-Brauer pair. For every $P\le D$ (resp.~ $Q\le E$) let $(P,e_P)\in\calBP_\calO(A)$ (resp.~$(Q,f_Q)\in\calBP_\calO(B)$) denote the unique Brauer pair with $(P,e_P)\le (D,e)$ (resp.~$(Q,f_Q)\le (E,f)$).

Then the data $(D,e)$, $(E,f)$, $\phi$, and, for $Q\le E$, the restriction $\mu_Q$ to $C_G(\phi(Q))\times C_H(Q)$ of the element
\begin{equation*} 
  e_{\phi(Q)}\mu(\Delta(\phi(Q),\phi,Q))f_Q\in
  R(\KK N_{G\times H}(\Delta(\phi(Q),\phi,Q), e_{\phi(Q)}\otimes f_Q^*))
\end{equation*}
form an isotypy between $A$ and $B$.
\end{theorem}

\begin{proof}
By Theorem~\ref{thm gamma is uniform}(c), $(D,e)$ is a maximal $A$-Brauer pair and $(E,f)$ is a maximal $B$-Brauer pair. By Theorem~\ref{thm isomorphic fusion systems} the isomorphism $\phi\colon E\myiso D$ is an isomorphism between the fusion system $\calB$ of $B$ associated to $(E,f)$ and the fusion system $\calA$ of $A$ associated to $(D,e)$. Let $Q\le E$ and set $P:=\phi(Q)$. Then $(\Delta(P,\phi,Q), e_P\otimes f_Q^*)\le (\Delta(D,\phi,E), e\otimes f^*)$ (see Remark~\ref{rem emuf}(c)) and, by Theorem~\ref{thm gamma is uniform}(a), $(\Delta(P,\phi,Q),e_P\otimes f_Q^*)$ is a $\gamma$-Brauer pair. Now, Proposition~\ref{prop local characters}(a), implies that 
$\mu_Q$ is a perfect isometry between $\KK C_G(P)e_P$ and $\KK C_H(Q)f_Q$. Therefore the data have the required properties.

\smallskip
To see that the perfect ismometries $\mu_Q$, $Q\le E$, satisfy the equivariance axiom (i), let $Q\le E$, $P:=\phi(Q)$, and $(g,h)\in G\times H$ such that 
\begin{equation}\label{eqn subconjugate}
\lexp{(g,h)}{(\Delta(P,\phi,Q), e_P\otimes f_Q^*)}\le (\Delta(D,\phi,E),e\otimes f^*)\,.
\end{equation} 
In order to prove that $\lexp{(g,h)}{\mu_Q}=\mu_{(\lexp{h}{Q})}$, it suffices to show that
\begin{equation}\label{eqn Brauer pair equation}
  \lexp{(g,h)}{\gammabar(\Delta(P,\phi,Q),e_P\otimes f_Q^*)} = 
  \gammabar(\Delta(\phi(\lexp{h}{Q}), \phi, \lexp{h}{Q}), e_{\phi(\lexp{h}{Q})}\otimes f_{(\lexp{h}{Q})}^*)\,.
\end{equation}
Since the Brauer construction commutes with conjugation, the left hand side is equal to 
$\gammabar(\lexp{(g,h)}{(\Delta(P,\phi,Q),e_P\otimes f_Q^*)})$ and 
$\lexp{(g,h)}{(\Delta(P,\phi,Q),e_P\otimes f_Q^*)}= (\Delta(\lexp{g}{P},c_g\phi c_h^{-1}, \lexp{h}{Q}), \lexp{g}{e_P}\otimes \lexp{h}{(f_Q^*)})$. Moreover, by (\ref{eqn subconjugate}) and Remark~\ref{rem emuf}(c), we obtain 
$\Delta(\lexp{g}{P},c_g\phi c_h^{-1}, \lexp{h}{Q})=\Delta(\phi(\lexp{h}{Q}),\phi, \lexp{h}{Q})$, 
$(\lexp{g}{P},\lexp{g}{e_P})\le (D,e)$, and $(\lexp{h}{Q},\lexp{h}{f_Q})\le (E,f)$. Thus, $\lexp{g}{P}=\phi(\lexp{h}{Q})$, $\lexp{g}{e_P}=e_{(\lexp{g}{P})}=e_{\phi(\lexp{h}{Q})}$, and $\lexp{h}{f_Q}=f_{(\lexp{h}{Q})}$. This establishes Equation~(\ref{eqn  Brauer pair equation}).

\smallskip
Finally, we show that the perfect isometries $\mu_Q$, $Q\le E$, satisfy the compatibility axiom (ii). Let $Q\le E$ and $y\in C_E(Q)$ and set $P:=\phi(Q)$, $x:=\phi(y)$, $Q':=Q\langle y\rangle$, and $P':=P\langle x\rangle = \phi(Q')$.

Let $\gamma_Q\in T^\Delta(\calO C_G(P)e_P,\calO C_H(Q)f_Q)$ denote the restriction of $e_P\gamma(\Delta(P,\phi,Q))f_Q$ to $C_G(P)\times C_H(Q)$ and let $\gamma_{Q'}$ be defined similarly. Recall from \cite[2.1]{BoltjeXu2008} the definition of a linear source $\calO G$-module and its associated representation group $L(\calO G)$. Moreover, recall from \cite[2.3]{BoltjeXu2008} the map $-(\langle x\rangle, x)\colon L(\calO C_G(P)e_P)\to L(\calO C_G(P')\br_{\langle x\rangle}(e_P))$. 
The map $-(\langle x\rangle, x)e_{P'}$ in the diagram below is defined as the composition of the map $-(\langle x\rangle, x)$ with the natural projection from $\KK L(\calO C_G(P')\br_{\langle x\rangle}(e_P))\to \KK L(\calO C_G(P')e_{P'})$, noting that $\br_{\langle x\rangle}(e_P)e_{P'}=e_{P'}$. Similarly, we define the map $-(\langle y \rangle,y)f_{Q'}$. For any finite group $X$ let $\pi_X\colon L(\calO X)\to T(FX)\to R(F X)$ denote the composition of the homomorphism $L(\calO X)\to T(FX)$ induced by the functor $F\otimes_{\calO}-\colon \lmod{\calO X}\to\lmod{FX}$ and the map $\eta_X\colon T(FX)\to R(FX)$ from the diagram in \ref{noth T(G)}(c). Moreover, $\kappa_X\colon L(\calO X)\to R(\KK X)$ will denote the homomorphism induced by the functor $\KK\otimes_\calO-\colon\lmod{\calO X}\to\lmod{\KK X}$. Now consider the following diagram:
\begin{diagram}
  L(\calO C_H(Q)f_Q) & & \Ear[70]{\gamma_Q\dottt{C_H(Q)}-} & & L(\calO C_G(P)e_P) & &  &&
  \Sar{-(\langle y\rangle, y)f_{Q'}} & & & & \Sdotar{-(\langle x\rangle,x)e_{P'}} & &  &&
  \KK L(\calO C_H(Q')f_{Q'}) & \ssear & \Edotar[70]{\gamma_{Q'}\dottt{C_H(Q')}-} & & \KK L(\calO C_G(P')e_{P'}) & \ssear & &&
  & & \kappa_{C_H(Q)} & & & & \kappa_{C_G(P)} &&
  & \ssear & R(\KK C_H(Q)f_Q) & & \Ear[70]{\mu_Q\dottt{C_H(Q)}-} & \ssedotar & R(\KK C_G(P)e_P) &&
  & \pi_{C_H(Q')} & \saR{d_{C_H(Q)}^{(y,f_{Q'})}} & & & \pi_{C_G(P')}& \saR{d_{C_G(P)}^{(x,e_{P'})}} &&
  & &  \KK R(FC_H(Q')f_{Q'}) & & \Ear[70]{\overline{\gamma_{Q'}}\dottt{C_H(Q')}-} & & \KK R(FC_G(P')e_{P'}) &&
\end{diagram}
Note that the diagram in the compatibility axiom equals the front square diagram involving the two generalized decomposition maps in the cube-like diagram. By Brauer's induction theorem, the map $\kappa_{C_H(Q)}\colon L(\calO C_H(Q)f_Q)\to R(\KK C_H(Q)f_Q)$ is surjective. Therefore, it suffices to show that the front square diagram commutes after precomposing with $\kappa_{C_H(Q)}$. It follows that the front square diagram commutes if the left, right and ceiling square diagram in the cube-like diagram commute and the concatenation of the rear and floor square diagrams commute. Clearly, the ceiling square of the cube (involving the $\kappa$-maps) commutes. Moreover, the left and right squares commute by Theorem~2.4 in \cite{BoltjeXu2008}. Thus, it suffices now to show that the concatenation of the rear square and the floor square commutes. 

So let $V$ be an indecomposable linear source $\calO C_H(Q)f_Q$-module. We need to show that
\begin{equation*}
 e_{P'} \bigl(\overline{\gamma_Q\dottt{C_H(Q)}[V]}\bigr)(\langle x\rangle,x) = 
 \overline{\gamma_{Q'}}\dottt{C_H(Q')} (f_{Q'}[\Vbar(\langle y\rangle,y)])\in \KK R(F C_G(P')e_{P'})\,,
\end{equation*}
in the notation of \cite{BoltjeXu2008}.
By the definition of $-(\langle x\rangle,x)$ and $-(\langle y\rangle,y)$ in \cite[2.3]{BoltjeXu2008} it suffices to show that for every $\theta\in\Hom(\langle x\rangle,\calO^\times)$, one has
\begin{equation}\label{eqn gen Brauer}
   e_{P'}(\overline{\gamma_Q\dottt{C_H(Q)}[V]})(\langle x\rangle,\theta) = 
   \overline{\gamma_{Q'}}\dottt{C_H(Q')} [f_{Q'}\Vbar(\langle y\rangle,\theta\circ \phi)]\,.
\end{equation}
By a slight variation of Lemma~3.5(d) in \cite{BoltjeXu2008}, identifying $\langle y\rangle$ and $\langle x\rangle$ via $\phi$, we see that the left hand side of Equation~(\ref{eqn gen Brauer}) is equal to
\begin{equation*}
  \sum_{\substack{\rho,\sigma\in\Hom(\langle x\rangle,\calO^\times)\\ \rho\circ\sigma=\theta}} 
  e_{P'}\overline{\gamma_Q}(\Delta(\langle x\rangle, \phi,\langle y\rangle),\tilde{\rho}) \dottt{FC_H(Q')} 
  \Vbar(\langle y\rangle, \sigma\circ \phi)\in \KK R(F C_G(P')e_{P'})\,,
\end{equation*}
where $\tilde{\rho}$ is defined as $\rho\circ p_1$ on $\Delta(\langle x\rangle, \phi,\langle y\rangle)$. But since $\gamma_Q$ is a virtual $p$-permutation module, one has $\overline{\gamma_Q}(\Delta(\langle x\rangle, \phi,\langle y\rangle),\tilde{\rho})=0$, unless $\rho=1$ is the trivial homomorphism, and in this case one has $\overline{\gamma_Q}(\Delta(\langle x\rangle, \phi,\langle y\rangle),1)=\overline{\gamma_Q}(\Delta(\langle x\rangle, \phi,\langle y\rangle))$, the usual Brauer construction. Thus, the left hand side of Equation~(\ref{eqn gen Brauer}) is equal to 
\begin{equation*}
  e_{P'}\overline{\gamma_{Q}}(\Delta(\langle x\rangle,\phi,\langle y\rangle))\dottt{C_H(Q')} 
  [\Vbar(\langle y\rangle,\theta\circ\phi)]\,.
\end{equation*}
Since $\gamma_Q=e_P\gamma(\Delta(P,\phi,Q))f_Q$, Proposition~\ref{prop more on p-perm}(b) and Lemma~\ref{lem Brauer map compatibility} imply that
\begin{equation}\label{eqn end of isotypy proof}
  e_{P'}\overline{\gamma_Q}(\Delta(\langle x\rangle,\phi,\langle y\rangle)) =
  e_{P'}\br_{\langle x\rangle}(e_P)\gammabar(\Delta(P',\phi,Q'))\br_{\langle y\rangle}(f_Q) =
  e_{P'}\gammabar(\Delta(P',\phi,Q'))\br_{Q'}(f_Q)\,.
\end{equation}
Here we used that $e_{P'}\br_{\langle x\rangle}(e_P)=e_{P'}\br_{P'}(e_P)=e_{P'}$, since $P\trianglelefteq P'$ and $(P,e_P)\le (P', e_{P'})$, see Proposition~\ref{prop Brauer pairs}(b). We claim that the last expression in (\ref{eqn end of isotypy proof}) is equal to $e_{P'}\gammabar(\Delta(P',\phi,Q'))f_{Q'}$. In fact, assume that $f'$ is a primitive idempotent of $Z(\calO C_H(Q'))$ such that $e_{P'}\gammabar(\Delta(P',\phi,Q'))f'\neq 0$. Then Corollary~10.9 implies that there exists $h\in H$ such that $(\phi,(Q',f_{Q'}))=\lexp{h}{(\phi, (Q',f'))} = (\phi\circ c_h^{-1}, (\lexp{h}{Q'},\lexp{h}{f'}))$. But $Q'=\lexp{h}{Q'}$ and $\phi=\phi\circ c_h^{-1}$ imply that $h\in C_H(Q')$, and we obtain $f_{Q'}=\lexp{h}{f'}=f'$. Moreover, since $\br_{Q'}(f_Q)f_{Q'}=f_{Q'}$, the claim is proved. Thus, the left hand side of Equation~(\ref{eqn gen Brauer}) is equal to
\begin{equation*}
  e_{P'}\gammabar(\Delta(P',\phi,Q'))f_{Q'} \dottt{C_H(Q')} [\Vbar(\langle y\rangle,\theta\circ\phi)]\,,
\end{equation*}
as desired, and the proof is complete.
\end{proof}


\bigskip\noindent
{\bf Acknowledgment.} The first author would like to express his gratitude to the {\sc Bernoulli Center} at the EPFL, where some part of this research was achieved during his stay in November and December 2016 for the program \lq Local representation theory and simple groups\rq.


\end{document}